\documentclass[12pt,a4]{amsart}
\usepackage[a4paper, left=26mm, right=26mm, top=28mm, bottom=34mm]{geometry}

\usepackage{amsmath, amscd}
\usepackage{amssymb}
\usepackage[alphabetic]{amsrefs}
\usepackage[T1]{fontenc}
\usepackage[all]{xy}
\usepackage{mathtools}
\usepackage{footmisc}
\usepackage{stmaryrd}
\usepackage{mathrsfs}

\usepackage{tikz}
\usetikzlibrary{calc}
\usetikzlibrary{matrix,arrows,decorations.pathmorphing}
\usepackage{tikz-cd}

\usepackage{relsize} 
\usepackage[bbgreekl]{mathbbol} 
\usepackage{amsfonts} 
\usepackage[colorlinks=true, hyperindex, linkcolor=cyan, citecolor=, pagebackref=false, urlcolor=cyan, linktocpage]{hyperref}

\DeclareSymbolFontAlphabet{\mathbb}{AMSb} 
\DeclareSymbolFontAlphabet{\mathbbl}{bbold} 
\newcommand{\Prism}{{\mathlarger{\mathbbl{\Delta}}}}

\theoremstyle{definition}
\newtheorem{thm}{Theorem}[section]
\newtheorem{mainthm}{Theorem}

\newtheorem{lem}[thm]{Lemma}
\newtheorem{cor}[thm]{Corollary}
\newtheorem{prop}[thm]{Proposition}

\theoremstyle{definition}

\newtheorem{rem}[thm]{Remark}

\newtheorem{dfn}[thm]{Definition}

\newtheorem{construction}[thm]{Construction}
\newtheorem{setting}[thm]{Setting}

\newcommand{\bA}{\mathbb{A}}

\newcommand{\bD}{\mathbb{D}}
\newcommand{\bG}{\mathbb{G}}

\newcommand{\bL}{\mathbb{L}}
\newcommand{\bN}{\mathbb{N}}
\newcommand{\bZ}{\mathbb{Z}}
\newcommand{\bF}{\mathbb{F}}

\newcommand{\bQ}{\mathbb{Q}}
\newcommand{\bT}{\mathbb{T}}

\newcommand{\cA}{\mathcal{A}}

\newcommand{\cE}{\mathcal{E}}

\newcommand{\cG}{\mathcal{G}}

\newcommand{\cI}{\mathcal{I}}
\newcommand{\cJ}{\mathcal{J}}

\newcommand{\cM}{\mathcal{M}}
\newcommand{\cN}{\mathcal{N}}
\newcommand{\cO}{\mathcal{O}}
\newcommand{\cP}{\mathcal{P}}
\newcommand{\sA}{\mathscr{A}}
\newcommand{\sC}{\mathscr{C}}

\newcommand{\sS}{\mathscr{S}}

\newcommand{\fS}{\mathfrak{S}}
\newcommand{\fU}{\mathfrak{U}}

\newcommand{\fW}{\mathfrak{W}}
\newcommand{\fX}{\mathfrak{X}}
\newcommand{\fY}{\mathfrak{Y}}
\newcommand{\fZ}{\mathfrak{Z}}

\newcommand{\isom}{\stackrel{\sim}{\to}}

\newcommand{\Coker}{\mathrm{Coker}}

\newcommand{\et}{\text{\'{e}t}}

\newcommand{\Spec}{\mathrm{Spec}}
\newcommand{\Spf}{\mathrm{Spf}}

\newcommand{\sHom}{\mathcal{H}om}

\title{Log prismatic Dieudonn\'{e} theory and its application to Shimura varieties}

\author{Kentaro Inoue}

\begin{document}

\begin{abstract}
    We study the log version of the prismatic Dieudonn\'{e} theory established by Ansch\"{u}tz-Le Bras. By applying this result to the integral toroidal compactification of a Shimura variety of Hodge type, we extend the prismatic realization, originally constructed by Imai-Kato-Youcis, to the compactification. This extension enables us to prove Lovering's conjecture on $p$-adic comparison isomorphisms for Shimura varieties.
\end{abstract}

\maketitle

\tableofcontents

\section{Introduction}

Shimura varieties form a class of algebraic varieties that play an important role in number theory and algebraic geometry. A guiding principle for Shimura varieties is that a Shimura variety associated with a Shimura datum $(G,X)$ should be the moduli space parameterizing ``motives with $G$-structure''. In particular, there ought to be ``a universal motive with $G$-structure'' (which is just a conjectural object). Several cohomological realizations of ``a universal motive with $G$-structure'' are constructed so far. As a recent development in this direction, a realization in the framework of prismatic cohomology has been constructed by \cites{iky23,iky24}. 

On the other hand, we have a good theory of compactifications of Shimura varieties. For canonical models of Shimura varieties, see \cite{pin90}. For integral models of Shimura varieties with hyperspecial levels, a theory of compactifications in Siegel cases, PEL-type cases, and Hodge type cases is established in \cite{fc90}, \cite{lan13}, and \cite{mad19}, respectively.
In some cases, ``the universal motive with $G$-structure'' on an interior Shimura variety is known to extend to ``the universal degenerating motive with $G$-structure'' on a toroidal compactification (for example, the realization as log Hodge structures in \cite{ku09} or $G$-zip realization in \cite{gk19}). In this paper, we construct a log prismatic realization of this degenerating object within the framework of log prismatic cohomology established by \cites{kos22,ky23} over toroidal compactifications of the integral canonical models of Shimura varieties of Hodge type with a hyperspecial level constructed in \cite{mad19}. This log prismatic realization restricts to the prismatic realization on the interior constructed in \cites{iky23,iky24}.

The log version of prismatic cohomology theory is established in \cites{kos22,ky23}. Log prismatic cohomology theory has coefficient objects called log prismatic $F$-crystals. We let $\mathrm{Vect}^{\varphi}((\fX,\cM_{\fX})_{\Prism})$ denote the category of log prismatic $F$-crystals on a bounded $p$-adic log formal scheme. Log prismatic cohomology theory enables us to treat $p$-adic Hodge theory on a smooth variety over a $p$-adic field $K$ admitting an open immersion to a proper semi-stable scheme over $\cO_{K}$ while (non-log) prismatic cohomology theory behaves well only for a proper smooth variety over $K$ with good reduction.

We return to Shimura varieties. To state precisely, let $(G,X)$ be a Shimura datum of Hodge type with reflex field $E$. Let $p>2$ be a prime number and fix a place $v$ of $E$ above $p$. Let $E_{v}$ be the completion of $E$ at $v$ and $\cO_{E_{v}}$ be a ring of integers of $E_{v}$. Let $K_{p}\subset G(\bQ_{p})$ be an open compact subgroup and $K^{p}\subset G(\bA^{\infty,p})$ be a neat open compact subgroup. Set $K\coloneqq K_{p}K^{p}$. Then the Shimura variety $\mathrm{Sh}_{K}(G,X)$ is an algebraic variety over $E$. Suppose that $K_{p}$ is hyperspecial and let $\cG$ be a reductive model of $G_{\bQ_{p}}$ over $\bZ_{p}$ with $\cG(\bZ_{p})=K_{p}$. In this situation, for a sufficiently small $K^{p}$, an integral canonical model $\sS_{K}(G,X)$ of $\mathrm{Sh}_{K}(G,X)_{E_{v}}$ over $\cO_{E_{v}}$ is constructed in \cite{kis10}, and the toroidal compactification $\sS_{K}^{\Sigma}(G,X)$ of $\sS_{K}(G,X)$ for a suitable (smooth projective) cone decomposition $\Sigma$ is constructed in \cite{mad19}. Let $\mathrm{Sh}_{K}^{\Sigma}(G,X)\coloneqq \sS_{K}^{\Sigma}(G,X)\otimes_{\cO_{E_{v}}} E_{v}$. Let $\widehat{\sS}_{K}(G,X)$ and $\widehat{\sS}_{K}^{\Sigma}(G,X)$ denote the $p$-adic completion of $\sS_{K}(G,X)$ and $\sS_{K}^{\Sigma}(G,X)$ respectively. Let  $(\sS_{K}^{\Sigma}(G,X),\cM_{\sS_{K}^{\Sigma}(G,X)})$ denote the log scheme equipped with the log structure defined by the boundary divisor. Similarly, we define $(\mathrm{Sh}_{K}^{\Sigma}(G,X),\cM_{\mathrm{Sh}_{K}^{\Sigma}(G,X)})$ and $(\widehat{\sS}_{K}^{\Sigma}(G,X),\cM_{\widehat{\sS}_{K}^{\Sigma}(G,X)})$. Note that $(\widehat{\sS}_{K}^{\Sigma}(G,X),\cM_{\widehat{\sS}_{K}^{\Sigma}(G,X)})$ is a horizontally semi-stable log formal scheme over $\cO_{E_{v}}$ in the sense of \cite[Definition 2.4]{ino25}.

The tower
\[
\varprojlim_{K'_{p}\subset K_{p}} \mathrm{Sh}_{K'_{p}K^{p}}(G,X)\to \mathrm{Sh}_{K}(G,X)
\]
gives a $\bZ_{p}$-linear exact tensor functor
\[
\omega_{\et}\colon \mathrm{Rep}_{\bZ_{p}}(\cG)\to \mathrm{Loc}_{\bZ_{p}}(\mathrm{Sh}_{K}(G,X)),
\]
which is regarded as the \'{e}tale realization of ``the universal motive with $G$-structure'' on $\mathrm{Sh}_{K}(G,X)$. Here, $\mathrm{Rep}_{\bZ_{p}}(\cG)$ denotes the category of representations of $\cG$ over $\bZ_{p}$. By purity, $\omega_{\et}$ induces a $\bZ_{p}$-linear exact tensor functor
\[
\omega_{\mathrm{k\et}}\colon \mathrm{Rep}_{\bZ_{p}}(\cG)\to \mathrm{Loc}_{\bZ_{p}}(\mathrm{Sh}_{K}^{\Sigma}(G,X),\cM_{\mathrm{Sh}_{K}^{\Sigma}(G,X)}),
\]
where the target category is the category of Kummer \'{e}tale $\bZ_{p}$-local systems. The object $\omega_{\mathrm{k\et}}$ can be regarded as the Kummer \'{e}tale realization of ``the universal degenerating motive with $G$-structure'' on $\mathrm{Sh}_{K}^{\Sigma}(G,X)$. In \cite{iky23}, the prismatic realization of ``the universal motive with $G$-structure'' is constructed as a $\bZ_{p}$-linear exact tensor functor
\[
\omega_{\Prism}\colon \mathrm{Rep}_{\bZ_{p}}(\cG)\to \mathrm{Vect}^{\varphi}(\widehat{\sS}_{K}(G,X)_{\Prism})
\]
with $T_{\et}\circ \omega_{\Prism}\cong \omega_{\et}$, where $T_{\et}$ is the \'{e}tale realization functor 
\[
T_{\et}\colon \mathrm{Vect}^{\varphi}(\widehat{\sS}_{K}(G,X)_{\Prism})\to \mathrm{Loc}_{\bZ_{p}}(\widehat{\sS}_{K}(G,X)_{\eta}).
\]
Then we can state our main theorem as follows.

\begin{mainthm}[see Theorem \ref{log pris realization}]\label{mainthm log pris real}
    There exists a unique $\bZ_{p}$-linear exact tensor functor
    \[
    \omega_{\Prism,\mathrm{log}}\colon \mathrm{Rep}_{\bZ_{p}}(\cG)\to \mathrm{Vect}^{\varphi}((\widehat{\sS}^{\Sigma}_{K}(G,X),\cM_{\widehat{\sS}^{\Sigma}_{K}(G,X)})_{\Prism})
    \]
    with $T_{\et}\circ \omega_{\Prism,\mathrm{log}}\cong \omega_{\mathrm{k\et}}$ and $\omega_{\Prism,\mathrm{log}}|_{\widehat{\sS}_{K}(G,X)}\cong \omega_{\Prism}$, where $T_{\et}$ is the \'{e}tale realization functor
    \[
    T_{\et}\colon \mathrm{Vect}^{\varphi}((\widehat{\sS}^{\Sigma}_{K}(G,X),\cM_{\widehat{\sS}^{\Sigma}_{K}(G,X)})_{\Prism})\to \mathrm{Loc}_{\bZ_{p}}((\widehat{\sS}^{\Sigma}_{K}(G,X),\cM_{\widehat{\sS}^{\Sigma}_{K}(G,X)})_{\eta}).
    \]
\end{mainthm}

As an application of our construction, we give the $p$-adic comparison isomorphism on (not necessarily compact) Shimura varieties, which is conjectured in \cite[3.6.2]{lov17}.

\begin{mainthm}(see Theorem \ref{comparison isom on shimura var})\label{mainthm for comp isom}
    Let $(\overline{\sS}_{K}^{\Sigma}(G,X),\cM_{\overline{\sS}_{K}^{\Sigma}(G,X)})$ be the special fiber of $(\sS_{K}^{\Sigma}(G,X),\cM_{\sS_{K}^{\Sigma}(G,X)})$. Consider a $\bZ_{p}$-linear exact tensor functor
    \[
    \omega_{\mathrm{fisoc}}\colon \mathrm{Rep}_{\bZ_{p}}(\cG)\to \mathrm{FilIsoc}^{\varphi}(\widehat{\sS}^{\Sigma}_{K}(G,X),\cM_{\widehat{\sS}^{\Sigma}_{K}(G,X)})
    \]
    defined by $\omega_{\mathrm{fisoc}}\coloneqq T_{\mathrm{fisoc}}\circ \omega_{\Prism,\mathrm{log}}$, where $T_{\mathrm{fisoc}}$ is the realization functor as filtered $F$-isocrystals (see \cite[Theorem 6.18]{ino25}). Then the $\mathrm{Gal}(\overline{E_{v}}/E_{v})$-representation
    \[
    H^{i}_{\et}(\mathrm{Sh}_{K}(G,X)_{\overline{E}_{v}},\omega_{\et}(\xi)[1/p])
    \]
    is crystalline, and we have a comparison isomorphism
    \begin{align*}
    &H^{i}_{\et}(\mathrm{Sh}_{K}(G,X)_{\overline{E}_{v}},\omega_{\et}(\xi)[1/p])\otimes_{\bQ_{p}} B_{\mathrm{crys}} \\
    &\cong H^{i}_{\mathrm{logcrys}}((\overline{\sS}_{K}^{\Sigma}(G,X),\cM_{\overline{\sS}_{K}^{\Sigma}(G,X)}),\omega_{\mathrm{fisoc}}(\xi))\otimes_{W} B_{\mathrm{crys}}    
    \end{align*}
    that is compatible with Galois actions and Frobenius isomorphisms and a filtered isomorphism after taking the base change along $B_{\mathrm{crys}}\to B_{\mathrm{dR}}$ for any $\xi\in \mathrm{Rep}_{\bZ_{p}}(\cG)$.
\end{mainthm}


Let us explain our construction. One of the key steps is the construction of a log prismatic $F$-crystal on $\widehat{\sS}_{K}^{\Sigma}(G,X)$ which is the extension of the prismatic $F$-crystal $R^{1}f_{*}\cO_{\Prism}$ on $\sS_{K}(G,X)$, where $f\colon \sA_{\mathrm{univ}}\to \sS_{K}(G,X)$ is the universal abelian scheme (i.e. the pullback of the universal abelian scheme on the Siegel modular variety along the Hodge embedding).

For simplicity, we assume that $(G,X)=(\mathrm{GSp},S^{\pm})$. The universal abelian variety $\sA_{\mathrm{univ}}$ on $\sS_{K}(G,X)$ degenerates to a semi-abelian scheme $\sA_{\mathrm{univ}}^{\mathrm{sab}}$ on $\sS_{K}^{\Sigma}(G,X)$. However, the relative prismatic cohomology of $\sA_{\mathrm{univ}}^{\mathrm{sab}}\to \sS_{K}^{\Sigma}(G,X)$ does not work well because this morphism is not proper. As a substitute, we have two candidates: the log abelian scheme on $\sS_{K}^{\Sigma}(G,X)$ extending $\sA_{\mathrm{univ}}$ and a compactification of $\sA_{\mathrm{univ}}^{\mathrm{sab}}$ constructed in \cite[Theorem 2.15]{lan12}. In this paper, we adopt the former one.

The notion of log abelian schemes is introduced by Kajiwara-Kato-Nakayama, and fundamental theory of them is developed in a series of papers \cites{kkn08,kkn15,kkn18,kkn19,kkn21,kkn22}. As an application of their works, $\sS_{K}^{\Sigma}(G,X)$ is reinterpreted as the moduli space of principally polarized log abelian schemes with level $K$-structure and local monodromies in $\Sigma$ (see \cite{kkn21} or \cite[Theorem 2.2.2, Proposition 4.3.4]{kkn22}). In particular, $\sA_{\mathrm{univ}}$ on $\sS_{K}(G,X)$ extends to the universal log abelian scheme $\sA^{\mathrm{log}}_{\mathrm{univ}}$ on $\sS_{K}^{\Sigma}(G,X)$. Then ``the relative log prismatic cohomology of $\sA^{\mathrm{log}}_{\mathrm{univ}}$ over $\widehat{\sS}_{K}^{\Sigma}(G,X)$'' should be the desired object. However, it seems technically difficult to consider ``the log prismatic cohomology of log abelian schemes'' because log abelian varieties are just sheaves (not even log algebraic spaces). To avoid this difficulty, we pass to the associated log $p$-divisible groups which are easier to treat.

The notion of log $p$-divisible groups is introduced by Kato in \cite{kat23}, and we can associate with log abelian schemes log $p$-divisible groups by taking $p$-power torsions as in non-log cases. The prismatic Dieudonn\'{e} theory established in \cite{alb23} is generalized to log $p$-divisible groups.

\begin{mainthm}[see Theorem
\ref{log pris dieudonne equiv}]\label{mainthm log pris dieudonne}
    Let $(\fX,\cM_{\fX})$ be a semi-stable log formal scheme over $\cO_{K}$ in the sense of \cite[Definition 2.4]{ino25}. Then there exists an exact anti-equivalence 
    \[
    \cM_{\Prism}\colon \mathrm{BT}(\fX,\cM_{\fX})\to \mathrm{DM}((\fX,\cM_{\fX})_{\Prism})
    \]
    (which is called a \emph{log prismatic Dieudonn\'{e} functor}), where $\mathrm{BT}(\fX,\cM_{\fX})$ is the category of log $p$-divisible groups over $(\fX,\cM_{\fX})$ and $\mathrm{DM}((\fX,\cM_{\fX})_{\Prism})$ is the category of log prismatic Dieudonn\'{e} crystals on $(\fX,\cM_{\fX})$ (for the definition, see Definition \ref{def of log pris crys}).
\end{mainthm}

In \cite{alb23}, it is proved that the relative prismatic cohomology of an abelian scheme is isomorphic to the prismatic Dieudonn\'{e} crystal of the associated $p$-divisible group. As considered by analogy, the log prismatic Dieudonn\'{e} crystal of $\sA_{\mathrm{univ}}^{\mathrm{log}}[p^{\infty}]$ is the desired object.

The idea of the proof of Theorem \ref{mainthm log pris dieudonne} is to reduce the problem to non-log cases (\cite{alb23}) by using ``Kummer quasi-syntomic descent''. Since the category of prismatic crystals does not satisfy this type of descent properties, we define a larger category whose objects are called \emph{kfl prismatic crystals} (Definition \ref{def of kfl pris crys}) and prove ``Kummer quasi-syntomic descent'' for kfl prismatic crystals (Proposition \ref{kqsyn descent for kfl log pris crys}). In a similar way, while the category of log $p$-divisible groups does not satisfy such descent property, we can define a larger category whose objects are called \emph{weak log $p$-divisible groups} (which is introduced in \cite{kat23}) and prove ``pro-Kummer log flat descent'' for weak log $p$-divisible groups (Proposition \ref{pro-kfl descent for log BT}). These descent properties and results in non-log cases (\cite{alb23}) gives an equivalence between larger categories (Proposition \ref{log pris dieudonne weak ver}). To prove that this  restricts to the desired equivalence, we use crystalline realization functors. In this step, we need the assumption that $(\fX,\cM_{\fX})$ is semi-stable log formal scheme over $\cO_{K}$.

\begin{rem}\label{gap in wz}
    Theorem \ref{mainthm log pris dieudonne} is proved in \cite{wz23} when $(\fX,\cM_{\fX})$ is $\mathrm{Spf}(\cO_{K})$ with the standard log structure, and they use a similar approach. They define the category of kfl prismatic crystals as the category of vector bundles on the log prismatic site equipped with Kummer log flat topology (\cite[Definition 2.1]{wz23}). However, it is doubtful that the log prismatic site equipped with Kummer log flat topology is indeed a site because coverings in this site may not be stable under base change due to the fact that the underlying ring map of a Kummer log flat map is not necessarily flat.
\end{rem}

\subsection*{Acknowledgements}
The author is grateful to his advisor, Tetsushi Ito, for useful discussions and warm encouragement. A special thanks goes to Teruhisa Koshikawa for innumerable discussions and advice. Moreover, the author would like to thank Kai-Wen Lan, Shengkai Mao, Peihang Wu, and Alex Youcis for helpful comments. This work was supported by JSPS KAKENHI Grant Number 23KJ1325.

\subsection*{Notation and conventions}\noindent

\begin{itemize}
\item The symbol $p$ always denotes a prime.
\item All rings and monoids are commutative.
\item Limits of categories mean $2$-limits. 
\item Formal schemes are assumed to admit a finitely generated ideal of definition Zariski locally. For a property $P$ of morphisms of schemes, an adic morphism of formal schemes is called $\cP$ if it is adically $P$.
\item For a monoid $P$ and an integer $n\geq 1$, let $P^{1/n}$ denote the monoid $P$ with $P\to P^{1/n}$ mapping $p$ to $p^{n}$. The colimit of $P^{1/n}$ with respect to $n\geq 1$ is denoted by $P_{\bQ_{\geq 0}}$.
\item The $i$-th standard basis of $\bN^{r}$ is denote by $e_{i}$ for $1\leq i\leq r$.
\item For $I$-adically complete ring $A$ and a monoid $M$, a ring $A\langle M\rangle$ denotes the $I$-adic completion of the monoid algebra $A[M]$.
\item For a prelog ring $(A,M)$, the associated log scheme is denoted by $(\mathrm{Spec}(A),M)^{a}$. 
\item Let $\fX$ be a formal scheme and $n\geq 1$ be an integer. Assume that $p$ is topologically nilpotent on $\fX$ when $n=1$. The category of finite and locally free group schemes (resp. truncated Barsotti-Tate groups of level $n$) (resp. $p$-divisible groups) on $\fX$ is denoted by $\mathrm{Fin}(\fX)$ (resp. $\mathrm{BT}_{n}(\fX)$) (resp. $\mathrm{BT}(\fX)$). 
\end{itemize}

We refer readers to \cite{ino25} for our notation and terminology concerning the followings:
\begin{itemize}
    \item semi-stable log formal schemes, small affine log formal schemes, and framings (see Section 2 in \emph{loc. cit.});
    \item (strict) absolute log crystalline sites (see Definition 3.1 \emph{loc. cit.});
    \item (strict) absolute log prismatic sites (see Definition 5.3 in \emph{loc. cit.});
    \item Breuil-Kisin log prisms $(\fS_{R},(E),\cM_{\fS_{R}})$ (see Construction 5.10 in \emph{loc. cit.});
    \item Breuil rings $(S_{R},\cM_{S_{R}})$ (see Construction 6.7).
\end{itemize}

\section{Kfl topology on fs log formal schemes}

\subsection{Kummer morphisms of log formal schemes}

We recall fundamental properties of Kummer morphisms. A monoid map $f\colon P\to Q$ of saturated monoids is called \emph{Kummer} if
the following conditions are satisfied:
\begin{enumerate}
    \item $f$ is injective;
    \item for every $q\in Q$, there exist $p\in P$ and an integer $n\geq 1$ such that $f(p)=q^{n}$.
\end{enumerate}

\begin{dfn}[Kummer morphisms of log formal schemes]
Let $f\colon (\fX,\cM_{\fX})\to (\fY,\cM_{\fY})$ be an adic morphism of saturated log formal schemes. The morphism $f$ is called \emph{Kummer} if, for every $x\in \fX$, the monoid map $\overline{\cM_{\fY,\overline{y}}}\to \overline{\cM_{\fX,\overline{x}}}$ is Kummer. Here, we set $y\coloneqq f(x)$.
\end{dfn}

\begin{lem}[{\cite[Proposition 2.7 (2)]{kat21}}] \label{kummer map is strict after n-power ext}
    Let $f\colon (\fX,\cM_{\fX})\to (\fY,\cM_{\fY})$ be a Kummer morphism of fs log formal schemes. Suppose that $\fX$ is quasi-compact and that we are given an fs chart $P\to \cM_{\fY}$. Then there exists an fs monoid $Q$ and a Kummer map $P\to Q$ such that, if we put $(\fY',\cM_{\fY'})\coloneqq(\fY,\cM_{\fY})\times_{(\bZ[P],P)^{a}} (\bZ[Q],Q)^{a}$, the natural morphism 
    \[
    (\fX,\cM_{\fX})\times_{(\fY,\cM_{\fY})} (\fY',\cM_{\fY'})\to (\fY',\cM_{\fY'})
    \]
    is strict. Moreover, we can take as $P\to Q$ the map $P\to P^{1/n}$ for an integer $n\geq 1$ which kills $\Coker((f^{*}\cM_{\fY})^{\mathrm{gp}}\to \cM_{\fX}^{\mathrm{gp}})$.
\end{lem}

\begin{proof}
    Since we can reduce the problem to scheme cases, this follows from \cite[Proposition 2.7 (2)]{kat21}.
\end{proof}

The following lemma is a generalization of \cite[Lemma 2.4]{kat21} to non-fs cases.

\begin{lem}[cf. {\cite[Lemma 2.4 (2)]{kat21}}]\label{four point lemma in nonfs case}
    Consider the following Cartesian diagram in the category of saturated log formal schemes:
    \[
    \begin{tikzcd}
        (\fX',\cM_{\fX'}) \ar[r,"g'"] \ar[d,"f'"] & (\fX,\cM_{\fX}) \ar[d,"f"] \\
        (\fY',\cM_{\fY'}) \ar[r,"g"] & (\fY,\cM_{\fY}).
    \end{tikzcd}
    \]
    Suppose that $(\fY,\cM_{\fY})$ is fs and that $f$ is Kummer. Then, for $x\in \fX$ and $y'\in \fY'$ such that $f(x)=g(y')$, there exists $x'\in \fX'$ satisfying $g'(x')=x$ and $f'(x')=y'$.
\end{lem}

\begin{proof}
    We may assume that $\fX=\fY=\fY'=\mathrm{Spec}(k)$ for an algebraically closed field $k$. It suffices to show that $\fX'$ is nonempty. We can write $\cM_{\fX}=k^{\times}\times Q$, $\cM_{\fY}=k^{\times}\times P$, and $\cM_{\fY'}=k^{\times}\times P'$, where $P,P',Q$ are saturated sharp monoids. The monoid map $\cM_{\fY}\to \cM_{\fX}$ induced by $f$ has the form of $(a,p)\mapsto (a\alpha(p),\beta(p))$ for a monoid map $\alpha\colon P\to k^{\times}$ and a Kummer monoid map $\beta\colon P\to Q$. Writing $Q$ as the union of fs submonoids of $Q$ containing $P$, we can write $(\fX,\cM_{\fX})$ as an inverse limit of fs log schemes which are Kummer over $(\fX,\cM_{\fX})$. In a similar way, we can write $(\fY',\cM_{\fY'})$ as an inverse limit of fs log schemes over $(\fY,\cM_{\fY})$. Then the assertion follows from the result in fs cases (\cite[Lemma 2.4 (2)]{kat21}) and the limit argument. 
\end{proof}

\subsection{Kfl sites of fs log formal schemes}

\begin{dfn}[Log flat morphisms of log schemes, {\cite[1.10]{kat21}}]
Let $f\colon (X,\cM_{X})\to (Y,\cM_{Y})$ be a morphism of fs log schemes.
    \begin{enumerate}
        \item $f$ is called \emph{log flat} if, fppf locally on $X$ and $Y$, there exist an injective chart $P\to Q$ of $f$ called a \emph{flat chart}) such that the induced morphism $(X,\cM_{X})\to (Y,\cM_{Y})\times_{(\bZ[P],P)^{a}} (\bZ[Q],Q)^{a}$ is strict flat.
        \item $f$ is called \emph{Kummer log flat} (or \emph{kfl} for short) if $f$ is Kummer and log flat.
    \end{enumerate}
\end{dfn}

\begin{lem}\label{def of log flat mor}
    Let $f:(\fX,\cM_{\fX})\to (\fY,\cM_{\fY})$ be an adic morphism of fs log formal schemes. Then the following conditions are equivalent.
    \begin{enumerate}
        \item There exists an affine open covering $\{\fU_{i}\}_{i\in I}$ of $\fY$ and finitely generated ideals of definition $\cI_{i}$ on $\fU_{i}$ for each $i\in I$ such that the morphism of schemes
        \[
        (\fX,\cM_{\fX})\times_{(\fY,\cM_{\fY})} (U_{i,n},\cM_{U_{i,n}})\to (U_{i,n},\cM_{U_{i,n}})
        \]
        is log flat, where $(U_{i,n},\cM_{U_{i,n}})$ is the strict closed subscheme of $(\fU_{i},\cM_{\fU_{i}})$ defined by $\cI_{i}^{n}$.
        \item For any open subset $\fU$ of $\fY$ and any finitely generated ideal of definition $\cJ$, the morphism of log schemes
        \[
        (\fX,\cM_{\fX})\times_{(\fY,\cM_{\fY})} (U',\cM_{U'})\to (U',\cM_{U'})
        \]
        is log flat. Here, $(U',\cM_{U'})$ is the strict closed subscheme of $(\fU,\cM_{\fU})$ defined by $\cJ$.
    \end{enumerate}
\end{lem}

\begin{proof}
    Note that log flatness is preserved under base change and local on the source and the target. The statement follows from the fact that, for an ideal of definition $\cI$ on an affine formal scheme, the set $\{\cI^{n}\}$ is cofinal among ideals of definition.
\end{proof}

\begin{dfn}[Kummer log flat morphisms of log formal schemes]
Let $f\colon (\fX,\cM_{\fX})\to (\fY,\cM_{\fY})$ be an adic morphism of fs log formal schemes.
\begin{enumerate}
\item The morphism $f$ is called \emph{log flat} if $f$ satisfies an equivalent condition in Lemma \ref{def of log flat mor}.
\item The morphism $f$ is called \emph{Kummer log flat} (or \emph{kfl} for short) if $f$ is Kummer and log flat.
\end{enumerate}
\end{dfn}

\begin{lem}[{\cites{ogu18,kat21}}]\label{fundamental properties of kfl mor}
    The following statements are true.
    \begin{enumerate}
        \item A strict and log flat morphism of fs log formal schemes is strict flat.
        \item The composition of log flat morphisms of fs log formal schemes is also log flat.
        \item Consider the following Cartesian diagram in the category of fs log formal schemes and adic morphisms.
        \[
        \begin{tikzcd}
            (\fX',\cM_{\fX'}) \ar[r] \ar[d,"f'"] & (\fX,\cM_{\fX}) \ar[d,"f"] \\
            (\fY',\cM_{\fY'}) \ar[r] & (\fY,\cM_{\fY})
        \end{tikzcd}
        \]
        Suppose $f$ is log flat. Then $f'$ is also log flat.
        \item  For a Kummer log flat morphism of fs log formal schemes $(\fX,\cM_{\fX})\to (\fY,\cM_{\fY})$ that is locally of finite presentation, the morphism of underlying formal schemes $\fX\to \fY$ is open.
    \end{enumerate}
\end{lem}

\begin{proof}
    For every assertion, we may assume that all formal schemes are schemes. For the proof in this case, see \cite[Chapter IV,Proposition 4.1.2]{ogu18} and \cite[Lemma 2.4, Proposition 2.5]{kat21}.
\end{proof}

\begin{dfn}
    A family of morphisms $\{f_{i}\colon (\fU_{i},\cM_{\fU_{i}})\to (\fX,\cM_{\fX})\}_{i\in I}$ is called a \emph{Kummer log flat covering} (or a \emph{kfl covering} for short) if the following conditions are satisfied.
    \begin{enumerate}
        \item Each $f_{i}$ is kfl and locally of finite presentation.
        \item The family is set-theoretically surjective i.e. $\displaystyle \fX=\bigcup_{i\in I} f_{i}(\fU_{i})$.
    \end{enumerate}
    
\end{dfn}

\begin{dfn}[Kfl sites, (cf.~{\cite[Definition 2.3]{kat21}})]
    For an fs log formal scheme $(\fX,\cM_{\fX})$, we let $(\fX,\cM_{\fX})_{\mathrm{kfl}}$ denote the category of fs log schemes that are adic over $(\fX,\cM_{\fX})$ equipped with the Grothendieck topology given by kfl coverings. This topology is called \emph{kfl topology}, and the resulting site $(\fX,\cM_{\fX})_{\mathrm{kfl}}$ is called a \emph{kfl site}.
\end{dfn}

\begin{prop}[{\cite[Theorem 3.1]{kat21}}]\label{representable presheaf is sheaf}
    Let $(\fY,\cM_{\fY})\to (\fX,\cM_{\fX})$ be an adic morphism of fs log formal schemes. Then the functor 
    \[
    (Z,\cM_{Z})\mapsto \mathrm{Mor}_{(\fX,\cM_{\fX})}((Z,\cM_{Z}),(\fY,\cM_{\fY}))
    \]
    is a sheaf on $(\fX,\cM_{\fX})_{\mathrm{kfl}}$.
\end{prop}

\begin{proof}
    See \cite[Theorem 3.1]{kat21} for the case that $\fX$ is a scheme. In general, we are reduced to this case by taking the reduction by ideals of definition.
\end{proof}

\begin{rem}\label{classical representable sheaf}
    Let $(\fX,\cM_{\fX})$ be an fs log formal scheme. For a formal scheme $\fY$ over $\fX$, we regard $\fY$ as a log formal scheme over $(\fX,\cM_{\fX})$ by equipping $\fY$ with the pullback log structure of $\cM_{\fX}$ by the structure morphism $\fY\to \fX$. As a result, we have fully faithful functors
    \[
    \mathrm{FSch}_{/\fX}\hookrightarrow \mathrm{LFSch}_{/(\fX,\cM_{\fX})}\hookrightarrow \mathrm{Shv}((\fX,\cM_{\fX})_{\mathrm{kfl}}).
    \]
    Here, $\mathrm{FSch}_{/\fX}$ denotes the category of formal schemes which are adic over $\fX$, and $\mathrm{LFSch}_{/(\fX,\cM_{\fX})}$ denotes the category of log formal schemes which are adic over $(\fX,\cM_{\fX})$.
\end{rem}

\begin{dfn}
    We define a presheaf $\cO_{(\fX,\cM_{\fX})}$ on $(\fX,\cM_{\fX})_{\mathrm{kfl}}$ by 
    \[ \Gamma((Y,\cM_{Y}),\cO_{(\fX,\cM_{\fX})})\coloneqq\Gamma(Y,\cO_{Y})
    \]
    for $(Y,\cM_{Y})\in (\fX,\cM_{\fX})_{\mathrm{kfl}}$. By Proposition \ref{classical representable sheaf}, this is a sheaf on $(\fX,\cM_{\fX})_{\mathrm{kfl}}$.
\end{dfn}

\subsection{Kfl vector bundles on fs log formal schemes}

In this subsection, we review some properties of vector bundles in kfl topology and generalize them to formal scheme cases. Although our primary interest is in vector bundles on kfl sites of log formal schemes, we consider all quasi-coherent sheaves on kfl sites of log formal schemes for completeness. Note that we do not consider coherent sheaves on kfl sites of log formal schemes.

For a formal scheme $\fX$, let $\mathrm{QCoh}(\fX)$ (resp. $\mathrm{Vect}(\fX)$)(resp. $\mathrm{FQCoh}(\fX)$) denote the category of quasi-coherent sheaves (resp. vector bundles)(resp. quasi-coherent sheaves of finite type) on $\fX$. For a scheme $X$, let $\mathrm{Coh}(X)$ denote the category of coherent sheaves on $X$.

\begin{dfn}
    Let $(\fX,\cM_{\fX})$ be an fs log formal scheme. Quasi-coherent sheaves (resp. vector bundles)(resp. quasi-coherent sheaves of finite type) on the ringed site $((\fX,\cM_{\fX})_{\mathrm{kfl}},\cO_{(\fX,\cM_{\fX})})$ are called \emph{kfl quasi-coherent sheaves} (resp. \emph{kfl vector bundles})(resp. \emph{kfl quasi-coherent sheaves of finite type}), and the category consisting of them is denoted by $\mathrm{QCoh}(\fX,\cM_{\fX})$ (resp. $\mathrm{Vect}_{\mathrm{kfl}}(\fX,\cM_{\fX})$)(resp. $\mathrm{FQCoh}(\fX,\cM_{\fX})$). These categories are equipped with the exact structure derived from the exactness as sheaves on the kfl site. 
    
    Similarly, when $\fX$ is a scheme, coherent sheaves on the ringed site $((\fX,\cM_{\fX})_{\mathrm{kfl}},\cO_{(\fX,\cM_{\fX})})$ are called \emph{kfl coherent sheaves} and the exact category of them is denoted by $\mathrm{Coh}_{\mathrm{kfl}}(\fX,\cM_{\fX})$. 
\end{dfn}

\begin{dfn}\label{def of iota}
    Let $(\fX,\cM_{\fX})$ be an fs log formal scheme. For a quasi-coherent sheaf $\cE$ on $\fX$, we define a presheaf $\iota(\cE)$ on $(\fX,\cM_{\fX})_{\mathrm{kfl}}$ by 
    \[ \Gamma((Y,\cM_{Y}),\iota(\cE))\coloneqq\Gamma(Y,f^{*}\cE)
    \]
    for $(Y,\cM_{Y})\in (\fX,\cM_{\fX})_{\mathrm{kfl}}$, where $f\colon (Y,\cM_{Y})\to (\fX,\cM_{\fX})$ is the structure morphism. It follows from the same argument as \cite[3.4]{kat21} that $\iota(\cE)$ is a sheaf. Furthermore, $\iota(\cE)$ is a kfl quasi-coherent sheaves on $(\fX,\cM_{\fX})$. If $\cE$ is a vector bundle (resp. a quasi-coherent sheaf of finite type), $\iota(\cE)$ is a kfl vector bundle (resp. a kfl quasi-coherent sheaf of finite type) on $(\fX,\cM_{\fX})$. When $\fX$ is a scheme and $\cE$ is a coherent sheaf, $\iota(\cE)$ is a kfl coherent sheaf.
\end{dfn}

\begin{lem}\label{fully faithfulness for kfl vect bdle}
    Let $(\fX,\cM_{\fX})$ be an fs log formal scheme. Then the functor
    \[
    \iota\colon \mathrm{QCoh}(\fX)\to \mathrm{QCoh}_{\mathrm{kfl}}(\fX,\cM_{\fX})
    \]
    is fully faithful, and the induced functor
    \[
    \iota\colon \mathrm{Vect}(\fX)\to \mathrm{Vect}_{\mathrm{kfl}}(\fX,\cM_{\fX})
    \]
    gives a bi-exact equivalence to the essential image. 
\end{lem}

\begin{proof}
    Fully faithfulness is clear from the definition. It suffices to prove that a sequence of (usual) vector bundles on $\fX$
    \[
    0\to \cE_{1}\to \cE_{2}\to \cE_{3}\to 0 
    \]
    is exact if the associated sequence
    \[
    0\to \iota(\cE_{1})\to \iota(\cE_{2})\to \iota(\cE_{3})\to 0
    \]
    is exact is $\mathrm{Vect}_{\mathrm{kfl}}(\fX,\cM_{\fX})$. By working Zariski locally on $\fX$, we may assume that $\fX$ admits a finitely generated ideal of definition. Nakayama's lemma allows us to assume that $X\coloneqq \fX$ is a scheme. Since $\iota(\cE_{3})$ is a vector bundle and $\iota(\cE_{2})\to \iota(\cE_{3})$ is surjective, there exists a kfl cover $\pi\colon (Y,\cM_{Y})\to (X,\cM_{X})$ such that $\pi^{*}\iota(\cE_{2})\to \pi^{*}\iota(\cE_{3})$ admits a section, which induces a section of $\pi^{*}\cE_{2}\to \pi^{*}\cE_{3}$ by the fully faithfulness of $\iota$. In particular, $\pi^{*}\cE_{2}\to \pi^{*}\cE_{3}$ is surjective. Since $\pi$ is surjective, $\cE_{2}\to \cE_{3}$ is also surjective by Nakayama's lemma. Consider a vector bundle $\cE'_{1}\coloneqq\mathrm{Ker}(\cE_{2}\to \cE_{3})$. Since $\iota$ is exact, we have $\iota(\cE_{1})\isom \iota(\cE'_{1})$, which induces $\cE'_{1}\isom \cE_{1}$. This proves the claim.
\end{proof}

We regard $\mathrm{QCoh}(\fX)$ as a full subcategory of $\mathrm{QCoh}_{\mathrm{kfl}}(\fX,\cM_{\fX})$ via $\iota$, and we simply write $\cE$ for $\iota(\cE)$ when no possibility of confusion occurs.

\begin{dfn}
    A kfl quasi-coherent sheaf $\cE$ on $(\fX,\cM_{\fX})$ is said \emph{classical} if $\cE$ belongs to $\mathrm{QCoh}(\fX)$. 
\end{dfn}

For an adic morphism of fs log formal schemes $f\colon (\fX,\cM_{\fX})\to (\fY,\cM_{\fY})$ and $\cE\in \mathrm{QCoh}_{\mathrm{kfl}}(\fY,\cM_{\fY})$, we can naturally define the pullback $f^{*}\cE\in \mathrm{QCoh}_{\mathrm{kfl}}(\fX,\cM_{\fX})$. If $\cE$ is a kfl vector bundle (resp. quasi-coherent sheaf of finite type), $f^{*}\cE$ is so. When $\fX$ is a scheme and $\cE$ is a kfl coherent sheaf, $f^{*}\cE$ is also a kfl coherent sheaf. Additionally, we have the following commutative diagram:
\[
\begin{tikzcd}
    \mathrm{QCoh}(\fY) \ar[r,"f^{*}"] \ar[d,"\iota"] & \mathrm{QCoh}(\fX) \ar[d,"\iota"] \\
    \mathrm{QCoh}_{\mathrm{kfl}}(\fY,\cM_{\fY}) \ar[r,"f^{*}"] & \mathrm{QCoh}_{\mathrm{kfl}}(\fX,\cM_{\fX}).
\end{tikzcd}
\]

\begin{lem}[Kato]\label{classicality is pointwise}
    Let $(X,\cM_{X})$ be an fs log scheme and $\cE$ be a kfl vector bundle on $(X,\cM_{X})$. Suppose that, for every $x\in X$, the pullback of $\cE$ to $(\mathrm{Spec}(k(x)),\cM_{k(x)})$ is classical, where $\cM_{k(x)}$ is the pullback log structure of $\cM_{X}$. Then $\cE$ is classical.
\end{lem}

\begin{proof}
    Using a limit argument (see \cite[Appendix]{ino23}), we may assume that $X$ is the spectrum of a noetherian strict local ring. Then the statement follows from \cite[Theorem 6.2]{kat21}.
\end{proof}

\begin{lem}\label{kfl vect bdle as lim}
    Let $(\fX,\cM_{\fX})$ be a log formal scheme admitting a finitely generated ideal of definition $\cI$. Let $(X_{n},\cM_{X_{n}})$ be the strict closed subscheme defined by $\cI^{n+1}$ for $n\geq 0$.
    Then natural functors
    \begin{align*}
        \mathrm{QCoh}_{\mathrm{kfl}}(\fX,\cM_{\fX})&\to  \varprojlim_{n\geq 0} \mathrm{QCoh}_{\mathrm{kfl}}(X_{n},\cM_{X_{n}}) \\
        \mathrm{Vect}_{\mathrm{kfl}}(\fX,\cM_{\fX})&\to  \varprojlim_{n\geq 0} \mathrm{Vect}_{\mathrm{kfl}}(X_{n},\cM_{X_{n}}) \\
        \mathrm{FQCoh}_{\mathrm{kfl}}(\fX,\cM_{\fX})&\to  \varprojlim_{n\geq 0} \mathrm{FQCoh}_{\mathrm{kfl}}(X_{n},\cM_{X_{n}})
    \end{align*}
    gives a bi-exact equivalence.
\end{lem}

\begin{proof}
    This follows from the fact that $(\fX,\cM_{\fX})_{\mathrm{kfl}}$ has a basis consisting of objects in $\bigcup_{n\geq 0} (X_{n},\cM_{X_{n}})_{\mathrm{kfl}}$.
\end{proof}

\begin{cor}\label{classicality of vect bdle can be checked after taking reduction}
     Let $(\fX,\cM_{\fX})$ be a log formal scheme admitting a finitely generated ideal of definition $\cI$. Let $(X_{0},\cM_{X_{0}})$ be the strict closed subscheme defined by $\cI$. Then for $\cE\in \mathrm{Vect}_{\mathrm{kfl}}(\fX,\cM_{\fX})$, the kfl vector bundle $\cE$ is classical if and only if the pullback of $\cE$ to $(X_{0},\cM_{X_{0}})$ is classical.
\end{cor}

\begin{proof}
    Suppose that $\cE_{0}$ is classical. Let $\cE_{n}$ denote the pullback of $\cE$ to $(X_{n},\cM_{X_{n}})$ for $n\geq 0$. The kfl vector bundle $\cE_{n}$ is also classical for every $n\geq 0$ by Lemma \ref{classicality is pointwise}, and the system $\{\cE_{n}\}_{n\geq 0}$ corresponds to a (classical) vector bundle $\cE'$ on $\fX$, which is isomorphic to $\cE$ by Lemma \ref{kfl vect bdle as lim}. In particular, $\cE$ is classical.
\end{proof}

\begin{lem}\label{kfl vect bdle is classical after n-power ext}
    Let $(\fX,\cM_{\fX})$ be a quasi-compact fs log formal scheme and $\cE$ be a kfl vector bundle on $(\fX,\cM_{\fX})$. Suppose that we are given an fs chart $P\to \cM_{\fX}$. Then the pullback of $\cE$ by a kfl covering
    \[
    (\fX,\cM_{\fX})\times_{(\bZ[P],P)^{a}} (\bZ[P^{1/n}],P^{1/n})^{a}\to (\fX,\cM_{\fX})
    \]
    is classical for some $n\geq 1$. When $\fX$ is a scheme, the analogous statement holds for kfl coherent sheaves as well.
\end{lem}

\begin{proof}
    We prove just the assertion for kfl vector bundles. By Lemma \ref{kummer map is strict after n-power ext} and flat descent for (classical) vector bundles, $\cE|_{(X_{0},\cM_{X_{0}})\times_{(\bZ[P],P)^{a}} (\bZ[P^{1/n}],P^{1/n})^{a}}$ is classical for some $n\geq 1$, which implies that $\cE|_{(\fX,\cM_{\fX})\times_{(\bZ[P],P)^{a}} (\bZ[P^{1/n}],P^{1/n})^{a}}$ is also classical by Lemma \ref{classicality of vect bdle can be checked after taking reduction}.
\end{proof}

\begin{prop}[cf.~{\cite[Proposition 3.29]{niz08}}]\label{ext of classical vector bundle}
Let $(\fX,\cM_{\fX})$ be an fs log formal scheme. Let
   \[
   0\to \cE_{1}\to \cE_{2}\to \cE_{3}\to 0
   \]
be an exact sequence of kfl vector bundles on $(\fX,\cM_{\fX})$. Then $\cE_{2}$ is classical if and only if $\cE_{1}$ and $\cE_{3}$ are classical.
\end{prop}

\begin{proof}
    By Corollary \ref{classicality of vect bdle can be checked after taking reduction}, we are reduced to the case that $\fX$ is a scheme. In this case, the assertion is proved in \cite[Proposition 3.29]{niz08}. 
\end{proof}

To formulate ``pro-kfl'' descent for kfl vector bundles (Proposition \ref{pro-kfl descent for kfl vect bdle}), we generalize the definition of the pullback functor of kfl vector bundles to non-fs cases. Let $f\colon (\fY,\cM_{\fY})\to (\fX,\cM_{\fX})$ be an adic morphism from a (not necessarily fs) saturated log formal scheme $(\fY,\cM_{\fY})$ to an fs log formal scheme $(\fX,\cM_{\fX})$.

\begin{dfn}
    Let $\cE$ be a kfl vector bundle on $(\fX,\cM_{\fX})$. We say that $\cE$ is $f^{*}$-\emph{admissible} if, strict flat locally on $(\fY,\cM_{\fY})$, the map $f$ factors as
    \[
    (\fY,\cM_{\fY})\stackrel{g}{\to} (\fZ,\cM_{\fZ})\stackrel{h}{\to} (\fX,\cM_{\fX}),
    \]
    where $(\fZ,\cM_{\fZ})$ is fs and $h^{*}\cE$ is classical. 
    
    Let $\mathrm{Vect}^{f^{*}\text{-adm}}_{\mathrm{kfl}}(\fX,\cM_{\fX})$ denote the full subcategory of $\mathrm{Vect}_{\mathrm{kfl}}(\fX,\cM_{\fX})$ consisting of $f^{*}$-admissible objects. We define an exact functor
    \[
    f^{*}\colon \mathrm{Vect}^{f^{*}\text{-adm}}_{\mathrm{kfl}}(\fX,\cM_{\fX})\to \mathrm{Vect}(\fY)
    \]
    as follows. Let $\cE\in \mathrm{Vect}^{f^{*}\text{-adm}}_{\mathrm{kfl}}(\fX,\cM_{\fX})$. By working flat locally on $\fY$, we may assume that $f$ admits a factorization as above. In this situation, we set
    \[
    f^{*}\cE\coloneqq g^{*}h^{*}\cE,
    \]
    where $g^{*}$ means the pullback functor for (usual) vector bundles along $g\colon \fY\to \fZ$. This definition is independent of the choice of a factorization $f=hg$ because two such factorizations are dominated by another factorization (for example, by taking the fiber product). Note that, if $(\fX,\cM_{\fX})$ is fs, $f^{*}\cE$ coincides with the usual pullback.

    When $\fX$ and $\fY$ are schemes, the $f^{*}$-admissibility for kfl coherent sheaves can be defined in the same way, and, if we let $\mathrm{Coh}^{f^{*}\text{-adm}}_{\mathrm{kfl}}(\fX,\cM_{\fX})$ denote the full subcategory of $\mathrm{Coh}_{\mathrm{kfl}}(\fX,\cM_{\fX})$ consisting of $f^{*}$-admissible objects, we can define a functor
    \[
    f^{*}\colon \mathrm{Coh}^{f^{*}\text{-adm}}_{\mathrm{kfl}}(\fX,\cM_{\fX})\to \mathrm{Coh}(\fY)
    \]
    in the same way.
\end{dfn}

\begin{lem}\label{functoriality of pro-kfl pullback of kfl vect bdles}

    \begin{enumerate}
        \item Let $(\fZ,\cM_{\fZ})$ be a saturated log formal scheme, and $g\colon (\fZ,\cM_{\fZ})\to (\fY,\cM_{\fY})$ be an adic morphism. Let $\cE$ be a kfl vector bundle on $(\fX,\cM_{\fX})$. If $\cE$ is $f^{*}$-admissible, then $\cE$ is also $(fg)^{*}$-admissible and $(fg)^{*}\cE\cong g^{*}f^{*}\cE$.
        \item Let $(\fW,\cM_{\fW})$ be an fs log formal scheme and $h\colon (\fX,\cM_{\fX})\to (\fW,\cM_{\fW})$ be an adic morphism. Let $\cE$ be a kfl vector bundle on $(\fW,\cM_{\fW})$. Then $\cE$ is $(hf)^{*}$-admissible if and only if $h^{*}\cE$ is $f^{*}$-admissible. Moreover, in this case, we have $(hf)^{*}\cE\cong f^{*}(h^{*}\cE)$.
    \end{enumerate}

    Moreover, when $\fX,\fY,\fZ,\fW$ are schemes, all assertions hold for kfl coherent sheaves. 
\end{lem}

\begin{proof}
    This follows directly from the definition. 
\end{proof}

The following lemma gives non-trivial examples of $f^{*}$-admissible kfl vector bundles that we treat in this paper.

\begin{lem}\label{pro-kfl pullback of kfl vect bdles is well-def}
     Suppose that we are given an fs chart $P\to \cM_{\fX}$ and a chart $P_{\bQ_{\geq 0}}\to \cM_{\fY}$ such that the following diagram is commutative:
    \[
    \begin{tikzcd}
        P \ar[r] \ar[d] & f^{-1}\cM_{\fX} \ar[d] \\
        P_{\bQ_{\geq 0}} \ar[r] & \cM_{\fY}.
    \end{tikzcd}
    \]
    Then every kfl vector bundle $\cE$ on $(\fX,\cM_{\fX})$ is $f^{*}$-admissible. When $\fX$ and $\fY$ are schemes, the analogous statement holds for kfl coherent sheaves.
\end{lem}

\begin{proof}
    By working Zariski locally on $\fX$, we may assume that $\fX$ is quasi-compact. Since $f$ factors as
    \[
    (\fY,\cM_{\fY})\to (\fX,\cM_{\fX})\times_{(\bZ[P],P)^{a}} (\bZ[P^{1/n}],P^{1/n})^{a}\to (\fX,\cM_{\fX})
    \]
    for each $n\geq 1$, the assertion follows from Lemma \ref{kfl vect bdle is classical after n-power ext}.
\end{proof}

\begin{setting}\label{setting for pro-kfl descent}

Let $(\fX,\cM_{\fX})$ be an fs log formal scheme and $\alpha\colon P\to \cM_{\fX}$ be an fs chart. We define $(\fX_{\infty,\alpha},\cM_{\fX_{\infty,\alpha}})$ by
\[
(\fX_{\infty,\alpha},\cM_{\fX_{\infty,\alpha}})\coloneqq (\fX,\cM_{\fX})\times_{(\bZ[P],P)^{a}} (\bZ[P_{\bQ_{\geq 0}}],P_{\bQ_{\geq 0}})^{a},
\]
and define $(\fX^{(\bullet)}_{\infty,\alpha},\cM_{\fX^{(\bullet)}_{\infty,\alpha}})$ as the \v{C}ech nerve of $(\fX_{\infty,\alpha},\cM_{\fX_{\infty,\alpha}})\to (\fX,\cM_{\fX})$ in the category of saturated log formal schemes. In the same way, we define $(\fX_{m,\alpha},\cM_{\fX_{m,\alpha}})$ by
\[
(\fX_{m,\alpha},\cM_{\fX_{m,\alpha}})\coloneqq (\fX,\cM_{\fX})\times_{(\bZ[P],P)^{a}} (\bZ[P^{1/m}],P^{1/m})^{a},
\]
and define $(\fX^{(\bullet)}_{m,\alpha},\cM_{\fX^{(\bullet)}_{m,\alpha}})$ as the \v{C}ech nerve of $(\fX_{m,\alpha},\cM_{\fX_{m,\alpha}})\to (\fX,\cM_{\fX})$ in the category of saturated log formal schemes for $m\geq 1$.
    
\end{setting}

\begin{prop}[``Pro-kfl'' descent for kfl vector bundles]\label{pro-kfl descent for kfl vect bdle} \noindent

Let $f\colon (\fY,\cM_{\fY})\to  (\fX,\cM_{\fX})$ be an adic surjection of saturated log formal schemes. Suppose that there exists an fs chart $\alpha\colon P\to \cM_{\fX}$ such that $f$ admits a factorization
\[
(\fY,\cM_{\fY})\to (\fX_{\infty,\alpha},\cM_{\fX_{\infty,\alpha}})\to (\fX,\cM_{\fX}),
\]
where $(\fY,\cM_{\fY})\to (\fX_{\infty,\alpha},\cM_{\fX_{\infty,\alpha}})$ is strict quasi-compact flat under the notation in Setting \ref{setting for pro-kfl descent}. Let $(\fY^{(\bullet)},\cM_{\fY^{(\bullet)}})$ be the \v{C}ech nerve of $f$ in the category of saturated log formal schemes. Then there exists a natural bi-exact equivalence
    \[
    \mathrm{Vect}_{\mathrm{kfl}}(\fX,\cM_{\fX})\isom \varprojlim_{\bullet\in \Delta}\mathrm{Vect}(\fY^{(\bullet)}).
    \]
Moreover, when $\fX$ is a scheme, there exists a natural equivalence
    \[
    \mathrm{Coh}_{\mathrm{kfl}}(\fX,\cM_{\fX})\isom \varprojlim_{\bullet\in \Delta}\mathrm{Coh}(\fY^{(\bullet)}).
    \]
\end{prop}

\begin{proof}
    We just treat $\mathrm{Vect}_{\mathrm{kfl}}$ because the assertion for $\mathrm{Cof}_{\mathrm{kfl}}$ follows from the same argument. By Lemma \ref{functoriality of pro-kfl pullback of kfl vect bdles} (1) and Lemma \ref{pro-kfl pullback of kfl vect bdles is well-def}, the functor $f^{*}$ induces an exact functor 
    \[
    \mathrm{Vect}_{\mathrm{kfl}}(\fX,\cM_{\fX})\to \varprojlim_{\bullet\in \Delta}\mathrm{Vect}(\fY^{(\bullet)}).
    \]
    We shall prove that this functor is a bi-exact equivalence. 
    
    Consider the following diagram in which every square is a saturated fiber product:
    \[
    \begin{tikzcd}
        (\fZ,\cM_{\fZ}) \ar[r] \ar[d]  & (\fX^{(1)}_{\infty,\alpha},\cM_{\fX^{(1)}_{\infty,\alpha}}) \ar[r] \ar[d] &(\fX_{\infty,\alpha},\cM_{\fX_{\infty,\alpha}}) \ar[d] \\
        (\fY,\cM_{\fY}) \ar[rr, bend right, "f"] \ar[r] & (\fX_{\infty,\alpha},\cM_{\fX_{\infty,\alpha}}) \ar[r] & (\fX,\cM_{\fX}).
    \end{tikzcd}
    \]
    By the assumption, the upper left horizontal map is strict quasi-compact flat. Lemma \ref{sat prod of two pro-kfl cover} implies that the upper right horizontal map and the middle vertical map are strict fpqc coverings. Hence, the left vertical map and the upper horizontal map $(\fZ,\cM_{\fZ})\to (\fX_{\infty,\alpha},\cM_{\fX_{\infty,\alpha}})$ are strict fpqc coverings by Lemma \ref{four point lemma in nonfs case}. Therefore, due to flat descent for (usual) vector bundles, it is enough to show that the functor
    \[
    \mathrm{Vect}_{\mathrm{kfl}}(\fX,\cM_{\fX})\to \varprojlim_{\bullet\in \Delta}\mathrm{Vect}(\fX_{\infty,\alpha}^{(\bullet)})
    \]
    is a bi-exact equivalence. Lemma \ref{kfl vect bdle as lim} reduces the problem to the case that $\fX$ is a scheme. Then we get bi-exact equivalences
    \begin{align*}
        \mathrm{Vect}_{\mathrm{kfl}}(\fX,\cM_{\fX})&\cong \varinjlim_{m\geq 1} \varprojlim_{\bullet\in \Delta} \mathrm{Vect}(\fX^{(\bullet)}_{m,\alpha}) \\
        &\cong \varprojlim_{\bullet\in \Delta} \varinjlim_{m\geq 1} \mathrm{Vect}(\fX^{(\bullet)}_{m,\alpha}) \\
        &\cong \varprojlim_{\bullet\in \Delta} \mathrm{Vect}(\fX^{(\bullet)}_{\infty,\alpha}),
    \end{align*}
    where the first equivalence follows from Lemma \ref{kfl vect bdle is classical after n-power ext}, and the third one follows from the limit argument (cf. \cite[Appendix]{ino23}).
\end{proof}

\section{Log $p$-divisible groups on fs formal log schemes}
\subsection{Log finite group schemes}

In this subsection, we review the notion of log finite group schemes and generalize some known results to the case of log formal schemes.

\begin{dfn}[cf.~{\cite[Definition 1.3 and \S 1.6]{kat23}}]
    Let $(\fX,\cM_{\fX})$ be an fs log formal scheme and $n\geq 1$ be an integer. Assume that $p$ is topologically nilpotent on $\fX$ when $n=1$. 
    \begin{enumerate}
        \item Let $G$ be a sheaf of abelian groups on $(\fX,\cM_{\fX})_{\mathrm{kfl}}$. We call $G$ a \emph{weak log finite group scheme} (resp.~a \emph{weak log truncated Barsotti-Tate group scheme of level $n$}) if there exists a kfl covering $\{(\fU_{i},\cM_{\fU_{i}})\to (\fX,\cM_{\fX})\}_{i\in I}$ such that the restriction of $G$ to $(\fU_{i},\cM_{\fU_{i}})_{\mathrm{kfl}}$ belongs to $\mathrm{Fin}(\fU_{i})$ (resp.~$\mathrm{BT}_{n}(\fU_{i})$) via the inclusion functors in Remark \ref{classical representable sheaf} for each $i\in I$. We let $\mathrm{wFin}(\fX,\cM_{\fX})$ (resp.~$\mathrm{wBT}_{n}(\fX,\cM_{\fX})$) denote the category of weak log finite group schemes (resp.~weak log truncated Barsotti-Tate group schemes of level $n$) over $(\fX,\cM_{\fX})$. The category $\mathrm{Fin}(\fX)$ (resp.~$\mathrm{BT}_{n}(\fX)$) is regarded as the full subcategory of $\mathrm{wFin}(\fX,\cM_{\fX})$ (resp.~$\mathrm{wBT}_{n}(\fX,\cM_{\fX})$), and an object $G\in \mathrm{wFin}(\fX,\cM_{\fX})$ is called \emph{classical} if $G$ belongs to $\mathrm{Fin}(\fX)$.
        \item For a weak log finite group scheme $G$ over $(\fX,\cM_{\fX})$, we define
        \[ G^{*}\coloneqq\sHom_{(\fX,\cM_{\fX})_{\mathrm{kfl}}}(G,\bG_{m})
        \]
        (which we call the \emph{Cartier dual} of $G$), where $\bG_{m}$ is regarded as a sheaf on $(\fX,\cM_{\fX})_{\mathrm{kfl}}$ via the inclusion functors in Remark \ref{classical representable sheaf}. We say that $G$ is a \emph{log finite group scheme} if $G$ and $G^{*}$ are representable by finite and kfl log formal schemes over $(\fX,\cM_{\fX})$. We let $\mathrm{Fin}(\fX,\cM_{\fX})$ denote the category of log finite group schemes over $(\fX,\cM_{\fX})$. 
        \item An object $G\in \mathrm{wFin}(\fX,\cM_{\fX})$ is called a \emph{log truncated Barsotti-Tate group schemes of level $n$} if $G$ belongs to $\mathrm{Fin}(\fX,\cM_{\fX})$ and $\mathrm{wBT}_{n}(\fX,\cM_{\fX})$. We let $\mathrm{BT}_{n}(\fX,\cM_{\fX})$ denote the category of log truncated Barsotti-Tate group schemes of level $n$ over $(\fX,\cM_{\fX})$. 
    \end{enumerate}
\end{dfn}

For an adic morphism of fs log formal schemes $f\colon (\fX,\cM_{\fX})\to (\fY,\cM_{\fY})$ and a weak log finite group scheme $G$ over $(\fY,\cM_{\fY})$, we define the pullback
\[
f^{*}G\in \mathrm{wFin}(\fX,\cM_{\fX})
\]
by taking the pullback as a sheaf. We write $G\times_{(\fY,\cM_{\fY})} (\fX,\cM_{\fX})$ or $G_{(\fX,\cM_{\fX})}$ for $f^{*}G$. Then we have the following commutative diagram:
\[
\begin{tikzcd}
    \mathrm{Fin}(\fY) \ar[r,"f^{*}"] \ar[d,hook] & \mathrm{Fin}(\fX) \ar[d,hook] \\
    \mathrm{wFin}(\fY,\cM_{\fY}) \ar[r,"f^{*}"] & \mathrm{wFin}(\fX,\cM_{\fX}).
\end{tikzcd}
\]

\begin{lem}\label{log fin grp is classical after n-power ext}
    Let $(\fX,\cM_{\fX})$ be a quasi-compact fs log formal scheme and $G$ be a weak log finite group scheme over $(\fX,\cM_{\fX})$. Suppose that we are given an fs chart $P\to \cM_{\fX}$. Then the pullback of $G$ along a kfl covering
    \[
    (\fX,\cM_{\fX})\times_{(\bZ[P],P)^{a}} (\bZ[P^{1/n}],P^{1/n})^{a}\to (\fX,\cM_{\fX})
    \]
    is classical for some $n\geq 1$.
\end{lem}

\begin{proof}
    This follows from Lemma \ref{kummer map is strict after n-power ext} and flat descent for usual finite and locally free group schemes.
\end{proof}

\begin{lem}[cf.~{\cite[Proposition 2.16]{kat23}}]\label{classicality of log fin grp is pointwise}
    Let $(X,\cM_{X})$ be an fs log scheme and $G$ be a weak log finite group scheme. Assume that, for every $x\in X$, the pullback of $G$ to $(\mathrm{Spec}(k(x)),\cM_{k(x)})$ is classical, where $\cM_{k(x)}$ is the pullback log structure of $\cM_{X}$. Then $G$ is classical.
\end{lem}

\begin{proof}
    This assertion is proved in \cite[Proposition 2.16]{kat23} under the assumption that $X$ is locally noetherian. In general cases, we can be reduced to this case by using a limit argument (cf.~\cite[Appendix]{ino23}).
\end{proof}

\begin{lem}\label{log fin grp as lim}
    Let $(\fX,\cM_{\fX})$ be an fs log formal scheme admitting a finitely generated ideal of definition $\cI$. Let $(X_{n},\cM_{X_{n}})$ be the strict closed subscheme defined by $\cI^{n+1}$ for $n\geq 0$. Let $m\geq 1$ be an integer, and assume that $p$ is topologically nilpotent on $\fX$ when $m=1$.
    Then the natural functors
    \begin{align*}
    \mathrm{wFin}(\fX,\cM_{\fX})&\to \varprojlim_{n} \mathrm{wFin}(X_{n},\cM_{X_{n}}) \\
    \mathrm{Fin}(\fX,\cM_{\fX})&\to \varprojlim_{n} \mathrm{Fin}(X_{n},\cM_{X_{n}}) \\
    \mathrm{wBT}_{m}(\fX,\cM_{\fX})&\to \varprojlim_{n} \mathrm{wBT}_{m}(X_{n},\cM_{X_{n}}) \\
    \mathrm{BT}_{m}(\fX,\cM_{\fX})&\to \varprojlim_{n} \mathrm{BT}_{m}(X_{n},\cM_{X_{n}}) 
    \end{align*}
    give equivalences.
\end{lem}

\begin{proof}
    We shall construct a quasi-inverse functor of the first one. Let $\{G_{n}\}_{n\geq 0}$ be an object of $\displaystyle \varprojlim_{n\geq 0} \mathrm{wFin}(X_{n},\cM_{X_{n}})$. In other words, $\{G_{n}\}_{n\geq 0}$ is a system of weak log finite group schemes $G_{n}$ on $(X_{n},\cM_{X_{n}})$ with isomorphisms $G_{n+1,(X_{n},\cM_{X_{n}})}\cong G_{n}$ for each $n\geq 0$. Since $(\fX,\cM_{\fX})_{\mathrm{kfl}}$ has a basis consisting of objects in $\displaystyle \bigcup_{n\geq 0}(X_{n},\cM_{X_{n}})_{\mathrm{kfl}}$, the system $\{G_{n}\}_{n\geq 0}$ gives a unique sheaf $G$ on $(\fX,\cM_{\fX})_{\mathrm{kfl}}$. It is enough to show that $G$ is a weak log finite group scheme. By working \'{e}tale locally on $\fX$, we may assume that $\fX$ is quasi-compact and there exists an fs chart $P\to \cM_{\fX}$. Lemma \ref{log fin grp is classical after n-power ext} implies that the pullback of $G_{0}$ to $(X_{0},\cM_{X_{0}})\times_{(\bZ[P],P)^{a}} (\bZ[P^{1/m}],P^{1/m})^{a}$ is classical for some $m\geq 1$. Passing to a kfl covering $(\fX,\cM_{\fX})\times_{(\bZ[P],P)^{a}} (\bZ[P^{1/m}],P^{1/m})^{a}$, we may assume that $G_{0}$ is classical. In this case, $G_{n}$ is classical for each $n\geq 0$ by Lemma \ref{classicality of log fin grp is pointwise}, and so $G$ is the (classical) finite locally free group scheme corresponding to the system $\{G_{n}\}_{n\geq 0}$. This proves the claim. The equivalence of the third functor also follows in the same way.

    To complete the proof, it is enough to prove the equivalence of the second functor in the statement. To do this, it suffices to show the following: for a weak log finite group scheme $H$ over $(\fX,\cM_{\fX})$, if $H_{n}\coloneqq H_{(X_{n},\cM_{X_{n}})}$ belongs to $\mathrm{Fin}(X_{n},\cM_{X_{n}})$ for each $n\geq 0$, $H$ also belongs to $\mathrm{Fin}(\fX,\cM_{\fX})$. Let $(Z_{n},\cM_{Z_{n}})$ be a finite and kfl fs log scheme over $(X_{n},\cM_{X_{n}})$ representing $H_{n}$. If we set $\displaystyle \fZ\coloneqq \varinjlim_{n\geq 0} Z_{n}$, the formal scheme $\fZ$ is finite (in particular, adic) over $\fX$ by \cite[Tag 09B8]{sp24}. If we put $\displaystyle \cM_{\fZ}\coloneqq\varprojlim_{n\geq 0} \cM_{Z_{n}}$ 
    under the identification $\fZ_{\et}\simeq (Z_{n})_{\et}$, the pair $(\fZ,\cM_{\fZ})$ is a finite and kfl fs log formal scheme over $(\fX,\cM_{\fX})$ with
    \[
    (\fZ,\cM_{\fZ})\times_{(\fX,\cM_{\fX})} (X_{n},\cM_{X_{n}}) \cong (Z_{n},\cM_{Z_{n}})
    \]
    and represents $H$. By the same argument, we see that $H^{*}$ is also represented by a finite and kfl fs log formal scheme over $(\fX,\cM_{\fX})$. Therefore $H$ belongs to $\mathrm{Fin}(\fX,\cM_{\fX})$. 
\end{proof}

\begin{cor}\label{classicality of log fin grp can be checked after taking reduction}
     Let $(\fX,\cM_{\fX})$ be a log formal scheme admitting a finitely generated ideal of definition $\cI$ and $(X_{0},\cM_{X_{0}})$ be the strict closed subscheme defined by $\cI$. Let $G$ be a weak log finite group scheme over $(\fX,\cM_{\fX})$. Then $G$ is classical if and only if the pullback of $G$ to $(X_{0},\cM_{X_{0}})$ is classical.
\end{cor}

\begin{proof}
    Let $G_{n}$ denote the pullback of $G$ to $(X_{n},\cM_{X_{n}})$ for $n\geq 0$. Suppose that $G_{0}$ is classical. By Lemma \ref{classicality of log fin grp is pointwise}, $G_{n}$ is also classical for every $n\geq 0$, and the system $\{G_{n}\}_{n\geq 0}$ corresponds to a (classical) finite locally free group scheme $G'$ over $\fX$. Then Lemma \ref{log fin grp as lim} implies that $G'$ is isomorphic to $G$. In particular, $G$ is classical.
\end{proof}

\begin{lem}\label{lem for conn et seq}
The following statements are true.
\begin{enumerate}
    \item Let $(S,\cM_{S})\hookrightarrow (T,\cM_{T})$ be a strict nilpotent closed immersion of fs log schemes. Let $G\in \mathrm{wFin}(T,\cM_{T})$ and $H$ be a (classical) finite \'{e}tale group scheme over $T$. Then the natural map
    \[
    \mathrm{Hom}_{(T,\cM_{T})}(G,H)\to \mathrm{Hom}_{(S,\cM_{S})}(G_{(S,\cM_{S})},H_{(S,\cM_{S})})
    \]
    is an isomorphism.
    \item Let $(S,\cM_{S})$ be an fs log scheme. Let $G\in \mathrm{wFin}(S,\cM_{S})$, $H$ be a (classical) finite \'{e}tale group scheme over $S$, and $f:G\to H$ be a group map. Then the sheaf $\mathrm{Ker}(f)$ is a weak log finite group scheme over $(S,\cM_{S})$.
\end{enumerate}

\end{lem}

\begin{proof}
    Since both problems are kfl local on bases, we may assume that $G$ is classical. The assertion $(1)$ is clear. For the assertion $(2)$, considering a graph morphism, we see that $f$ is finite and locally free, which implies that $\mathrm{Ker}(f)$ is a finite and locally free group scheme.
\end{proof}

\begin{lem}\label{conn et seq}
    Let $(\fX,\cM_{\fX})$ be an fs log formal scheme and $G$ be a log finite group scheme over $(\fX,\cM_{\fX})$. Suppose that $G$ is killed by some power of $p$ and that $p$ is topologically nilpotent on $\fX$. Then there exists a strict \'{e}tale cover $(\fY,\cM_{\fY})\to (\fX,\cM_{\fX})$ and an exact sequence
    \[
    0\to H^{0}\to G_{(\fY,\cM_{\fY})}\to H^{\et}\to 0,
    \]
    where $H^{0}$ is a (classical) finite locally free group scheme over $\fY$ and $H^{\et}$ is a (classical) finite \'{e}tale group scheme over $\fY$. Moreover, if $G\in \mathrm{BT}_{n}(\fX,\cM_{\fX})$, then $H^{0}$ and $H^{\et}$ also belong to $\mathrm{BT}_{n}(\fY)$.
\end{lem}

\begin{proof}
    By working Zariski locally on $\fX$, we may assume that $\fX$ admits a finitely generated ideal of definition $\cI$. Let $(X_{n},\cM_{X_{n}})$ be a strict closed subscheme of $(\fX,\cM_{\fX})$ defined by $\cI^{n+1}$ and let $G_{n}\coloneqq G_{(X_{n},\cM_{X_{n}})}$ for $n\geq 0$. By \cite[Proposition 2.7(3)]{kat23} and the limit argument (see \cite[Appendix]{ino23}), there exists a strict \'{e}tale cover $(Y_{0},\cM_{Y_{0}})\to (X_{0},\cM_{X_{0}})$ and an exact sequence
    \[
    0\to H^{0}_{0}\to G_{0,(Y_{0},\cM_{Y_{0}})}\to H^{\et}_{0}\to 0
    \]
    where $H^{0}_{0}$ is a (classical) finite locally free group scheme over $Y_{0}$ and $H^{\et}_{0}$ is a (classical) finite \'{e}tale group scheme over $Y_{0}$. Let $Y_{n}\to X_{n}$ be a unique \'{e}tale cover lifting $Y_{0}\to X_{0}$ and $H^{\et}_{n}$ be a unique lift of $H^{\et}_{0}$ over $Y_{n}$. By Lemma \ref{lem for conn et seq}(1), the surjection $G_{0,(Y_{0},\cM_{Y_{0}})}\to H^{\et}_{0}$ has a unique lift $G_{n,(Y_{n},\cM_{Y_{n}})}\to H^{\et}_{n}$. Let $H^{0}_{n}\coloneqq \mathrm{Ker}(G_{n,(Y_{n},\cM_{Y_{n}})}\to H^{\et}_{n})$. By Lemma \ref{lem for conn et seq}(2), $H^{0}_{n}$ is a weak log finite group scheme over $(Y_{n},\cM_{Y_{n}})$, and we have an exact sequence
    \[
    0\to H^{0}_{n}\to G_{n,(Y_{n},\cM_{Y_{n}})}\to H^{\et}_{n}\to 0.
    \]
    By taking the pullback to $(Y_{0},\cM_{Y_{0}})$, we get $H^{0}_{n,(Y_{0},\cM_{Y_{0}})}\cong H^{0}_{0}$, which implies that $H^{0}_{n}$ is classical by Lemma \ref{classicality of log fin grp is pointwise}. If we let $H^{0}$ (resp. $H^{\et}$) denote the finite locally free (resp. finite \'{e}tale) group scheme over $\fY$ corresponding to a system $\{H^{0}_{n}\}_{n\geq 0}$ (resp. $\{H^{\et}_{n}\}_{n\geq 0}$), the assertion follows from Lemma \ref{log fin grp as lim}.
\end{proof}

We define the generic fiber of weak log finite group schemes over $p$-adic fs log formal schemes. To give a formulation over general $p$-adic fs log formal schemes, we use the formalism of log diamonds in \cite[Chapter 7]{ky23}.

\begin{construction}[Generic fibers of weak log finite group schemes]\label{gen fib of weak log fin}
Let $(\fX,\cM_{\fX})$ be a $p$-adic fs log formal scheme and $G$ be a weak log finite group scheme over $(\fX,\cM_{\fX})$. Let $\sC$ denote the category of a $p$-adic fs log formal scheme $(\fY,\cM_{\fY})$ over $(\fX,\cM_{\fX})$ satisfying the following conditions:
    \begin{itemize}
        \item $(\fY,\cM_{\fY})\to (\fX,\cM_{\fX})$ is kfl;
        \item $G_{(\fY,\cM_{\fY})}$ is classical;
        \item $(\fY,\cM_{\fY})^{\Diamond}_{\eta}\to (\fX,\cM_{\fX})^{\Diamond}_{\eta}$ is Kummer log \'{e}tale.
    \end{itemize}
The category $\sC$ has finite products. there exists a kfl covering of $(\fX,\cM_{\fX})$ consisting of objects of $\sC$. Indeed, when $\fX$ is quasi-compact and we are given an fs chart $P\to \cM_{\fX}$, a kfl cover
\[
(\fX,\cM_{\fX})\times_{(\bZ[P],P)^{a}} (\bZ[P^{1/n}],P^{1/n})^{a}\to (\fX,\cM_{\fX})
\]
belongs to $\sC$ for a sufficiently large $n\geq 1$ by Lemma \ref{log fin grp is classical after n-power ext} and \cite[Proposition 7.20]{ky23}. Hence, by Kummer log \'{e}tale descent, we have a fully faithful functor
\[
\varprojlim_{(\fY,\cM_{\fY})\in \sC} \mathrm{LC}(\fY^{\Diamond}_{\eta,\mathrm{qpro\et}})\to \mathrm{LC}((\fX,\cM_{\fX})^{\Diamond}_{\eta,\mathrm{qprok\et}}),
\]
where $\mathrm{LC}(\sA)$ denotes the category of locally constant sheaves on $\sA$ for a site $\sA$.

For $(\fY,\cM_{\fY})\in \sC$, we have the generic fiber $(G_{(\fY,\cM_{\fY})})_{\eta}\in \mathrm{LC}(\fY^{\Diamond}_{\eta,\mathrm{qpro\et}})$ of the (classical) finite locally free group $G_{(\fY,\cM_{\fY})}$ over $\fY$, and, via the functor in the previous paragraph, the family of $(G_{(\fY,\cM_{\fY})})_{\eta}$ gives a locally constant sheaf $G_{\eta}$ on $(\fX,\cM_{\fX})^{\Diamond}_{\eta,\mathrm{qprok\et}}$. The sheaf $G_{\eta}$ is called the \emph{generic fiber} of $G$. The functor taking generic fibers is exact. When $G$ is a weak log Barsotti-Tate group scheme of level $n$, the sheaf $G_{\eta}$ is a $\bZ/p^{n}$-local system on $(\fX,\cM_{\fX})^{\Diamond}_{\eta,\mathrm{qprok\et}}$.
\end{construction}

To formulate ``pro-kfl'' descent for weak log finite group schemes (Proposition \ref{pro-kfl descent for log fin grp}), we generalize the definition of the pullback functor of weak log finite group schemes to non-fs cases. Let $f\colon (\fY,\cM_{\fY})\to (\fX,\cM_{\fX})$ be an adic morphism from a (not necessarily fs) saturated log formal scheme $(\fY,\cM_{\fY})$ to an fs log formal scheme $(\fX,\cM_{\fX})$.

\begin{dfn}
    For a weak log finite group scheme $G$ over $(\fX,\cM_{\fX})$, we say that $G$ is \emph{$f^{*}$-admissible} if, strict flat locally on $(\fX,\cM_{\fX})$, the map $f$ factors as
    \[
    (\fY,\cM_{\fY})\stackrel{g}{\to} (\fZ,\cM_{\fZ})\stackrel{h}{\to} (\fX,\cM_{\fX}),
    \]
    where $(\fZ,\cM_{\fZ})$ is fs and $G_{(\fZ,\cM_{\fZ})}$ is classical. 

    Let $\mathrm{wFin}^{{f^{*}\text{-adm}}}(\fX,\cM_{\fX})$ denote the full subcategory of $\mathrm{wFin}(\fX,\cM_{\fX})$ consisting of $f^{*}$-admissible objects. We define an exact functor
    \[
    f^{*}\colon \mathrm{wFin}^{{f^{*}\text{-adm}}}(\fX,\cM_{\fX})\to \mathrm{Fin}(\fY)
    \]
    as follows. Let $G\in \mathrm{wFin}^{{f^{*}\text{-adm}}}(\fX,\cM_{\fX})$. By working flat locally on $\fY$, we may assume that $f$ admits a factorization as above. In this situation, we set 
    \[ G_{(\fX,\cM_{\fX})}\coloneqq G_{(\fZ,\cM_{\fZ})}\times_{\fZ} \fX,
    \]
    where $g^{*}$ means the pullback functor for finite locally free group schemes along $g\colon \fY\to \fZ$. This definition is independent of the choice of a factorization $f=hg$ because two such factorizations are dominated by another factorization (for example, by taking the fiber product). Note that, if $(\fX,\cM_{\fX})$ is fs, $G_{(\fX,\cM_{\fX})}$ coincides with the usual pullback.
\end{dfn}

\begin{lem}\label{functoriality of pro-kfl pullback of log fin grp}
    \begin{enumerate}
        \item Let $(\fZ,\cM_{\fZ})$ be a saturated log formal scheme, and $g\colon (\fZ,\cM_{\fZ})\to (\fY,\cM_{\fY})$ be a morphism. Let $G\in \mathrm{wFin}(\fX,\cM_{\fX})$. If $G$ is $f^{*}$-admissible, then $G$ is also $(fg)^{*}$-admissible and $G_{(\fZ,\cM_{\fZ})}\cong (G_{(\fY,\cM_{\fY})})\times_{\fY} \fZ$.
        \item Let $(\fW,\cM_{\fW})$ be an fs log formal scheme and $h\colon (\fX,\cM_{\fX})\to (\fW,\cM_{\fW})$ be a morphism. Let $G\in \mathrm{wFin}(\fW,\cM_{\fW})$. Then $G$ is $(hf)^{*}$-admissible if and only if $G_{(\fX,\cM_{\fX})}$ is $f^{*}$-admissible. Moreover, in this case, we have $G_{(\fY,\cM_{\fY})}\cong (G_{(\fX,\cM_{\fX})})\times_{\fX} {\fY}$.
    \end{enumerate}
\end{lem}

\begin{proof}
    This follows directly from the definition. 
\end{proof}

\begin{lem}\label{pro-kfl pullback of log fin grp is well-def}
    Suppose that we are given an fs chart $P\to \cM_{\fX}$ and a chart $P_{\bQ_{\geq 0}}\to \cM_{\fY}$ such that the following diagram is commutative:
    \[
    \begin{tikzcd}
        P \ar[r] \ar[d] & f^{-1}\cM_{\fX} \ar[d] \\
        P_{\bQ_{\geq 0}} \ar[r] & \cM_{\fY}.
    \end{tikzcd}
    \]
    Then $G$ is $f^{*}$-admissible for every $G\in \mathrm{wFin}(\fX,\cM_{\fX})$. 
\end{lem}

\begin{proof}
     By working \'{e}tale locally on $\fX$, we may assume that $\fX$ is quasi-compact. Since $f$ factors as
    \[
    (\fY,\cM_{\fY})\to (\fX,\cM_{\fX})\times_{(\bZ[P],P)^{a}} (\bZ[P^{1/n}],P^{1/n})^{a}\to (\fX,\cM_{\fX})
    \]
    for each $n\geq 1$, the assertion follows from Lemma \ref{log fin grp is classical after n-power ext}.
\end{proof}

\begin{prop}[``Pro-kfl'' descent for weak log finite group schemes]\label{pro-kfl descent for log fin grp}
    Let $f\colon (\fY,\cM_{\fY})\to (\fX,\cM_{\fX})$ be an adic surjection of saturated log formal schemes. Suppose that there exists an fs chart $\alpha\colon P\to \cM_{\fX}$ such that $f$ admits a factorization
    \[
    (\fY,\cM_{\fY})\to (\fX_{\infty,\alpha},\cM_{\fX_{\infty,\alpha}})\to (\fX,\cM_{\fX}),
    \]
    where $(\fY,\cM_{\fY})\to (\fX_{\infty,\alpha},\cM_{\fX_{\infty,\alpha}})$ is a strict, quasi-compact, and flat. Let $(\fY^{(\bullet)},\cM_{\fY^{(\bullet)}})$ denote the \v{C}ech nerve of $f$ in the category of saturated log formal schemes. Then there exists a natural bi-exact equivalences
    \begin{align*}
        \mathrm{wFin}(\fX,\cM_{\fX})&\isom \varprojlim_{\bullet\in \Delta}\mathrm{Fin}(\fY^{(\bullet)}), \\
        \mathrm{wBT}_{n}(\fX,\cM_{\fX})&\isom \varprojlim_{\bullet\in \Delta}\mathrm{BT}_{n}(\fY^{(\bullet)}).
    \end{align*}
\end{prop}

\begin{proof}
    By Lemma \ref{functoriality of pro-kfl pullback of log fin grp} (1) and Lemma \ref{pro-kfl pullback of log fin grp is well-def}, the functor $f^{*}$ induces the functor
    \[
    \mathrm{wFin}(\fX,\cM_{\fX})\to \varprojlim_{\bullet\in \Delta}\mathrm{Fin}(\fY^{(\bullet)}).
    \]
    The equivalence can be  proved in the same way as the proof of Proposition \ref{pro-kfl descent for kfl vect bdle}.
\end{proof}

\subsection{Log $p$-divisible groups}

\begin{dfn}
Let $(\fX,\cM_{\fX})$ be an fs log formal scheme. Let $G$ be a sheaf of abelian groups on $(\fX,\cM_{\fX})_{\mathrm{kfl}}$. We call $G$ a \emph{weak log $p$-divisible group} if the following conditions are satisfied.
\begin{enumerate}
    \item A map $\times p\colon G\to G$ is surjective.
    \item For every $n\geq 1$, the sheaf $G[p^n]\coloneqq\mathrm{Ker}(\times p^{n}\colon G\to G)$ is a weak log finite group scheme over $(\fX,\cM_{\fX})$.
    \item $G=\bigcup_{n\geq 1} G[p^{n}].$
\end{enumerate}
The category of weak log $p$-divisible groups over $(\fX,\cM_{\fX})$ is denoted by $\mathrm{wBT}(\fX,\cM_{\fX})$. A weak log $p$-divisible group $G$ over $(\fX,\cM_{\fX})$ is called a \emph{log $p$-divisible group} if $G[p^n]$ is a log finite group scheme for each $n\geq 1$. The category of log $p$-divisible groups over $(\fX,\cM_{\fX})$ is denoted by $\mathrm{BT}(\fX,\cM_{\fX})$. By Remark \ref{classical representable sheaf}, the category $\mathrm{BT}(\fX)$ is regarded as the full subcategory of $\mathrm{wBT}(\fX,\cM_{\fX})$. A weak log $p$-divisible group $G$ over $(\fX,\cM_{\fX})$ is called \emph{classical} if $G$ belongs to $\mathrm{BT}(\fX)$. Clearly, $G$ is classical if and only if $G[p^n]$ is classical for each $n\geq 1$. 
\end{dfn}

\begin{lem}[cf. {\cite[Lemma 4.2]{kat23}}]\label{p tor is cl implies weak log bt is cl}
    Let $(\fX,\cM_{\fX})$ be an fs log formal scheme. A weak log $p$-divisible group $G$ over $(\fX,\cM_{\fX})$ belongs to $\mathrm{BT}(\fX,\cM_{\fX})$ if $G[p^{n}]$ belongs to $\mathrm{BT}_{n}(\fX,\cM_{\fX})$ for some $n\geq 1$.
\end{lem}

\begin{proof}
    Lemma \ref{log fin grp as lim} reduces us to the case that $\fX$ is a scheme. In this case, the assertion follows from \cite[Lemma 4.2]{kat23}. 
\end{proof}

\begin{lem}\label{log BT as lim}
    Let $(\fX,\cM_{\fX})$ be an fs log formal scheme admitting a finitely generated ideal of definition $\cI$. Let $(X_{n},\cM_{X_{n}})$ be the strict closed subscheme of $(\fX,\cM_{\fX})$ defined by $\cI^{n+1}$ for $n\geq 0$. Then the natural functors
    \[
    \mathrm{wBT}(\fX,\cM_{\fX})\to \varprojlim_{n} \mathrm{wBT}(X_{n},\cM_{X_{n}})
    \]
    \[
    \mathrm{BT}(\fX,\cM_{\fX})\to \varprojlim_{n} \mathrm{BT}(X_{n},\cM_{X_{n}})
    \]
    give bi-exact equivalences.
\end{lem}

\begin{proof}
    This follows immediately from Lemma \ref{log fin grp as lim}.
\end{proof}

\begin{cor}\label{classicality of weak log bt can be checked after taking reduction}
     Let $(\fX,\cM_{\fX})$ be a log formal scheme admitting a finitely generated ideal of definition $\cI$. Let $(X_{0},\cM_{X_{0}})$ be the strict closed subscheme defined by $\cI$. Let $G$ be a weak log $p$-divisible group over $(\fX,\cM_{\fX})$. Then, $G$ is classical if and only if the pullback of $G$ to $(X_{0},\cM_{X_{0}})$ is classical.
\end{cor}

\begin{proof}
    This follows from Lemma \ref{classicality of log fin grp can be checked after taking reduction} and Lemma \ref{p tor is cl implies weak log bt is cl}.
\end{proof}

\begin{dfn}[Generic fibers of weak log $p$-divisible groups]

Let $(\fX,\cM_{\fX})$ be a $p$-adic fs log formal scheme and $G$ be a weak log $p$-divisible group on $(\fX,\cM_{\fX})$. By Construction \ref{gen fib of weak log fin}, we have $\bZ/p^{n}$-local systems $G[p^{n}]_{\eta}$ on $(\fX,\cM_{\fX})^{\Diamond}_{\eta,\mathrm{qprok\et}}$ for each $n\geq 1$. Hence, $G_{\eta}\coloneqq \displaystyle \varprojlim_{n\geq 1} G[p^{n}]_{\eta}$ is a $\bZ_{p}$-local system on $(\fX,\cM_{\fX})_{\eta,\mathrm{qprok\et}}$, and $G_{\eta}$ is called the \emph{generic fiber} of $G$.
\end{dfn}

\begin{dfn}
    Let $f\colon (\fY,\cM_{\fY})\to (\fX,\cM_{\fX})$ be a morphism from (not necessarily fs) saturated log formal scheme $(\fY,\cM_{\fY})$ to an fs log formal scheme $(\fX,\cM_{\fX})$. For a weak log $p$-divisible group $G$ on $(\fX,\cM_{\fX})$, we say that $G$ is \emph{$f^{*}$-admissible} if $G[p^{n}]$ is $f^{*}$-admissible for each $n\geq 1$. When $G$ is $f^{*}$-admissible, we define a $p$-divisible group on $\fY$ by
    \[
    G_{(\fY,\cM_{\fY})}\coloneqq\varinjlim_{n\geq 1} G[p^{n}]_{(\fY,\cM_{\fY})}.
    \]
\end{dfn}

\begin{prop}\label{pro-kfl descent for log BT}
     Let $f\colon (\fY,\cM_{\fY})\to (\fX,\cM_{\fX})$ be an adic surjection of saturated log formal schemes. Suppose that there exists an fs chart $\alpha\colon P\to \cM_{\fX}$ such that $f$ admits a factorization
    \[
    (\fY,\cM_{\fY})\to (\fX_{\infty,\alpha},\cM_{\fX_{\infty,\alpha}})\to (\fX,\cM_{\fX}),
    \]
    where $(\fY,\cM_{\fY})\to (\fX_{\infty,\alpha},\cM_{\fX_{\infty,\alpha}})$ is a strict, quasi-compact, and flat. Let $(\fY^{(\bullet)},\cM_{\fY^{(\bullet)}})$ denote the \v{C}ech nerve of $f$ in the category of saturated log formal schemes. Then there exists a natural bi-exact equivalence 
    \[
    \mathrm{wBT}(\fX,\cM_{\fX})\to \varprojlim_{\bullet\in \Delta}\mathrm{BT}(\fY^{(\bullet)}).
    \]
\end{prop}

\begin{proof}
    This follows from Proposition \ref{pro-kfl descent for log fin grp}.
\end{proof}

\begin{lem}\label{lie alg of weak log bt}
    Let $(\fX,\cM_{\fX})$ be an fs log formal scheme on which $p$ is topologically nilpotent and $G$ be a weak log $p$-divisible group over $(\fX,\cM_{\fX})$. We let $\mathrm{Lie}(G)$ denote the sheaf on $(\fX,\cM_{\fX})_{\mathrm{kfl}}$ given by
    \[
    \mathrm{Lie}(G)(Y,\cM_{Y})\coloneqq\mathrm{Ker}(G(Y[\epsilon],\cM_{Y[\epsilon]})\to G(Y,\cM_{Y})),
    \]
    where the scheme $Y[\epsilon]$ is defined by
    $Y[\epsilon]\coloneqq Y\times_{\bZ} \bZ[t]/(t^{2})$ and $\cM_{Y[\epsilon]}$ is the pullback log structure of $\cM_{Y}$.
    Then $\mathrm{Lie}(G)$ is a kfl vector bundle on $(\fX,\cM_{\fX})$.
\end{lem}

\begin{proof}
    We may assume that $X\coloneqq \fX$ is a scheme with $p^{N}=0$ on $X$ for some $N\geq 1$. Then we have
    \[
    \mathrm{Ker}(G[p^{N+i}](Y[\epsilon],\cM_{Y[\epsilon]})\to G[p^{N+i}](Y,\cM_{Y}))\subset G[p^{N}](Y[\epsilon],\cM_{Y[\epsilon]})
    \]
    for every $i\geq 0$ and $(Y,\cM_{Y})\in (X,\cM_{X})_{\mathrm{kfl}}$. Indeed, for each fixed $i$, working kfl locally on $(X,\cM_{X})$ reduces the problem to the case that $G[p^{N+i}]$ is classical, in which case the claim follows from \cite[Chapter II, Corollary 3.3.16]{mes72}. Therefore, the sheaf $\mathrm{Lie}(G)$ is given by
    \[
    (Y,\cM_{Y})\mapsto \mathrm{Ker}(G[p^{N}](Y[\epsilon],\cM_{Y[\epsilon]})\to G[p^{N}](Y,\cM_{Y})).
    \]
    By working kfl locally on $(\fX,\cM_{\fX})$ again, we may assume that $G[p^{N}]$ is a classical truncated Barsotti-Tate group scheme of level $N$ over $X$. Then the sheaf $\mathrm{Lie}(G)$ coincides with the (usual) Lie algebra of $G[p^{N}]$, which is a (classical) vector bundle on $X$ by \cite[Proposition 2.2.1 (b)]{ill85}. This completes the proof.
\end{proof}

\begin{dfn}[Lie algebras of weak log $p$-divisible groups]

Let $(\fX,\cM_{\fX})$ be an fs log formal scheme on which $p$ is topologically nilpotent and $G$ be a weak log $p$-divisible group over $(\fX,\cM_{\fX})$. The kfl vector bundle $\mathrm{Lie}(G)$ defined in Lemma \ref{lie alg of weak log bt} is called \emph{the Lie algebra} of $G$.
\end{dfn}

\begin{lem}\label{pullback compatibility for lie alg of log bt}

Let $f\colon (\fY,\cM_{\fY})\to (\fX,\cM_{\fX})$ be an adic morphism from a saturated log formal scheme $(\fY,\cM_{\fY})$ to an fs log formal scheme $(\fX,\cM_{\fX})$ on which $p$ is topologically nilpotent. Let $G$ be a $f^{*}$-admissible weak log $p$-divisible group on $(\fX,\cM_{\fX})$. Then $\mathrm{Lie}(G)$ is also $f^{*}$-admissible, and there exists a natural isomorphism in $\mathrm{Vect}(\fY)$
\[
f^{*}\mathrm{Lie}(G)\cong \mathrm{Lie}(G_{(\fY,\cM_{\fY})}).
\]
\end{lem}

\begin{proof}
    By working flat locally on $\fX$ and $\fY$, we may assume that $\fX$ admits a finitely generated ideal of definition $\cI$ containing $p$ and that there exists a factorization of $f$
    \[
    (\fY,\cM_{\fY})\stackrel{g}{\to} (\fZ,\cM_{\fZ})\stackrel{h}{\to} (\fX,\cM_{\fX})
    \]
    such that $G[p]_{(\fZ,\cM_{\fZ})}$ is classical. Let $(Z_{0},\cM_{Z_{0}})$ denote the strict closed subscheme of $(\fZ,\cM_{\fZ})$ defined by $\cI\cO_{\fZ}$. Then we have isomorphisms
    \[
    (h^{*}\mathrm{Lie}(G))|_{(Z_{0},\cM_{Z_{0}})}\cong \mathrm{Lie}(G_{(Z_{0},\cM_{Z_{0}})})\cong \mathrm{Lie}(G[p]_{(Z_{0},\cM_{Z_{0}})}),
    \]
    which implies that $h^{*}\mathrm{Lie}(G)$ is classical by Corollary \ref{classicality of vect bdle can be checked after taking reduction}. Therefore, $\mathrm{Lie}(G)$ is $f^{*}$-admissible. The isomorphism in the assertion is constructed in a straightforward way.
\end{proof}

\begin{prop}[cf.~{\cite[Proposition 7.3]{kat23}}]\label{lie is classical iff weak log bt is log bt}
    Let $(\fX,\cM_{\fX})$ be an fs log formal scheme on which $p$ is topologically nilpotent. Then, for $G\in \mathrm{wBT}(\fX,\cM_{\fX})$, kfl vector bundles $\mathrm{Lie}(G)$ and $\mathrm{Lie}(G^{*})$ are classical if and only if $G$ is a log $p$-divisible group.
\end{prop}

\begin{proof}
    Lemma \ref{classicality of vect bdle can be checked after taking reduction} and Lemma \ref{classicality of weak log bt can be checked after taking reduction} allows us to assume that $\fX$ is a scheme. Then the assertion follows from \cite[Proposition 7.3]{kat23}.
\end{proof}

\section{Kfl log prismatic crystals}

\subsection{Log prismatic crystals}

\begin{dfn}[Log prismatic crystals]\label{def of log pris crys}
    Let $(\fX,\cM_{\fX})$ be a bounded $p$-adic log formal scheme. We let $\mathrm{Vect}((\fX, \cM_{\fX})_{\Prism})$ (resp. $\mathrm{Vect}((\fX, \cM_{\fX})^{\mathrm{str}}_{\Prism})$) denote the category of vector bundles on the ringed site on $((\fX,\cM_{\fX})_{\Prism},\cO_{\Prism})$ (resp. $((\fX, \cM_{\fX})^{\mathrm{str}}_{\Prism},\cO_{\Prism})$). Objects of $\mathrm{Vect}((\fX,\cM_{\fX})_{\Prism})$ are called \emph{log prismatic crystals}. 
    
    We let $\mathrm{Vect}^{\varphi}((\fX,\cM_{\fX})_{\Prism})$ denote the category of pairs $(\cE, \varphi_{\cE})$ where $\cE$ is an object of $\mathrm{Vect}((\fX,\cM_{\fX})_{\Prism})$ and $\varphi_{\cE}$ is an isomorphism $(\phi^{\ast}\cE)[1/\cI_{\Prism}]\isom \cE[1/\cI_{\Prism}]$. Objects of $\mathrm{Vect}^{\varphi}((\fX,\cM_{\fX})_{\Prism})$ are called \emph{log prismatic F-crystals}. For integers $a\leq b$, we let $\mathrm{Vect}^{\varphi}_{[a,b]}((\fX,\cM_{\fX})_{\Prism})$ denote the full subcategory of $\mathrm{Vect}^{\varphi}((\fX,\cM_{\fX})_{\Prism})$ consisting of objects $(\cE,\varphi_{\cE})$ with
    \[
    \cI_{\Prism}^{b}\cE\subset \varphi_{\cE}(\phi^{*}\cE)\subset \cI_{\Prism}^{a}\cE.
    \]
    In particular, let $\mathrm{DM}((\fX,\cM_{\fX})_{\Prism})\coloneqq \mathrm{Vect}^{\varphi}_{[0,1]}((\fX,\cM_{\fX})_{\Prism})$, whose objects are called \emph{log prismatic Dieudonn\'{e} crystals}. 
    Finally, we let $\mathrm{Vect}((\fX,\cM_{\fX})_{\Prism},\widetilde{\cO}_{\Prism})$ denote the category of vector bundles on the ringed site $((\fX,\cM_{\fX})_{\Prism},\widetilde{\cO}_{\Prism})$. Similarly, we define variants
    \[
    \mathrm{Vect}^{\varphi}((\fX,\cM_{\fX})_{\Prism}^{\mathrm{str}}), \mathrm{Vect}^{\varphi}_{[a,b]}((\fX,\cM_{\fX})_{\Prism}^{\mathrm{str}}), \mathrm{DM}((\fX,\cM_{\fX})_{\Prism}^{\mathrm{str}}), \mathrm{Vect}((\fX,\cM_{\fX})_{\Prism}^{\mathrm{str}},\widetilde{\cO}_{\Prism}).
    \]
    All of these categories are canonically equipped with the structure of an exact category.
\end{dfn}

\begin{rem}\label{pris crys as lim}
    For $\star\in \{\mathrm{str}, \emptyset\}$, we have the following canonical bi-exact equivalences;
    \begin{align*}
    \mathrm{Vect}((\fX,\cM_{\fX})_{\Prism}^{\star})&\isom \varprojlim_{(A,I,\cM_{A})\in (\fX,\cM_{\fX})_{\Prism}^{\star}} \mathrm{Vect}(A), \\
    \mathrm{Vect}^{\varphi}((\fX,\cM_{\fX})^{\star}_{\Prism})&\isom \varprojlim_{(A,I,\cM_{A})\in (\fX,\cM_{\fX})^{\star}_{\Prism}} \mathrm{Vect}^{\varphi}(A,I), \\
    \mathrm{Vect}^{\varphi}_{[a,b]}((\fX,\cM_{\fX})^{\star}_{\Prism})&\isom \varprojlim_{(A,I,\cM_{A})\in (\fX,\cM_{\fX})^{\star}_{\Prism}} \mathrm{Vect}^{\varphi}_{[a,b]}(A,I), \\
    \mathrm{Vect}((\fX,\cM_{\fX})_{\Prism}^{\star},\widetilde{\cO}_{\Prism})&\isom \varprojlim_{(A,I,\cM_{A})\in (\fX,\cM_{\fX})^{\star}_{\Prism}} \mathrm{Vect}(A/(p,I)).
    \end{align*}
    Here, $\mathrm{Vect}(R)$ is the category of finite projective $R$-modules for a ring $R$. For a prism $(A,I)$, $\mathrm{Vect}^{\varphi}(A,I)$ is the category of pairs $(N,\varphi_{N})$ of a finite projective $A$-module $N$ and an isomorphism $\varphi_{N}\colon (\phi_{A}^{\ast}N)[1/I]\isom N[1/I]$, and $\mathrm{Vect}^{\varphi}_{[a,b]}(A,I)$ is the full subcategory of $\mathrm{Vect}^{\varphi}(A,I)$ consisting of objects $(N,\varphi_{N})$ such that $I^{b}N\subset \varphi_{N}(\phi_{A}^{*}N)\subset I^{a}N$. For an object $\cE$ of the source category, the image of $\cE$ via the projection with respect to $(A,I,\cM_{A})$ is denoted by $\cE_{(A,I,\cM_{A})}$. When no confusion occurs, we simply write $\cE_{A}$ for this.
\end{rem}

\begin{dfn}[Analytic log prismatic \texorpdfstring{$F$}--crystals, {\cite{gr24}}, {\cite{dlms24}}]\noindent

    Let $(A,I)$ be a prism. We let $\mathrm{Vect}^{\mathrm{an}}(A,I)$ denote the exact category of vector bundles on a scheme $U(A,I)\coloneqq\Spec(A)\backslash V(p,I)$, and $\mathrm{Vect}^{\mathrm{an},\varphi}(A,I)$ denote the exact category of pairs $(\cE,\varphi_{\cE})$ consisting of $\cE\in \mathrm{Vect}^{\mathrm
    {an}}(A,I)$ and an isomorphism $\varphi_{\cE}\colon (\phi_{A}^{*}\cE)[1/I]\isom \cE[1/I]$. 

    For a bounded $p$-adic formal scheme $\fX$, we define the exact categories $\mathrm{Vect}^{\mathrm{an}}(\fX_{\Prism})$ and $\mathrm{Vect}^{\mathrm{an},\varphi}(\fX_{\Prism})$ by
    \begin{align*}
    \mathrm{Vect}^{\mathrm{an}}(\fX_{\Prism})&\coloneqq \varprojlim_{(A,I,\cM_{A})\in \fX_{\Prism}} \mathrm{Vect}^{\mathrm{an}}(A,I), \\
    \mathrm{Vect}^{\mathrm{an},\varphi}(\fX_{\Prism})&\coloneqq \varprojlim_{(A,I,\cM_{A})\in \fX_{\Prism}} \mathrm{Vect}^{\mathrm{an},\varphi}(A,I).
    \end{align*}
    
    For a bounded $p$-adic log formal scheme $(\fX,\cM_{\fX})$, we define the exact categories $\mathrm{Vect}^{\mathrm{an}}((\fX,\cM_{\fX})_{\Prism})$ and $\mathrm{Vect}^{\mathrm{an},\varphi}((\fX,\cM_{\fX})_{\Prism})$ by
    \begin{align*}
    \mathrm{Vect}^{\mathrm{an}}((\fX,\cM_{\fX})_{\Prism})&\coloneqq \varprojlim_{(A,I,\cM_{A})\in (\fX,\cM_{\fX})_{\Prism}} \mathrm{Vect}^{\mathrm{an}}(A,I), \\
    \mathrm{Vect}^{\mathrm{an},\varphi}((\fX,\cM_{\fX})_{\Prism})&\coloneqq \varprojlim_{(A,I,\cM_{A})\in (\fX,\cM_{\fX})_{\Prism}} \mathrm{Vect}^{\mathrm{an},\varphi}(A,I).
    \end{align*}
    Similarly, we define the exact categories $\mathrm{Vect}^{\mathrm{an}}((\fX,\cM_{\fX})^{\mathrm{str}}_{\Prism})$ and  $\mathrm{Vect}^{\mathrm{an},\varphi}((\fX,\cM_{\fX})^{\mathrm{str}}_{\Prism})$ by
    \begin{align*}
    \mathrm{Vect}^{\mathrm{an}}((\fX,\cM_{\fX})^{\mathrm{str}}_{\Prism})&\coloneqq \varprojlim_{(A,I,\cM_{A})\in (\fX,\cM_{\fX})^{\mathrm{str}}_{\Prism}} \mathrm{Vect}^{\mathrm{an}}(A,I), \\
    \mathrm{Vect}^{\mathrm{an},\varphi}((\fX,\cM_{\fX})^{\mathrm{str}}_{\Prism})&\coloneqq \varprojlim_{(A,I,\cM_{A})\in (\fX,\cM_{\fX})^{\mathrm{str}}_{\Prism}} \mathrm{Vect}^{\mathrm{an},\varphi}(A,I).
    \end{align*}
    Objects of the category $\mathrm{Vect}^{\mathrm{an}}((\fX,\cM_{\fX})_{\Prism})$ (resp. $\mathrm{Vect}^{\mathrm{an},\varphi}((\fX,\cM_{\fX})_{\Prism})$) are called \emph{analytic log prismatic crystals} (resp. \emph{analytic log prismatic $F$-crystals}) on $(\fX,\cM_{\fX})$. 
\end{dfn}

\begin{lem}[{\cite[Lemma 5.10]{ino25}}]\label{category of crystals is unchanged}
    Let $(\fX,\cM_{\fX})$ be a quasi-coherent and integral bounded $p$-adic formal scheme. Then the inclusion functor $(\fX, \cM_{\fX})^{\mathrm{str}}_{\Prism}\hookrightarrow (\fX, \cM_{\fX})_{\Prism}$ has a right adjoint functor
    \[
    (\fX, \cM_{\fX})_{\Prism}\to (\fX, \cM_{\fX})^{\mathrm{str}}_{\Prism} \ \ \ \  ((A,I,\cM_{A})\mapsto (A,I,\cN_{A})).
    \]
    In particular, for any $\star\in \{\emptyset,\varphi,\mathrm{an},(\mathrm{an},\varphi)\}$, the following restriction functors give bi-exact equivalences:
    \begin{align*}
    \mathrm{Vect}^{\star}((\fX, \cM_{\fX})_{\Prism})&\to \mathrm{Vect}^{\star}((\fX, \cM_{\fX})^{\mathrm{str}}_{\Prism}), \\
    \mathrm{Vect}^{\varphi}_{[a,b]}((\fX,\cM_{\fX})_{\Prism})&\to \mathrm{Vect}^{\varphi}_{[a,b]}((\fX,\cM_{\fX})_{\Prism}^{\mathrm{str}}), \\
    \mathrm{Vect}((\fX, \cM_{\fX})_{\Prism},\widetilde{\cO}_{\Prism})&\to \mathrm{Vect}((\fX, \cM_{\fX})^{\mathrm{str}}_{\Prism},\widetilde{\cO}_{\Prism}).
    \end{align*}
\end{lem}

\begin{lem}[{\cite[Lemma 5.17]{ino25}}]\label{divisible monoid and log pris site}
    Let $(\fX,\cM_{\fX})$ be a bounded $p$-adic integral log formal scheme. Assume that $(\fX,\cM_{\fX})$ admits a chart $M\to \cM_{\fX}$ where $M$ is uniquely $p$-divisible (i.e. the monoid map $\times p\colon M\to M$ is bijective). Then the natural functor
    \[
    (\fX,\cM_{\fX})^{\mathrm{str}}_{\Prism}\to \fX_{\Prism}
    \]
    sending $(A,I,\cM_{A})$ to $(A,I)$ is an equivalence.

    In particular, for $\star\in \{\emptyset,\varphi,\mathrm{an},(\mathrm{an},\varphi)\}$, the following natural functors give bi-exact equivalences:
    \begin{align*}
    \mathrm{Vect}^{\star}(\fX_{\Prism})&\to \mathrm{Vect}^{\star}((\fX,\cM_{\fX})^{\mathrm{str}}_{\Prism}), \\
    \mathrm{Vect}^{\varphi}_{[a,b]}(\fX_{\Prism})&\to 
    \mathrm{Vect}^{\varphi}_{[a,b]}((\fX,\cM_{\fX})_{\Prism}^{\mathrm{str}}), \\
    \mathrm{Vect}(\fX_{\Prism},\widetilde{\cO}_{\Prism})&\to \mathrm{Vect}((\fX,\cM_{\fX})^{\mathrm{str}}_{\Prism},\widetilde{\cO}_{\Prism}). 
    \end{align*}
    
\end{lem}

\subsection{Kfl prismatic \texorpdfstring{$F$}--crystals}

The category of log prismatic crystals does not satisfy kfl descent. In this subsection, we construct a category which includes the category of log prismatic $F$-crystals as a full subcategory and satisfies the descent property for special kfl covers (Proposition \ref{kqsyn descent for kfl log pris crys}).

\begin{dfn}[Kfl prismatic crystals]\label{def of kfl pris crys}
    Let $(\fX,\cM_{\fX})$ be a bounded $p$-adic fs log formal scheme. We define an exact category $\mathrm{Vect}_{\mathrm{kfl}}((\fX,\cM_{\fX})_{\Prism})$ by
    \[
    \mathrm{Vect}_{\mathrm{kfl}}((\fX,\cM_{\fX})_{\Prism})\coloneqq \varprojlim_{(A,I,\cM_{A})\in (\fX,\cM_{\fX})^{\mathrm{str}}_{\Prism}} \mathrm{Vect}_{\mathrm{kfl}}((\Spf(A),\cM_{A})).
    \]
    An object of $\mathrm{Vect}_{\mathrm{kfl}}((\fX,\cM_{\fX})_{\Prism})$ are called a \emph{kfl prismatic crystal} on $(\fX,\cM_{\fX})$.

    Similarly, the exact category $\mathrm{Vect}_{\mathrm{kfl}} ((\fX,\cM_{\fX})_{\Prism},\widetilde{\cO}_{\Prism})$ is defined by
    \[
    \mathrm{Vect}_{\mathrm{kfl}}((\fX,\cM_{\fX})_{\Prism},\widetilde{\cO}_{\Prism})\coloneqq \varprojlim_{(A,I,\cM_{A})\in (\fX,\cM_{\fX})^{\mathrm{str}}_{\Prism}} \mathrm{Vect}_{\mathrm{kfl}}(\mathrm{Spec} (A/(p,I)),\cM_{A/(p,I)}).
    \]
    Here, $\cM_{A/(p,I)}$ is the inverse image log structure of $\cM_{A}$.
\end{dfn}

\begin{dfn}[Kfl prismatic \texorpdfstring{$F$}--crystals]

 Let $(\fX,\cM_{\fX})$ be a bounded $p$-adic fs log formal scheme. For a log prism $(A,I,\cM_{A})\in (\fX,\cM_{\fX})_{\Prism}^{\mathrm{str}}$ and integers $a\leq b$, we let
$\mathrm{Vect}^{\varphi}_{\mathrm{kfl},[a,b]}(A,I,\cM_{A})$ denote the exact category of triples consisting of a kfl vector bundle $\cE$ on $(\mathrm{Spf}(A),\cM_{A})$ and maps of kfl vector bundles $I^{b}\otimes_{A} \cE\to \phi_{A}^{*}\cE$ and  $\phi^{*}_{A}\cE\to I^{a}\otimes_{A} \cE$ such that the composition map $I^{b}\otimes_{A} \cE\to I^{a}\otimes_{A} \cE$ coincides with the multiplication by $I^{b-a}$. We define the exact category $\mathrm{Vect}^{\varphi}_{\mathrm{kfl},[a,b]}((\fX,\cM_{\fX})_{\Prism})$ by
    \[
    \mathrm{Vect}^{\varphi}_{\mathrm{kfl},[a,b]}((\fX,\cM_{\fX})_{\Prism})\coloneqq \varprojlim_{(A,I,\cM_{A})\in (\fX,\cM_{\fX})_{\Prism}^{\mathrm{str}}} \mathrm{Vect}^{\varphi}_{\mathrm{kfl},[a,b]}(A,I,\cM_{A}).
    \]
Especially, we let $\mathrm{DM}_{\mathrm{kfl}}((\fX,\cM_{\fX})_{\Prism})$ denote the category $\mathrm{Vect}^{\varphi}_{\mathrm{kfl},[0,1]}((\fX,\cM_{\fX})_{\Prism})$, and an object of the category $\mathrm{DM}_{\mathrm{kfl}}((\fX,\cM_{\fX})_{\Prism})$ is called a \emph{kfl prismatic Dieudonn\'{e} crystal} on $(\fX,\cM_{\fX})$. 

For a log prism $(A,I,\cM_{A})\in (\fX,\cM_{\fX})_{\Prism}^{\mathrm{str}}$, we let $\mathrm{Vect}_{\mathrm{kfl}}^{\varphi}(A,I,\cM_{A})$ denote the exact category $\displaystyle \varprojlim_{a\leq b} \mathrm{Vect}_{\mathrm{kfl},[a,b]}^{\varphi}(A,I,\cM_{A})$. We define the exact category $\mathrm{Vect}^{\varphi}_{\mathrm{kfl}}((\fX,\cM_{\fX})_{\Prism})$ by
    \[
    \mathrm{Vect}^{\varphi}_{\mathrm{kfl}}((\fX,\cM_{\fX})_{\Prism})\coloneqq \varprojlim_{(A,I,\cM_{A})\in (\fX,\cM_{\fX})_{\Prism}^{\mathrm{str}}} \mathrm{Vect}^{\varphi}_{\mathrm{kfl}}(A,I,\cM_{A}).
    \]
    An object of $\mathrm{Vect}^{\varphi}_{\mathrm{kfl}}((\fX,\cM_{\fX})_{\Prism})$ is called a \emph{kfl prismatic $F$-crystal} on $(\fX,\cM_{\fX})$.
\end{dfn}

\begin{rem}\label{rem on def of kfl pris F-crys}
    In general, a pair $(\cE,\varphi_{\cE})$ does not uniquely determine an object of $\mathrm{Vect}^{\varphi}_{\mathrm{kfl},[a,b]}(A,I,\cM_{A})$ because the natural map $I^{b}\otimes_{A} \cE\to I^{a}\otimes_{A} \cE$ is not necessarily a monomorphism in the category of kfl vector bundles. This problem is due to the fact that an underlying ring map of a Kummer log flat morphism is not necessarily flat. Hence, for example, when $(\mathrm{Spf}(A),\cM_{A})$ admits a free chart \'{e}tale locally, the natural map $I^{b}\otimes_{A} \cE\to I^{a}\otimes_{A} \cE$ is a monomorphism and the map $I^{b}\otimes_{A} \cE\to \phi_{A}^{*}\cE$ is uniquely determined by the pair $(\cE,\varphi_{\cE})$. However, we do not use it in this paper.
\end{rem}

\begin{rem}\label{no notion of kfl anal pris crys}
    We cannot define a kfl version of analytic log prismatic $F$-crystal in the same way because, for a log prism $(A,I,\cM_{A})$, we have no canonical log structure on $U(A,I)$.
\end{rem}

\begin{lem}\label{fully faithfulness for kfl pris crys}
    Let $(\fX,\cM_{\fX})$ be a bounded $p$-adic fs log formal scheme. Then the following functors are fully faithful and give bi-exact equivalences to the essential images:
    \begin{align*}
    \mathrm{Vect}((\fX,\cM_{\fX})_{\Prism})&\to \mathrm{Vect}_{\mathrm{kfl}}((\fX,\cM_{\fX})_{\Prism}), \\
    \mathrm{Vect}((\fX,\cM_{\fX})_{\Prism},\widetilde{\cO}_{\Prism})&\to \mathrm{Vect}_{\mathrm{kfl}}((\fX,\cM_{\fX})_{\Prism},\widetilde{\cO}_{\Prism}), \\
    \mathrm{Vect}^{\varphi}_{[a,b]}((\fX,\cM_{\fX})_{\Prism})&\to \mathrm{Vect}^{\varphi}_{\mathrm{kfl},[a,b]}((\fX,\cM_{\fX})_{\Prism}), \\
    \mathrm{Vect}^{\varphi}((\fX,\cM_{\fX})_{\Prism})&\to \mathrm{Vect}^{\varphi}_{\mathrm{kfl}}((\fX,\cM_{\fX})_{\Prism}).
    \end{align*}
\end{lem}

\begin{proof}
    Lemma \ref{category of crystals is unchanged} allows us to replace $(\fX,\cM_{\fX})_{\Prism}$ in the source categories with the strict site $(\fX,\cM_{\fX})^{\mathrm{str}}_{\Prism}$. Then, considering Remark \ref{pris crys as lim}, the functor $\iota$ in Definition \ref{def of iota}  induces the desired functors. The remaining assertions follow from Lemma \ref{fully faithfulness for kfl vect bdle}.
\end{proof}

To formulate ``Kummer quasi-syntomic descent for kfl prismatic crystals'' (Proposition \ref{kqsyn descent for kfl log pris crys}), we need to generalize the functoriality of kfl prismatic crystals to non-fs bases. Let $f\colon (\fY,\cM_{\fY})\to (\fX,\cM_{\fX})$ be a morphism from a bounded $p$-adic (not necessarily fs) saturated log formal scheme $(\fY,\cM_{\fY})$ to a bounded $p$-adic fs log formal scheme.

\begin{dfn}
    For a kfl prismatic crystal $\cE$ on $(\fX,\cM_{\fX})$, we say that $\cE$ is \emph{$f^{*}$-admissible} if, for any $(A,I,\cM_{A})\in (\fY,\cM_{\fY})_{\Prism}^{\mathrm{str}}$, the kfl vector bundle $\cE_{(A,I,\cN_{A})}$ is $\pi^{*}$-admissible, where the log prism $(A,I,\cN_{A})\in (\fX,\cM_{\fX})_{\Prism}^{\mathrm{str}}$ is the image of $(A,I,\cM_{A})$ via the functor $(\fX,\cM_{\fX})_{\Prism}\to (\fX,\cM_{\fX})_{\Prism}^{\mathrm{str}}$ in Lemma \ref{category of crystals is unchanged} and $\pi:(A,I,\cN_{A})\to (A,I,\cM_{A})$ is the natural map. Let $\mathrm{Vect}^{f^{*}\text{-}\mathrm{adm}}_{\mathrm{kfl}}((\fX,\cM_{\fX})_{\Prism})$ be the full subcategory of $\mathrm{Vect}_{\mathrm{kfl}}((\fX,\cM_{\fX})_{\Prism})$ consisting of $f^{*}$-admissible objects. We define an exact functor
    \[ 
    f^{*}\colon \mathrm{Vect}^{f^{*}\text{-}\mathrm{adm}}_{\mathrm{kfl}}((\fX,\cM_{\fX})_{\Prism})\to \mathrm{Vect}((\fY,\cM_{\fY})_{\Prism}^{\mathrm{str}})
    \] 
    as follows. Let $\cE\in \mathrm{Vect}^{f^{*}\text{-}\mathrm{adm}}_{\mathrm{kfl}}((\fX,\cM_{\fX})_{\Prism})$. We define $(f^{*}\cE)_{(A,I,\cM_{A})}\in \mathrm{Vect}(A)$ by
    \[ (f^{*}\cE)_{(A,I,\cM_{A})}\coloneqq\pi^{*}\cE_{(A,I,\cN_{A})}\in \mathrm{Vect}(A)
    \]
    for $(A,I,\cM_{A})\in (\fY,\cM_{\fY})_{\Prism}^{\mathrm{str}}$, and a map $(A,I,\cM_{A})\to (B,J,\cM_{B})$ in $(\fY,\cM_{\fY})_{\Prism}^{\mathrm{str}}$ induces an isomorphism 
    \[
    (f^{*}\cE)_{(A,I,\cM_{A})}\otimes_{A} B\isom (f^{*}\cE)_{(B,J,\cM_{B})}
    \]
    by Lemma \ref{functoriality of pro-kfl pullback of kfl vect bdles}. As a result, we obtain $f^{*}\cE\in \mathrm{Vect}((\fY,\cM_{\fY})_{\Prism}^{\mathrm{str}})$.

    In the same way, we can define the $f^{*}$-admissibility of an object of $\mathrm{Vect}_{\mathrm{kfl}}((\fX,\cM_{\fX})_{\Prism},\widetilde{\cO}_{\Prism})$, the full subcategory $\mathrm{Vect}^{f^{*}\text{-}\mathrm{adm}}_{\mathrm{kfl}}((\fX,\cM_{\fX})_{\Prism},\widetilde{\cO}_{\Prism})$ of $\mathrm{Vect}_{\mathrm{kfl}}((\fX,\cM_{\fX})_{\Prism},\widetilde{\cO}_{\Prism})$, and an exact functor
    \[
    f^{*}\colon \mathrm{Vect}^{f^{*}\text{-}\mathrm{adm}}_{\mathrm{kfl}}((\fX,\cM_{\fX})_{\Prism},\widetilde{\cO}_{\Prism})\to \mathrm{Vect}_{\mathrm{kfl}}((\fY,\cM_{\fY})_{\Prism}^{\mathrm{str}},\widetilde{\cO}_{\Prism}).
    \]

    Let $\mathrm{Vect}^{\varphi,f^{*}\text{-}\mathrm{adm}}_{\mathrm{kfl},[a,b]}((\fX,\cM_{\fX})_{\Prism})$ denote the full subcategory of $\mathrm{Vect}^{\varphi}_{\mathrm{kfl},[a,b]}((\fX,\cM_{\fX})_{\Prism})$ consisting of an object whose underlying kfl prismatic crystal is $f^{*}$-admissible. Let $(A,I,\cM_{A})\in (\fY,\cM_{\fY})_{\Prism}^{\mathrm{str}}$ and  $\mathrm{Vect}^{\varphi,\pi^{*}\text{-}\mathrm{adm}}_{\mathrm{kfl},[a,b]}(A,I,\cN_{A})$ be the full subcategory of $\mathrm{Vect}^{\varphi}_{\mathrm{kfl},[a,b]}(A,I,\cN_{A})$ consisting of an object whose underlying kfl vector bundle on $(\mathrm{Spf}(A),\cN_{A})$ is $\pi^{*}$-admissible. For an object
    \[
    (\cE,I^{b}\otimes_{A} \cE\to \phi_{A}^{*}\cE\to I^{a}\otimes_{A} \cE)
    \]
    of $\mathrm{Vect}^{\varphi,\pi^{*}\text{-}\mathrm{adm}}_{\mathrm{kfl},[a,b]}(A,I,\cN_{A})$, the kfl vector bundle $\phi_{A}^{*}\cE$ is also $\pi^{*}$-admissible and we have an isomorphism $\pi^{*}\phi_{A}^{*}\cE\cong \phi_{A}^{*}\pi^{*}\cE$ by Lemma \ref{functoriality of pro-kfl pullback of kfl vect bdles}. Hence, $\pi^{*}$ gives a functor
    \[ 
    \pi^{*}\colon \mathrm{Vect}^{\varphi,\pi^{*}\text{-}\mathrm{adm}}_{\mathrm{kfl},[a,b]}(A,I,\cN_{A})\to \mathrm{Vect}^{\varphi}_{[a,b]}(A,I).
    \]
    Taking a limit with respect to $(A,I,\cM_{A})\in (\fY,\cM_{\fY})_{\Prism}^{\mathrm{str}}$, we obtain an exact functor
    \[
    f^{*}\colon \mathrm{Vect}^{\varphi,f^{*}\text{-}\mathrm{adm}}_{\mathrm{kfl},[a,b]}((\fX,\cM_{\fX})_{\Prism})\to \mathrm{Vect}^{\varphi}_{[a,b]}((\fY,\cM_{\fY})^{\mathrm{str}}_{\Prism}).
    \]

    In the same way, we can define the full subcategory $\mathrm{Vect}^{\varphi,f^{*}\text{-}\mathrm{adm}}_{\mathrm{kfl}}((\fX,\cM_{\fX})_{\Prism})$ of the category $\mathrm{Vect}^{\varphi}_{\mathrm{kfl}}((\fX,\cM_{\fX})_{\Prism})$ and an exact functor
    \[
    f^{*}\colon \mathrm{Vect}^{\varphi,f^{*}\text{-}\mathrm{adm}}_{\mathrm{kfl}}((\fX,\cM_{\fX})_{\Prism})\to \mathrm{Vect}^{\varphi}((\fY,\cM_{\fY})^{\mathrm{str}}_{\Prism}).
    \]
\end{dfn}

\begin{lem}\label{functoriality of pro-kfl pullback of kfl pris crys}
    \begin{enumerate}
        \item Let $(\fZ,\cM_{\fZ})$ be a bounded $p$-adic saturated log formal scheme and $g:(\fZ,\cM_{\fZ})\to (\fY,\cM_{\fY})$ be a morphism. Let $\cE\in \mathrm{Vect}_{\mathrm{kfl}}((\fX,\cM_{\fX})_{\Prism})$. If $\cE$ is $f^{*}$-admissible, then $\cE$ is also $(fg)^{*}$-admissible and $(fg)^{*}\cE\cong g^{*}f^{*}\cE$.
        \item Let $(\fW,\cM_{\fW})$ be a bounded $p$-adic fs log formal scheme and $h:(\fX,\cM_{\fX})\to (\fW,\cM_{\fW})$ be a morphism. Let $\cE\in \mathrm{Vect}_{\mathrm{kfl}}((\fW,\cM_{\fW})_{\Prism})$. Then $\cE$ is $(hf)^{*}$-admissible if and only if $h^{*}\cE$ is $f^{*}$-admissible. Moreover, in this case, we have $(hf)^{*}\cE\cong f^{*}(h^{*}\cE)$.
    \end{enumerate}
\end{lem}

\begin{proof}
    This follows from Lemma \ref{functoriality of pro-kfl pullback of kfl vect bdles}. 
\end{proof}

\begin{lem}\label{pro-kfl pullback of kfl pris crys is well-def}
    Suppose that there exist a torsion-free fs monoid $P$, a chart $\alpha:P\to \cM_{\fX}$, and a chart $P_{\bQ_{\geq 0}}\to \cM_{\fY}$ such that the following diagram is commutative:
    \[
    \begin{tikzcd}
        P \ar[r] \ar[d] & f^{-1}\cM_{\fX} \ar[d] \\
        P_{\bQ_{\geq 0}} \ar[r] & \cM_{\fY}.
    \end{tikzcd}
    \]
    Here, $P_{\bQ_{\geq 0}}$ is the colimit of $P^{1/n}$ for integers $n\geq 1$. Then every kfl prismatic crystal $\cE$ on $(\fX,\cM_{\fX})$ is $f^{*}$-admissible.
\end{lem}

\begin{proof}
Let $(A,I,\cM_{A})\in (\fY,\cM_{\fY})^{\mathrm{str}}_{\Prism}$ and $g\colon (\mathrm{Spf}(A/I),\cM_{A/I})\to (\fY,\cM_{\fY})$ denote the structure morphism. By \cite[Lemma A.3]{ino25}, the monoid map 
\[
P_{\bQ_{\geq 0}}\to \Gamma(\fY,\cM_{\fY})\to (\mathrm{Spf}(A/I),\cM_{A/I})
\]
lifts uniquely to a monoid map $P_{\bQ_{\geq 0}}\to \Gamma(\mathrm{Spf}(A),\cM_{A})$. By the definition of $\cN_{A}$, we have the induced map $P\to \cN_{A}$ as in the following diagram:
\[
\begin{tikzcd}
    P_{\bQ_{\geq 0}} \ar[r] & \cM_{A} \ar[r] & \cM_{A/I} \\
    P \ar[r,dashed] \ar[u] \ar[rr, bend right, "\alpha"'] & \cN_{A} \ar[u] \ar[r] & (fg)^{*}\cM_{\fX} \ar[u].
\end{tikzcd}
\]
Therefore, Lemma \ref{pro-kfl pullback of kfl vect bdles is well-def} implies that every kfl vector bundle on $(\mathrm{Spf}(A),\cN_{A})$ is $\pi^{*}$-admissible, and so every kfl prismatic crystal on $(\fX,\cM_{\fX})$ is $f^{*}$-admissible.
\end{proof}

\begin{lem}\label{crys property for kfl pris crys}
    Let $\cE$ be an $f^{*}$-admissible prismatic crystal on $(\fX,\cM_{\fX})$. Consider log prisms $(A,I,\cM_{A})\in (\fX,\cM_{\fX})_{\Prism}^{\mathrm{str}}$ and  $(B,J,\cM_{B})\in (\fY,\cM_{\fY})_{\Prism}^{\mathrm{str}}$ with a map $p\colon (A,I,\cM_{A})\to (B,J,\cM_{B})$ in $(\fX,\cM_{\fX})_{\Prism}$. Then $\cE_{(A,I,\cM_{A})}$ is $p^{*}$-admissible, and we have a natural isomorphism
    \[
    p^{*}\cE_{(A,I,\cM_{A})}\isom (f^{*}\cE)_{(B,J,\cM_{B})}.
    \]
\end{lem}

\begin{proof}
    The log prism map $p$ factors as
    \[
    (A,I,\cM_{A})\stackrel{q}{\to} (B,J,\cN_{B})\stackrel{r}{\to} (B,J,\cM_{B}).
    \]
    Since $\cE_{(B,J,\cN_{B})}\cong q^{*}\cE_{(A,I,\cM_{A})}$ is $r^{*}$-admissible, $\cE_{(A,I,\cM_{A})}$ is $p^{*}$-admissible by Lemma \ref{functoriality of pro-kfl pullback of kfl vect bdles}, and so we have isomorphisms
    \[
    p^{*}\cE_{(A,I,\cM_{A})}\isom r^{*}q^{*}\cE_{(A,I,\cM_{A})}\isom r^{*}\cE_{(B,J,\cN_{B})}\isom (f^{*}\cE)_{(B,J,\cM_{B})}.
    \]
\end{proof}

The goal in this subsection is to prove Proposition \ref{kqsyn descent for kfl log pris crys}. To do this, we start with some technical lemmas.

\begin{lem}\label{pro-kfl cover of log prism}
    Let $(A,I,\cM_{A})$ be a log prism.
    \begin{enumerate}
        \item Suppose that we are given a chart $\widetilde{\alpha}\colon \bN^{r}\to \cM_{A}$ with $\delta_{\mathrm{log}}(\widetilde{\alpha}(m))=0$ for every $m\in \bN^{r}$. We let
        $(\mathrm{Spf}(A_{\infty,\widetilde{\alpha}}),\cM_{A_{\infty,\widetilde{\alpha}}})$ denote the fiber product 
        \[
        (\mathrm{Spf}(A),\cM_{A})\times_{(\bZ[\bN^{r}],\bN^{r})^{a}} (\bZ[\bQ_{\geq 0}^{r}],\bQ_{\geq 0}^{r})^{a}.
        \]
        Then $(A_{\infty,\widetilde{\alpha}},IA_{\infty,\widetilde{\alpha}},\cM_{A_{\infty,\widetilde{\alpha}}})$ has a unique log prism structure such that the natural map
        \[
        (A,I,\cM_{A})\to (A_{\infty,\widetilde{\alpha}},IA_{\infty,\widetilde{\alpha}},\cM_{A_{\infty,\widetilde{\alpha}}})
        \]
        is a log prism map.
        \item Suppose that, for each $1\leq i\leq n$, we are given a  charts $\widetilde{\alpha}_{i}\colon \bN^{r}\to \cM_{A}$ with $\delta_{\mathrm{log}}(\widetilde{\alpha}_{i}(m))=0$ for every $m\in \bN^{r}$ and a strict $(p,I)$-completely flat map $(\mathrm{Spf}(B_{i}),\cM_{B_{i}})\to (\mathrm{Spf}(A_{\infty,\widetilde{\alpha}_{i}}),\cM_{A_{\infty,\widetilde{\alpha}_{i}}})$. We let 
        \[
        (\mathrm{Spf}(A_{\infty,\widetilde{\alpha}_{1},\dots,\widetilde{\alpha}_{n}}),\cM_{A_{\infty,\widetilde{\alpha}_{1},\dots,\widetilde{\alpha}_{n}}}) \ (\text{resp.} \ (\mathrm{Spf}(B'),\cM_{B'}))
        \]
        denote the saturated fiber product of $(\mathrm{Spf}(A_{\infty,\widetilde{\alpha}_{i}}),\cM_{A_{\infty,\widetilde{\alpha}_{i}}})$ (resp. $(\mathrm{Spf}(B_{i}),\cM_{B_{i}})$) for $1\leq i\leq n$ over $(\mathrm{Spf}(A),\cM_{A})$. Then $(B',IB',\cM_{B'})$ has a unique log prism structure such that the natural maps
        \[ (B_{i},IB_{i},\cM_{B_{i}})\to (B',IB',\cM_{B'})
        \] 
        are log prism maps for every $1\leq i\leq n$.
    \end{enumerate}
\end{lem}

\begin{proof}\noindent
\begin{enumerate}
    \item We regard the prelog ring $(\bZ[\bN^{r}],\bN^{r})$ (resp. $(\bZ[\bQ_{\geq 0}^{r}],\bQ_{\geq 0}^{r})$) as a $\delta_{\mathrm{log}}$-ring by setting $\delta(m)=\delta_{\mathrm{log}}(m)=0$ for every $m\in \bN^{r}$ (resp. $m\in \bQ_{\geq 0}^{r}$). Then a map of prelog rings $(\bZ[\bN^{r}],\bN^{r})\to (A,\bN^{r})$ is compatible with $\delta_{\mathrm{log}}$-structures by the hypothesis on $\widetilde{\alpha}$, and the prelog ring $(A_{\infty,\widetilde{\alpha}},\bQ_{\geq 0}^{r})=((A,\bN^{r})\otimes_{(\bZ[\bN^{r}],\bN^{r})} (\bZ[\bQ_{\geq 0}^{r}],\bQ_{\geq 0}^{r}))^{\wedge}$ is equipped with a unique $\delta_{\mathrm{log}}$-structure by \cite[Remark 2.6 and Lemma 2.9]{kos22}. Since $A\to A_{\infty,\widetilde{\alpha}}$ is $(p,I)$-completely flat cover, \cite[Lemma 3.7(3)]{bs22} implies that $(A_{\infty,\widetilde{\alpha}},IA_{\infty,\widetilde{\alpha}},\bQ_{\geq 0}^{r})$ is a prelog prism and $(A_{\infty,\widetilde{\alpha}},IA_{\infty,\widetilde{\alpha}},\bQ_{\geq 0}^{r})^{a}=(A_{\infty,\widetilde{\alpha}},IA_{\infty,\widetilde{\alpha}},\cM_{ A_{\infty,\widetilde{\alpha}}})$ belongs to $(\fX_{\infty,\alpha},\cM_{\fX_{\infty,\alpha}})^{\mathrm{str}}_{\Prism}$. 

    \item For a subset $I\subset \{1,2,\dots,n\}$, we let $(\mathrm{Spf}(B_{I}),\cM_{B_{I}})$ denote the saturated fiber product of $(\mathrm{Spf}(B_{i}),\cM_{B_{i}})$ for $i\in I$ over $(\mathrm{Spf}(A),\cM_{A})$. When $n\geq 3$, we have the following Cartesian diagram in the category of saturated log formal schemes:
    \[
    \begin{tikzcd}
        (\mathrm{Spf}(B'),\cM_{B'}) \ar[r] \ar[d] & (\mathrm{Spf}(B_{\{1,2\}}),\cM_{B_{\{1,2\}}}) \ar[d] \\
        (\mathrm{Spf}(B_{\{2,\dots,n\}}),\cM_{B_{\{2,\dots,n\}}}) \ar[r] & (\mathrm{Spf}(B),\cM_{B}).
    \end{tikzcd}
    \]
    Every morphism in this diagram is strict flat by Lemma \ref{sat prod of two pro-kfl cover}. Hence, by \cite[Lemma 3.7 (3)]{bs22}, we can reduce the problem to the $n=1$ case by induction.

    We let $(\mathrm{Spf}(C_{i}),\cM_{C_{i}})$ denote the saturated fiber product of $(\mathrm{Spf}(B_{i}),\cM_{B_{i}})$ and $(\mathrm{Spf}(A_{\infty,\widetilde{\alpha}_{3-i}}),\cM_{A_{\infty,\widetilde{\alpha}_{3-i}}})$ over $(\mathrm{Spf}(A),\cM_{A})$ for $i=1,2$. Then we have the following diagram in which every square is a Cartesian diagram in the category of saturated log formal schemes:
    \[
    \begin{tikzcd}
    (\mathrm{Spf}(B'),\cM_{B'}) \ar[d] \ar[r] & (\mathrm{Spf}(C_{1}),\cM_{C_{1}}) \ar[d] \ar[r] & (\mathrm{Spf}(B_{1}),\cM_{B_{1}}) \ar[d] \\
    (\mathrm{Spf}(C_{2}),\cM_{C_{2}}) \ar[r] \ar[d] &(\mathrm{Spf}(A_{\infty,\widetilde{\alpha}_{1},\widetilde{\alpha}_{2}}),\cM_{A_{\infty,\widetilde{\alpha}_{1},\widetilde{\alpha}_{2}}}) \ar[d] \ar[r] & (\mathrm{Spf}(A_{\infty,\widetilde{\alpha}_{1}}),\cM_{A_{\infty,\widetilde{\alpha}_{1}}}) \\
    (\mathrm{Spf}(B_{2}),\cM_{B_{2}}) \ar[r] & (\mathrm{Spf}(A_{\infty,\widetilde{\alpha}_{2}}),\cM_{A_{\infty,\widetilde{\alpha}_{2}}}).
    \end{tikzcd}
    \]
    Every morphism in this diagram is strict flat by Lemma \ref{sat prod of two pro-kfl cover}. Therefore, by the same argument as the previous paragraph, it suffices to prove that the triple $(A_{\infty,\widetilde{\alpha}_{1},\widetilde{\alpha}_{2}},IA_{\infty,\widetilde{\alpha}_{1},\widetilde{\alpha}_{2}},\cM_{A_{\infty,\widetilde{\alpha}_{1},\widetilde{\alpha}_{2}}})$ has a unique log prism structure such that the natural maps
    \[ (A_{\infty,\widetilde{\alpha}_{i}},IA_{\infty,\widetilde{\alpha}_{i}},\cM_{A_{\infty,\widetilde{\alpha}_{i}}})\to (A_{\infty,\widetilde{\alpha}_{1},\widetilde{\alpha}_{2}},IA_{\infty,\widetilde{\alpha}_{1},\widetilde{\alpha}_{2}},\cM_{A_{\infty,\widetilde{\alpha}_{1},\widetilde{\alpha}_{2}}})
    \]
    are log prism maps for $i=1,2$.
        
    \cite[Lemma 2.13]{kos22} allows us to work \'{e}tale locally on $\mathrm{Spf}(A)$. By Lemma \ref{free monoid to free monoid} (cf. the proof of Lemma \ref{sat prod of two pro-kfl cover}), we may assume that there exists an integer $s$ with $1\leq s\leq r$ such that we can write
    \begin{align*}
        \widetilde{\alpha}_{1}(e_{i})&=u_{i}\widetilde{\alpha}_{2}(e_{i})\ \ (1\leq i\leq s) \\
        \widetilde{\alpha}_{1}(e_{i})&=v_{i}\ \ (s<i\leq r) \\
        \widetilde{\alpha}_{2}(e_{i})&=w_{i}\ \ (s<i\leq r)
    \end{align*}
        for $u_{i},v_{i},w_{i}\in A^{\times}$. Consider a ring map $f\colon \bZ[x_{1},\dots,x_{r}]\to A_{\infty,\widetilde{\alpha}_{2}}$ mapping $x_{i}$ to $u_{i}$ or $v_{i}$ depending on whether $i\leq s$ or not. We set
        \[ D\coloneqq(A_{\infty,\widetilde{\alpha}_{2}}\otimes_{\bZ[x_{1},\dots,x_{r}]} \bZ[x_{1}^{\bQ_{\geq 0}},\dots,x_{r}^{\bQ_{\geq 0}}])^{\wedge}_{(p,I)}.
        \]
        Let $\cM_{D}$ be the inverse image log structure of $\cM_{A_{\infty,\widetilde{\alpha}_{2}}}$. We consider a monoid map $\beta\colon \bQ^{r}\to \Gamma(\mathrm{Spf}(D),\cM_{D})$ mapping $(1/n)e_{i}$ to $x_{i}^{1/n}\widetilde{\alpha}_{2}((1/n)e_{i})$ or $x_{i}^{1/n}$ depending on whether $i\leq s$ or not. Here, $\widetilde{\alpha}_{2}((1/n)e_{i})$ is the image of $(1/n)e_{i}$ by the map
        \[
        \bQ_{\geq 0}^{r}\to \Gamma(\mathrm{Spf}(A_{\infty,\widetilde{\alpha}_{2}}),\cM_{A_{\infty,\widetilde{\alpha}_{2}}})\to \Gamma(\mathrm{Spf}(D),\cM_{D}).
        \]
        The map $\beta$ induces a morphism 
        \[
        (\mathrm{Spf}(D),\cM_{D})\to (\mathrm{Spf}(A_{\infty,\widetilde{\alpha}_{1}}),\cM_{A_{\infty,\widetilde{\alpha}_{1}}}).
        \]
        Then the morphisms
        \[
        (\mathrm{Spf}(D),\cM_{D})\to (\mathrm{Spf}(A_{\infty,\widetilde{\alpha}_{i}}),\cM_{A_{\infty,\widetilde{\alpha}_{i}}})\ \ \ (i=1,2)
        \]
        induce a morphism
        \[
        (\mathrm{Spf}(D),\cM_{D})\to (\mathrm{Spf}(A_{\infty,\widetilde{\alpha}_{1},\widetilde{\alpha}_{2}}),\cM_{A_{\infty,\widetilde{\alpha}_{1},\widetilde{\alpha}_{2}}}).
        \]
        We can show directly that this morphism is an isomorphism by checking that $(\mathrm{Spf}(D),\cM_{D})$ satisfies the universal property of the saturated fiber product.
        
        To complete the proof, it suffices to prove that $(\mathrm{Spf}(D),\cM_{D})$ admits a unique $\delta_{\mathrm{log}}$-structure such that projection maps
        \[
        (\mathrm{Spf}(D),\cM_{D})\to (A_{\infty,\widetilde{\alpha}_{i}},\cM_{A_{\infty,\widetilde{\alpha}_{i}}})\ \ \ (i=1,2)
        \]
        are compatible with $\delta_{\mathrm{log}}$-structures by \cite[Lemma 3.7 (3)]{bs22}. Since $\delta_{\mathrm{log}}(\widetilde{\alpha}_{1}(e_{i}))=\delta_{\mathrm{log}}(\widetilde{\alpha}_{2}(e_{i}))=0$, we have $\delta_{\mathrm{log}}(u_{i})=\delta_{\mathrm{log}}(v_{i})=\delta_{\mathrm{log}}(w_{i})=0$ for every $i$. If we regard $\bZ[x_{1},\dots,x_{r}]$ (resp. $\bZ[x_{1}^{\bQ_{\geq 0}},\dots,x_{r}^{\bQ_{\geq 0}}]$) as a $\delta$-ring by defining $\delta(x_{i})=0$ for every $i$ (resp. $\delta(x_{i}^{1/n})=0$ for every $i$ and $n\geq 1$), the ring map $f\colon \bZ[x_{1},\dots,x_{r}]\to A_{\infty,\widetilde{\alpha}_{2}}$ is compatible with $\delta$-structures, and so $D$ is endowed with the natural $\delta$-structure. We equip $(\mathrm{Spf}(D),\cM_{D})$ with the $\delta_{\mathrm{log}}$-structure such that the strict morphism 
        \[
        (\mathrm{Spf}(D),\cM_{D})\to (\mathrm{Spf}(A_{\infty,\widetilde{\alpha}_{2}}),\cM_{A_{\infty,\widetilde{\alpha}_{2}}})
        \]
        is compatible with $\delta_{\mathrm{log}}$-structures. Then the projection map
        \[
        (\mathrm{Spf}(D),\cM_{D})\to (\mathrm{Spf}(A_{\infty,\widetilde{\alpha}_{1}}),\cM_{A_{\infty,\widetilde{\alpha}_{1}}})
        \]
        is also compatible with $\delta_{\mathrm{log}}$-structures by construction. The uniqueness of the $\delta_{\mathrm{log}}$-structure on $(\mathrm{Spf}(D),\cM_{D})$ follows from \cite[Lemma 5.16]{ino25}.
    \end{enumerate}
\end{proof}

Being motivated by this lemma, we define the following.

\begin{dfn}
    Let $(\fX,\cM_{\fX})$ be a bounded $p$-adic fs log formal scheme and $\alpha\colon \bN^{r}\to \cM_{\fX}$ be a chart. We say that $\alpha$ is \emph{prismatically liftable} if, for an arbitrary log prism $(A,I,\cM_{A})$ in $(\fX,\cM_{\fX})^{\mathrm{str}}_{\Prism}$, there exists a strict flat cover $(A,I,\cM_{A})\to (B,J,\cM_{B})$ such that a map 
    \[
    \bN^{r}\to \Gamma(\fX,\cM_{\fX})\to \Gamma(\Spf(A/I),\cM_{A/I})\to \Gamma(\Spf(B/J),\cM_{B/J})
    \]
    admits a lift $\widetilde{\alpha}:\bN^{r}\to \Gamma(\Spf(B),\cM_{B})$ with $\delta_{\mathrm{log}}(\widetilde{\alpha}(m))=0$ for every $m\in \bN^{r}$.
\end{dfn}

\begin{lem}\label{ex of pris liftable chart}\noindent
    \begin{enumerate}
        \item Let $(\fX,\cM_{\fX})$ be a small affine log formal scheme over $\cO_{K}$ and $\alpha:\bN^{r}\to \cM_{\fX}$ be a tautological chart. Then $\alpha$ is prismatically liftable. 
        \item Let $f\colon (\fX,\cM_{\fX})\to (\fY,\cM_{\fY})$ be a strict map of $p$-adic fs log formal schemes. Suppose that we are given a prismatically liftable chart $\alpha\colon \bN^{r}\to \cM_{\fY}$ of $(\fY,\cM_{\fY})$. Then a chart
        \[
        \bN^{r}\stackrel{\alpha}{\to} f^{-1}\cM_{\fY}\to \cM_{\fX}
        \]
        is also prismatically liftable.
    \end{enumerate}
\end{lem} 

\begin{proof}\noindent
    \begin{enumerate}
        \item \cite[Lemma 5.12]{ino25} imply that, for every $(A,I,\cM_{A})$ belonging to $ (\fX,\cM_{\fX})^{\mathrm{str}}_{\Prism}$, there exist a flat cover $(A,I,\cM_{A})\to (B,J,\cM_{B})$ and a log prism map 
        \[
        (\fS_{R},(E),\bN^{r})^{a}=(\fS_{R},(E),\cM_{\fS_{R}})\to (B,J,\cM_{B}),
        \]
        which gives the desired lift $\widetilde{\alpha}:\bN^{r}\to \Gamma(\mathrm{Spf}(B),\cM_{B})$.
        \item This follows from the definition.
    \end{enumerate}
\end{proof}

\begin{prop}[``Kummer quasi-syntomic'' descent for kfl prismatic crystals]\label{kqsyn descent for kfl log pris crys}\noindent
Let $(\fX,\cM_{\fX})=(\mathrm{Spf}(R),\cM_{R})$ and $(\fY,\cM_{\fY})=(\mathrm{Spf}(S),\cM_{S})$ be bounded $p$-adic saturated log formal schemes, and $f\colon (\fY,\cM_{\fY})\to (\fX,\cM_{\fX})$ be a surjective map. Suppose that there exists a prismatically liftable chart $\alpha\colon \bN^{r}\to \cM_{\fX}$ such that $f$ admits a factorization
\[
(\fY,\cM_{\fY})\to (\fX_{\infty,\alpha},\cM_{\fX_{\infty,\alpha}})=(\mathrm{Spf}(R_{\infty,\alpha}),\cM_{R_{\infty,\alpha}})\to (\fX,\cM_{\fX}),
\]
where $(\fX_{\infty,\alpha},\cM_{\fX_{\infty,\alpha}})$ is one defined in Setting \ref{setting for pro-kfl descent} and the map $(\fY,\cM_{\fY})\to (\fX_{\infty,\alpha},\cM_{\fX_{\infty,\alpha}})$ is strict quasi-syntomic. We let $(\fY^{(\bullet)},\cM_{\fY^{(\bullet)}})$ denote the \v{C}ech nerve of $f$ in the category of $p$-adic saturated log formal schemes. Then there exist natural bi-exact equivalences:
\begin{align*}
\mathrm{Vect}_{\mathrm{kfl}}((\fX,\cM_{\fX})_{\Prism})&\isom \varprojlim_{\bullet\in \Delta} \mathrm{Vect}(\fY^{(\bullet)}_{\Prism}), \\
\mathrm{Vect}_{\mathrm{kfl}}((\fX,\cM_{\fX})_{\Prism},\widetilde{\cO}_{\Prism})&\isom \varprojlim_{\bullet\in \Delta} \mathrm{Vect}(\fY^{(\bullet)}_{\Prism},\widetilde{\cO}_{\Prism}), \\
\mathrm{Vect}^{\varphi}_{\mathrm{kfl}}((\fX,\cM_{\fX})_{\Prism})&\isom \varprojlim_{\bullet\in \Delta} \mathrm{Vect}^{\varphi}(\fY^{(\bullet)}_{\Prism}), \\
\mathrm{Vect}^{\varphi}_{\mathrm{kfl},[a,b]}((\fX,\cM_{\fX})_{\Prism})&\isom \varprojlim_{\bullet\in \Delta} \mathrm{Vect}^{\varphi}_{[a,b]}(\fY^{(\bullet)}_{\Prism}).    
\end{align*}
 
\end{prop}
\begin{proof}
    By Lemma \ref{divisible monoid and log pris site}, Lemma \ref{functoriality of pro-kfl pullback of kfl pris crys} (1), and Lemma \ref{pro-kfl pullback of kfl pris crys is well-def}, the functor $f^{*}$ induces functors in the statement. We use the notations in Lemma \ref{pro-kfl cover of log prism}. First, we shall prove the statement for $\mathrm{Vect}_{\mathrm{kfl}}((\fX,\cM_{\fX})_{\Prism})$. We construct an exact quasi-inverse functor. Let $p_{i}\colon \fY^{(1)}\to \fY$ be an $i$-th projection for $i=1,2$ and $(\cE_{\infty},p_{1}^{*}\cE_{\infty}\cong p_{2}^{*}\cE_{\infty})$ be an object of $\displaystyle \varprojlim\mathrm{Vect}(\fY^{(\bullet)}_{\Prism})$. We define $\cE\in \mathrm{Vect}_{\mathrm{kfl}}((\fX,\cM_{\fX})_{\Prism})$ associated with $(\cE_{\infty},p_{1}^{*}\cE_{\infty}\cong p_{2}^{*}\cE_{\infty})$ as follows. In order to define $\cE$, it suffices to define the evaluation of $\cE$ at a log prism $(A,I,\cM_{A})\in (\fX,\cM_{\fX})^{\mathrm{str}}_{\Prism}$ for which the map $\bN^{r} \stackrel{\alpha}{\to} \Gamma(\fX,\cM_{\fX})\to \Gamma(\mathrm{Spf}(A/I),\cM_{A/I})$ 
    lifts to a monoid map $\widetilde{\alpha}:\bN^{r}\to \Gamma(\mathrm{Spf}(A),\cM_{A})$ with $\delta_{\mathrm{log}}(\widetilde{\alpha}(m))=0$ for every $m\in \bN^{r}$ because $\alpha$ is prismatically liftable. Applying \cite[Proposition 7.11]{bs22} to a log prism $(A_{\infty,\widetilde{\alpha}},IA_{\infty,\widetilde{\alpha}},\cM_{A_{\infty,\widetilde{\alpha}}})\in (\fX_{\infty,\alpha},\cM_{\fX_{\infty,\alpha}})_{\Prism}^{\mathrm{str}}$ constructed in Lemma \ref{pro-kfl cover of log prism} and a quasi-syntomic map 
    \[ A_{\infty,\widetilde{\alpha}}/IA_{\infty,\widetilde{\alpha}}\to A_{\infty,\widetilde{\alpha}}/IA_{\infty,\widetilde{\alpha}}\otimes_{R_{\infty,\alpha}} S \ (\cong A/I\otimes_{R} S),
    \]
    we obtain a log prism $(B,IB,\cM_{B})\in (\fY,\cM_{\fY})_{\Prism}^{\mathrm{str}}$ and a log prism map 
    \[
    (A,I,\cM_{A})\to (A_{\infty,\widetilde{\alpha}},IA_{\infty,\widetilde{\alpha}},\cM_{A_{\infty,\widetilde{\alpha}}})\to (B,IB,\cM_{B})
    \]
    in $(\fX,\cM_{\fX})_{\Prism}^{\mathrm{str}}$
    such that the map of log formal schemes $(\mathrm{Spf}(B),\cM_{B})\to (\mathrm{Spf}(A),\cM_{A})$ is surjective. We let $(\mathrm{Spf}(B^{(n)}),\cM_{B^{(n)}})$ denote the $(n+1)$-fold saturated fiber product of $(\mathrm{Spf}(B),\cM_{B})$ over $(\mathrm{Spf}(A),\cM_{A})$. Then Lemma \ref{pro-kfl cover of log prism} (2) implies that $(B^{(n)},IB^{(n)},\cM_{B^{(n)}})$ is a log prism belonging to $(\fY^{(n)},\cM_{\fY^{(n)}})^{\mathrm{str}}_{\Prism}$. We have a vector bundle $(\cE_{\infty})_{B}$ with an isomorphism
    \[
        q_{1}^{*}((\cE_{\infty})_{B})\cong (p_{1}^{*}\cE_{\infty})_{B^{(1)}}\cong (p_{2}^{*}\cE_{\infty})_{B^{(1)}}\cong q_{2}^{*}((\cE_{\infty})_{B}),
    \]
    which satisfies the cocycle condition in $\mathrm{Vect}(B^{(2)})$. Here, $q_{i}\colon \mathrm{Spf}(B^{(1)})\to \mathrm{Spf}(B)$ is an $i$-th projection for $i=1,2$. This object descents to a kfl vector bundle on $\mathrm{Vect}_{\mathrm{kfl}}(\mathrm{Spf}(A),\cM_{A})$ by Proposition \ref{pro-kfl descent for kfl vect bdle}. The resulting kfl vector bundle is denoted by $\cE_{A}$.  The formation of $\cE_{A}$ is independent of the choice of $\widetilde{\alpha}$ and $B$. To check this, take another $\widetilde{\alpha}'$ and $(B',IB',\cM_{B'})\in (\fY,\cM_{\fY})_{\Prism}^{\mathrm{str}}$ satisfying the same condition as $\widetilde{\alpha}$ and $(B,IB,\cM_{B})$. Let $(\mathrm{Spf}(B''),\cM_{B''})$ be the saturated fiber product of $(\mathrm{Spf}(B),\cM_{B})$ and $(\mathrm{Spf}(B'),\cM_{B'})$ over $(\mathrm{Spf}(A),\cM_{A})$, and $q\colon (\mathrm{Spf}(B''),\cM_{B''})\to (\mathrm{Spf}(B),\cM_{B})$ and $q'\colon (\mathrm{Spf}(B''),\cM_{B''})\to (\mathrm{Spf}(B'),\cM_{B'})$ (which are strict flat covers) be natural projection maps. Then, by Lemma \ref{pro-kfl cover of log prism}, $(B'',IB'',\cM_{B''})$ is a log prism belonging to $(\fY^{(1)},\cM_{\fY^{(1)}})$, and the given isomorphism $p_{1}^{*}\cE_{\infty}\cong p_{2}^{*}\cE_{\infty}$ induces an isomorphism
    \[
        q^{*}((\cE_{\infty})_{B})\cong (p_{1}^{*}\cE_{\infty})_{B''}\cong (p_{2}^{*}\cE_{\infty})_{B''}\cong q'^{*}((\cE_{\infty})_{B'}),
    \]
    which gives the desired independence. Therefore, the family of $\cE_{A}$ forms a kfl prismatic crystal $\cE$ on $(\fX,\cM_{\fX})$, and the functor sending $(\cE_{\infty},p_{1}^{*}\cE_{\infty}\cong p_{2}^{*}\cE_{\infty})$ to $\cE$ is an exact quasi-inverse by Proposition \ref{pro-kfl descent for kfl vect bdle}. 

    In order to apply the same argument to the remaining cases, it suffices to prove that the following functors give bi-exact equivalences:
    \begin{align*}
    \mathrm{Vect}^{\varphi}_{\mathrm{kfl},[a,b]}(A,I,\cM_{A})&\to \varprojlim_{\bullet\in \Delta} \mathrm{Vect}^{\varphi}_{[a,b]}(B^{(\bullet)},IB^{(\bullet)}), \\
    \mathrm{Vect}^{\varphi}_{\mathrm{kfl}}(A,I,\cM_{A})&\to \varprojlim_{\bullet\in \Delta} \mathrm{Vect}^{\varphi}(B^{(\bullet)},IB^{(\bullet)}). \end{align*}
    The former equivalence follows from Lemma \ref{pro-kfl descent for kfl vect bdle}, and the latter one is reduced to the former one by definition.
\end{proof}

\begin{dfn}\label{def of log qrsp site}
     Let $(\fX,\cM_{\fX})$ be a bounded $p$-adic fs log formal scheme. Suppose that $(\fX,\cM_{\fX})$ admits a free chart \'{e}tale locally and that $\fX$ is quasi-syntomic. We define $(\fX,\cM_{\fX})'_{\mathrm{kqsyn}}$ as the category of a bounded $p$-adic saturated log formal scheme $(\fY,\cM_{\fY})$ over $(\fX,\cM_{\fX})$ such that the structure morphism $(\fY,\cM_{\fY})\to (\fX,\cM_{\fX})$ admits a factorization
    \[
    (\fY,\cM_{\fY})\to (\fU_{\infty,\alpha},\cM_{\fU_{\infty,\alpha}})\to (\fU,\cM_{\fU})\to (\fX,\cM_{\fX}),
    \] 
    where the morphism $(\fU,\cM_{\fU})\to (\fX,\cM_{\fX})$ is strict \'{e}tale, the map $\alpha\colon \bN^{r}\to \cM_{\fU}$ is a chart, and the morphism $(\fY,\cM_{\fY})\to (\fU_{\infty,\alpha},\cM_{\fU_{\infty,\alpha}})$ is strict quasi-syntomic. The category $(\fX,\cM_{\fX})'_{\mathrm{kqsyn}}$ has finite products by Lemma \ref{sat prod of two pro-kfl cover} (2).

    Let $(\fX,\cM_{\fX})'_{\mathrm{qrsp}}$ denote the full subcategory of $(\fX,\cM_{\fX})'_{\mathrm{kqsyn}}$ consisting of $(\mathrm{Spf}(S),\cM_{S})$ with $S$ being qrsp. 
\end{dfn}

\begin{cor}\label{kfl pris crys and kqsyn site}
    Let $(\fX,\cM_{\fX})$ be a $p$-adic fs log formal scheme. Suppose that $(\fX,\cM_{\fX})$ admits a prismatically liftable chart \'{e}tale locally and that $\fX$ is quasi-syntomic. Then there exists a natural bi-exact equivalence
    \begin{align*}
        \mathrm{Vect}_{\mathrm{kfl}}((\fX,\cM_{\fX})_{\Prism})&\isom  \varprojlim_{(\fY,\cM_{\fY})\in (\fX,\cM_{\fX})'_{\mathrm{kqsyn}}} \mathrm{Vect}(\fY_{\Prism}) \\
        &\isom  \varprojlim_{(S,\cM_{S})\in (\fX,\cM_{\fX})'_{\mathrm{qrsp}}} \mathrm{Vect}(S_{\Prism}).
    \end{align*}
    The analogous assertion also holds for $\mathrm{Vect}_{\mathrm{kfl}}((\fX,\cM_{\fX})_{\Prism},\widetilde{\cO}_{\Prism})$, $ \mathrm{Vect}^{\varphi}_{\mathrm{kfl}}((\fX,\cM_{\fX})_{\Prism})$, and $\mathrm{Vect}^{\varphi}_{\mathrm{kfl},[a,b]}((\fX,\cM_{\fX})_{\Prism})$.
\end{cor}

\begin{proof}
    For any $(f\colon (\fY,\cM_{\fY})\to (\fX,\cM_{\fX}))\in (\fX,\cM_{\fX})'_{\mathrm{kqsyn}}$, every kfl prismatic crystal on $(\fX,\cM_{\fX})$ is $f^{*}$-admissible, and the functors $f^{*}$ induce the functor in the statement by Lemma \ref{divisible monoid and log pris site}, Lemma \ref{functoriality of pro-kfl pullback of kfl pris crys}, and Lemma \ref{pro-kfl pullback of kfl crys crys is well-def}.
    
    We shall prove that this functor gives a bi-exact equivalence. By working \'{e}tale locally on $\fX$, we may assume that there exists a prismatically liftable chart $\alpha\colon \bN^{r}\to \cM_{\fX}$. By "Kummer quasi-syntomic descent for log prismatic crystals" (Proposition \ref{kqsyn descent for kfl log pris crys}), it is enough to show that the natural functor
    \[
    \varprojlim_{(\fY,\cM_{\fY})\in (\fX,\cM_{\fX})'_{\mathrm{kqsyn}}} \mathrm{Vect}(\fY_{\Prism})\to \varprojlim_{\bullet\in \Delta} \mathrm{Vect}((\fX^{(\bullet)}_{\infty,\alpha})_{\Prism})
    \]
    is a bi-exact equivalence. For any $(f\colon (\fY,\cM_{\fY})\to (\fX,\cM_{\fX}))\in (\fX,\cM_{\fX})'_{\mathrm{kqsyn}}$, the projection morphism from the saturated fiber product
    \[ (\fX_{\infty,\alpha},\cM_{\fX_{\infty,\alpha}})\times_{(\fX,\cM_{\fX})}^{\mathrm{sat}} (\fY,\cM_{\fY})\to (\fY,\cM_{\fY})
    \]
    is a quasi-syntomic cover by Lemma \ref{four point lemma in nonfs case} and Lemma \ref{sat prod of two pro-kfl cover} (2). Therefore, the claim follows from quasi-syntomic descent for (non-log) prismatic crystals (\cite[Proposition 7.11]{bs22}).
\end{proof}

\subsection{Kfl crystalline crystals}

We introduce the notion of the kfl version of crystalline crystals. Let us begin with review on fundamental properties of crystals on absolute log crystalline sites.

\begin{dfn}[Crystalline crystals on log schemes]\label{def of F-isoc} \noindent

Let $\mathrm{Crys}((X,\cM_{X})_{\mathrm{crys}})$ be the category of crystals of quasi-coherent sheaves of finite type on the log crystalline site $(X,\cM_{X})_{\mathrm{crys}}$. An object of $\mathrm{Crys}((X,\cM_{X})_{\mathrm{crys}})$ is called a \emph{crystalline crystal} on $(X,\cM_{X})$. The Frobenius morphism $F\colon (X,\cM_{X})\to (X,\cM_{X})$ induces a functor
\[
F^{*}\colon \mathrm{Crys}((X,\cM_{X})_{\mathrm{crys}})\to \mathrm{Crys}((X,\cM_{X})_{\mathrm{crys}}).
\]
Let $\mathrm{Vect}((X,\cM_{X})_{\mathrm{crys}})$ (resp. $\mathrm{Coh}((X,\cM_{X})_{\mathrm{crys}})$) denote the full subcategory of the category $\mathrm{Coh}((X,\cM_{X})_{\mathrm{crys}})$ consisting of crystals of vector bundles (resp. coherent sheaves) on $(X,\cM_{X})_{\mathrm{crys}}$. An object of $\mathrm{Vect}((X,\cM_{X})_{\mathrm{crys}})$ (resp. $\mathrm{Coh}((X,\cM_{X})_{\mathrm{crys}})$) is called a \emph{locally free crystalline crystal} (resp. \emph{coherent crystalline crystal}) on $(X,\cM_{X})$. 

Let $\mathrm{Vect}^{\varphi}((X,\cM_{X})_{\mathrm{crys}})$ be the category of pairs $(\cE,\varphi_{\cE})$ consisting of a locally free crystalline crystal $\cE$ on $(X,\cM_{X})$ and an isomorphism $\varphi_{\cE}\colon F^{*}\cE\isom \cE$ in the isogeny category $\mathrm{Crys}((X,\cM_{X})_{\mathrm{crys}})\otimes_{\bZ_{p}} \bQ_{p}$. Morphisms $(\cE_{1},\varphi_{\cE_{1}})\to (\cE_{2},\varphi_{\cE_{2}})$ are morphisms $\cE_{1}\to \cE_{2}$ that are compatible with $\varphi_{\cE_{1}}$ and $\varphi_{\cE_{2}}$. When no confusion occurs, we simply write $\cE$ for $(\cE,\varphi_{\cE})$. 

We also define similar categories by replacing $(X,\cM_{X})_{\mathrm{crys}}$ with a strict site $(X,\cM_{X})_{\mathrm{crys}}^{\mathrm{str}}$.
\end{dfn}

\begin{rem}\label{crystalline crystal as lim}
    For $\star\in \{\mathrm{str},\emptyset \}$, we have canonical bi-exact equivalences
    \begin{align*}
        \mathrm{Crys}((X,\cM_{X})^{\star}_{\mathrm{crys}})&\simeq \varprojlim_{(T,\cM_{T})\in (X,\cM_{X})^{\star}_{\mathrm{crys}}} \mathrm{FQCoh}(T), \\
        \mathrm{Vect}((X,\cM_{X})^{\star}_{\mathrm{crys}})&\simeq \varprojlim_{(T,\cM_{T})\in (X,\cM_{X})^{\star}_{\mathrm{crys}}} \mathrm{Vect}(T), \\
        \mathrm{Coh}((X,\cM_{X})^{\star}_{\mathrm{crys}})&\simeq \varprojlim_{(T,\cM_{T})\in (X,\cM_{X})^{\star}_{\mathrm{crys}}} \mathrm{Coh}(T).
    \end{align*}
\end{rem}

\begin{rem}\label{category of crystalline crystals is unchanged}
    When $(X,\cM_{X})$ is integral, it follows from the same argument as \cite[Remark 3.6]{ino25} that the restriction functors
    \begin{align*}
        \mathrm{Crys}((X,\cM_{X})_{\mathrm{crys}})&\to \mathrm{Crys}((X,\cM_{X})^{\mathrm{str}}_{\mathrm{crys}}) \\
        \mathrm{Vect}^{\star}((X, \cM_{X})_{\mathrm{crys}})&\to \mathrm{Vect}^{\star}((X, \cM_{X})^{\mathrm{str}}_{\mathrm{crys}}) \\
        \mathrm{Coh}((X, \cM_{X})_{\mathrm{crys}})&\to \mathrm{Coh}((X, \cM_{X})^{\mathrm{str}}_{\mathrm{crys}})
    \end{align*}
    give bi-exact equivalences for $\star\in \{\emptyset,\varphi \}$. 
\end{rem}

\begin{lem}\label{log crys site p-div case}
    Let $(X,\cM_{X})$ be a log scheme over $\bF_{p}$. Suppose that there is a chart $M\to \cM_{X}$ where $M$ is uniquely $p$-divisible (i.e. the monoid map $\times p\colon M\to M$ given by $m\mapsto m^{p}$ is an isomorphism). Then a natural functor
    \[
    (X,\cM_{X})^{\mathrm{str}}_{\mathrm{crys}}\to X_{\mathrm{crys}}
    \]
    sending $(U,T,\cM_{T})$ to $(U,T)$ gives an equivalence.

    In particular, natural functors
    \begin{align*}
        \mathrm{Crys}(X_{\mathrm{crys}})&\to \mathrm{Crys}((X,\cM_{X})_{\mathrm{crys}}) \\
        \mathrm{Vect}^{\star}(X_{\mathrm{crys}})&\to \mathrm{Vect}^{\star}((X,\cM_{X})_{\mathrm{crys}}) \\
        \mathrm{Coh}(X_{\mathrm{crys}})&\to \mathrm{Coh}((X,\cM_{X})_{\mathrm{crys}})
    \end{align*}
    give a bi-exact equivalence for $\star\in \{\emptyset,\varphi \}$.
\end{lem}

\begin{proof}
    This can be proved in the same way as \cite[Lemma 5.17]{ino25}.
\end{proof}

\begin{dfn}[Kfl crystalline crystals]
    Let $(X,\cM_{X})$ be an fs log scheme over $\bF_{p}$. We define exact categories $\mathrm{Crys}_{\mathrm{kfl}}((X,\cM_{X})_{\mathrm{crys}})$,  $\mathrm{Vect}_{\mathrm{kfl}}((X,\cM_{X})_{\mathrm{crys}})$, and $\mathrm{Coh}((X,\cM_{X})_{\mathrm{crys}})$ by
    \begin{align*}
        \mathrm{Crys}_{\mathrm{kfl}}((X,\cM_{X})_{\mathrm{crys}})&\coloneqq \varprojlim_{(T,\cM_{T})\in (X,\cM_{X})^{\mathrm{str}}_{\mathrm{crys}}} \mathrm{FQCoh}_{\mathrm{kfl}}(T,\cM_{T}), \\
        \mathrm{Vect}_{\mathrm{kfl}}((X,\cM_{X})_{\mathrm{crys}})&\coloneqq \varprojlim_{(T,\cM_{T})\in (X,\cM_{X})^{\mathrm{str}}_{\mathrm{crys}}} \mathrm{Vect}_{\mathrm{kfl}}(T,\cM_{T}), \\
        \mathrm{Coh}_{\mathrm{kfl}}((X,\cM_{X})_{\mathrm{crys}})&\coloneqq \varprojlim_{(T,\cM_{T})\in (X,\cM_{X})^{\mathrm{str}}_{\mathrm{crys}}} \mathrm{Coh}_{\mathrm{kfl}}(T,\cM_{T}).
    \end{align*}
The categories $\mathrm{Vect}_{\mathrm{kfl}}((X,\cM_{X})_{\mathrm{crys}})$ and $\mathrm{Coh}((X,\cM_{X})_{\mathrm{crys}})$ are full subcategories of the category $\mathrm{Crys}_{\mathrm{kfl}}((X,\cM_{X})_{\mathrm{crys}})$. An object of $\mathrm{Crys}_{\mathrm{kfl}}((X,\cM_{X})_{\mathrm{crys}})$ (resp. $\mathrm{Vect}_{\mathrm{kfl}}((X,\cM_{X})_{\mathrm{crys}})$) (resp. $\mathrm{Coh}_{\mathrm{kfl}}((X,\cM_{X})_{\mathrm{crys}})$) is called a \emph{kfl crystalline crystal} (resp. \emph{locally free kfl crystalline crystal})(resp. \emph{coherent kfl crystalline crystal}) on $(X,\cM_{X})$. For a morphism $f\colon (X,\cM_{X})\to (Y,\cM_{Y})$ of fs log schemes over $\bF_{p}$, we have a functor
    \[
    f^{*}\colon \mathrm{Crys}_{\mathrm{kfl}}((Y,\cM_{Y})_{\mathrm{crys}})\to \mathrm{Crys}_{\mathrm{kfl}}((X,\cM_{X})_{\mathrm{crys}})
    \]
induced from the morphism of sites $(Y,\cM_{Y})_{\mathrm{crys}}^{\mathrm{str}}\to (X,\cM_{X})_{\mathrm{crys}}^{\mathrm{str}}$. In particular, the Frobenius morphism $F\colon (X,\cM_{X})\to (X,\cM_{X})$ induces
    \[
    F^{*}\colon \mathrm{Coh}_{\mathrm{kfl}}((X,\cM_{X})_{\mathrm{crys}})\to \mathrm{Coh}_{\mathrm{kfl}}((X,\cM_{X})_{\mathrm{crys}}).
    \]

Let $\mathrm{Vect}^{\varphi}_{\mathrm{kfl}}((X,\cM_{X})_{\mathrm{crys}})$ be the category of pairs $(\cE,\varphi_{\cE})$ consisting of a locally free kfl crystalline crystal $\cE$ and an isomorphism $\varphi_{\cE}\colon F^{*}\cE\isom \cE$ in the isogeny category $\mathrm{Crys}_{\mathrm{kfl}}((X,\cM_{X})_{\mathrm{crys}})\otimes_{\bZ_{p}} \bQ_{p}$. Morphisms are ones of isocrystals that are compatible with Frobenius isomorphisms.
\end{dfn}

\begin{rem}
    We can equip the log crystalline site with kfl topology unlike the log prismatic site (cf. Remark \ref{gap in wz}), and the notion of kfl crystalline crystals coincides with one of coherent sheaves on this site. However, this approach is not considered in this work.
\end{rem}

\begin{lem}\label{fully faithfulness for kfl crys crys}
    Let $(X,\cM_{X})$ be an fs log scheme over $\bF_{p}$. Then, for $\star\in \{\emptyset,\varphi\}$, there are natural fully faithful functors:
\begin{align*}
    \mathrm{Crys}((X,\cM_{X})_{\mathrm{crys}})&\to \mathrm{Crys}_{\mathrm{kfl}}((X,\cM_{X})_{\mathrm{crys}}), \\
    \mathrm{Vect}^{\star}((X,\cM_{X})_{\mathrm{crys}})&\to \mathrm{Vect}^{\star}_{\mathrm{kfl}}((X,\cM_{X})_{\mathrm{crys}}), \\
    \mathrm{Coh}((X,\cM_{X})_{\mathrm{crys}})&\to \mathrm{Coh}_{\mathrm{kfl}}((X,\cM_{X})_{\mathrm{crys}}).
    \end{align*}
\end{lem}

\begin{proof}
    Due to Remark \ref{category of crystalline crystals is unchanged}, we may replace $(X,\cM_{X})_{\mathrm{crys}}$ in source categories with the strict site $(X,\cM_{X})_{\mathrm{crys}}^{\mathrm{str}}$. Then, by Remark \ref{crystalline crystal as lim}, the functor $\iota$ in Definition \ref{def of iota}  induces the desired functors. The remaining assertions follow from Lemma \ref{fully faithfulness for kfl vect bdle}.
\end{proof}

\begin{dfn}[Root morphisms, (cf. {\cite[\S 7.1]{lau18}})]\label{def of root map} \noindent

\begin{enumerate}
    \item Let $R\to S$ be a ring map. We say that $S$ is an \emph{extraction of a $p$-th root} (resp. \emph{extraction of a root}) if $S\cong R[a^{1/p}]\coloneqq R[x]/(x^{p}-a)$ for some $a\in R$ (resp. $S\cong R[a^{1/n}]\coloneqq R[x]/(x^{n}-a)$ for some $a\in R$ and integer $n\geq 1$). More generally, the map $R\to S$ is called a \emph{$p$-root extension} (resp. \emph{root extension}) if we can write $R\to S$ as the transfinite composition of extractions of a $p$-th root (resp. extractions of a root).
    \item Let $f\colon X\to Y$ be a morphism of schemes. We say that $f$ is a \emph{$p$-root morphism} (resp. \emph{root morphism}) if, for each $x\in X$, there exist an affine \'{e}tale neighborhood $U$ of $x$ and an affine \'{e}tale neighborhood $V$ of $f(x)$ such that $f$ restricts to a morphism $U\to V$ corresponding to the $p$-root extension ring map (resp. root extension ring map). A family of morphisms $\{U_{i}\to X\}_{i\in I}$ is called a \emph{$p$-root covering} (resp. \emph{root covering}) if $U_{i}\to X$ is a $p$-root morphism (resp. root morphism) for each $i\in I$ and the induced morphism $\displaystyle \bigsqcup_{i\in I} U_{i}\to X$ is a quasi-compact surjection. The Grothendieck topology on the category of schemes  defined by $p$-root coverings (resp. root coverings) is called \emph{$p$-root topology} (resp. \emph{root topology}).
\end{enumerate}
\end{dfn}

\begin{lem}\label{root descent for crys crys}
    Let $f\colon X\to Y$ be a quasi-compact root morphism of schemes over $\bF_{p}$ and $(U,T)$ be an affine PD-thickening in $Y_{\mathrm{crys}}$. Then there exist an affine PD-thickening $(U',T')\in X_{\mathrm{crys}}$ and a Cartesian flat map $(U',T')\to (U,T)$ of affine PD-thickenings such that the natural map $U'\to X\times_{Y} U$ is an \'{e}tale covering. In particular, for $\star\in \{\emptyset,\varphi\}$, the fibered categories over the category of schemes over $\bF_{p}$ given by 
    \[
    X\mapsto \mathrm{Crys}(X_{\mathrm{crys}}) \ (\text{resp.} \ \mathrm{Vect}^{\star}(X_{\mathrm{crys}})) \ (\text{resp.} \ \mathrm{Coh}(X_{\mathrm{crys}})) 
    \]
    are $2$-sheaves for root topology. 
\end{lem}

\begin{proof}
    Replacing $f$ with $X\times_{Y} U\to U$, we may assume that $Y=U$ (in particular, $Y$ is affine). Take an \'{e}tale covering $\{V_{i}\to X\}_{1\leq i\leq n}$ of $X$ and a family of \'{e}tale maps $\{W_{i}\to Y\}_{1\leq i\leq n}$ such that $V_{i}$ and $W_{i}$ are affine and $f$ restricts to a map $V_{i}\to W_{i}$ corresponding to a root extension ring map for each $1\leq i\leq n$. The \'{e}tale map $W_{i}\to U$ extends uniquely to an \'{e}tale map $T_{i}\to T$, and the closed immersion $W_{i}\hookrightarrow T_{i}$ has a unique PD-structure extending the PD-structure of $U\hookrightarrow T$. Since the map $V_{i}\to W_{i}$ corresponds to a root extension ring map, this extends to a Cartesian flat map of affine PD-thickenings
    \[
    (V_{i}\hookrightarrow T'_{i})\to (W_{i}\hookrightarrow T_{i})
    \]
    by the same argument as \cite[Lemma 7.5]{lau18}. Then, if we set
    \[
    (U',T')\coloneqq (\bigsqcup_{i=1}^{n} V_{i},\bigsqcup_{i=1}^{n} T'_{i}),
    \]
    the assertion holds.
\end{proof}

To formulate ``Kummer root descent for kfl crystalline crystals'' (Proposition \ref{kpr descent for kfl crystalline crystals}), we need to generalize the functoriality of kfl crystalline crystals to non-fs bases. Let $f\colon (Y,\cM_{Y})\to (X,\cM_{X})$ be a morphism from a (not necessarily fs) saturated log scheme $(Y,\cM_{Y})$ to an fs log formal scheme $(X,\cM_{X})$ over $\bF_{p}$.

\begin{dfn}
    For a kfl crystalline crystal $\cE$ on $(X,\cM_{X})$, we say that $\cE$ is \emph{$f^{*}$-admissible} if, for any $(T,\cM_{T})\in (Y,\cM_{Y})_{\mathrm{crys}}^{\mathrm{str}}$, the kfl coherent sheaf $\cE_{(T,\cN_{T})}$ is $\pi^{*}$-admissible, where the log PD-thickening $(T,\cN_{T})\in (X,\cM_{X})_{\mathrm{crys}}^{\mathrm{str}}$ is the image of $(T,\cM_{T})$ via the functor $(X,\cM_{X})_{\mathrm{crys}}\to (X,\cM_{X})_{\mathrm{crys}}^{\mathrm{str}}$ in Remark \ref{category of crystalline crystals is unchanged} and $\pi\colon (T,\cM_{T})\to (T,\cN_{T})$ is the natural map. Let $\mathrm{Crys}^{f^{*}\text{-adm}}_{\mathrm{kfl}}((X,\cM_{X})_{\mathrm{crys}})$ be the full subcategory of the category $\mathrm{Crys}_{\mathrm{kfl}}((X,\cM_{X})_{\mathrm{crys}})$ consisting of $f^{*}$-admissible objects. We define an exact functor 
    \[
    f^{*}\colon \mathrm{Crys}^{f^{*}\text{-adm}}_{\mathrm{kfl}}((X,\cM_{X})_{\mathrm{crys}})\to \mathrm{Crys}((Y,\cM_{Y})_{\mathrm{crys}}^{\mathrm{str}})
    \]
    as follows. Let $\cE\in \mathrm{Crys}^{f^{*}\text{-adm}}_{\mathrm{kfl}}((X,\cM_{X})_{\mathrm{crys}})$. For $(T,\cM_{T})\in (Y,\cM_{Y})_{\mathrm{crys}}^{\mathrm{str}}$, we define $(f^{*}\cE)_{(T,\cM_{T})}\in \mathrm{FQCoh}(T)$ by
    \[ (f^{*}\cE)_{(T,\cM_{T})}\coloneqq\pi^{*}\cE_{(T,\cN_{T})},
    \]
    and a map $g\colon (T',\cM_{T'})\to (T,\cM_{T})$ in $(Y,\cM_{Y})_{\mathrm{crys}}^{\mathrm{str}}$ induces an isomorphism
    \[
    g^{*}((f^{*}\cE)_{(T,\cM_{T})})\cong (f^{*}\cE)_{(T',\cM_{T'})}
    \]
    by Lemma \ref{functoriality of pro-kfl pullback of kfl vect bdles}. As a result, we obtain $f^{*}\cE\in \mathrm{Crys}((Y,\cM_{Y})_{\mathrm{crys}}^{\mathrm{str}})$.

    We define categories
    \begin{align*}
        \mathrm{Vect}^{f^{*}\text{-adm}}_{\mathrm{kfl}}((X,\cM_{X})_{\mathrm{crys}})&\coloneqq \mathrm{Vect}_{\mathrm{kfl}}((X,\cM_{X})_{\mathrm{crys}})\cap \mathrm{Crys}^{f^{*}\text{-adm}}_{\mathrm{kfl}}((X,\cM_{X})_{\mathrm{crys}}), \\
        \mathrm{Coh}^{f^{*}\text{-adm}}_{\mathrm{kfl}}((X,\cM_{X})_{\mathrm{crys}})&\coloneqq \mathrm{Coh}_{\mathrm{kfl}}((X,\cM_{X})_{\mathrm{crys}})\cap \mathrm{Crys}^{f^{*}\text{-adm}}_{\mathrm{kfl}}((X,\cM_{X})_{\mathrm{crys}}),
    \end{align*}
    and the functor $f^{*}$ defined in the previous paragraph restricts to functors
    \begin{align*}
        f^{*}\colon \mathrm{Vect}^{f^{*}\text{-adm}}_{\mathrm{kfl}}((X,\cM_{X})_{\mathrm{crys}})\to \mathrm{Vect}((Y,\cM_{Y})_{\mathrm{crys}}^{\mathrm{str}}), \\
        f^{*}\colon \mathrm{Coh}^{f^{*}\text{-adm}}_{\mathrm{kfl}}((X,\cM_{X})_{\mathrm{crys}})\to \mathrm{Coh}((Y,\cM_{Y})_{\mathrm{crys}}^{\mathrm{str}}).
    \end{align*}

Let $\mathrm{Vect}^{\varphi,f^{*}\text{-adm}}_{\mathrm{kfl}}((X,\cM_{X})_{\mathrm{crys}})$ denote the full subcategory of $\mathrm{Vect}^{\varphi}_{\mathrm{kfl}}((X,\cM_{X})_{\mathrm{crys}})$ consisting of an object whose underlying locally free kfl crystalline crystal is $f^{*}$-admissible. For an object $(\cE,\varphi_{\cE})\in \mathrm{Vect}^{\varphi,f^{*}\text{-adm}}_{\mathrm{kfl}}((X,\cM_{X})_{\mathrm{crys}})$, a locally free kfl crystalline crystal $F^{*}\cE$ is $f^{*}$-admissible and there exists an isomorphism
\[
F^{*}f^{*}\cE[1/p]\cong f^{*}F^{*}\cE[1/p]\stackrel{f^{*}(\varphi_{\cE})}{\isom} f^{*}\cE[1/p]
\]
by Lemma \ref{functoriality of pro-kfl pullback of kfl crys crys} below. Hence, the functor $f^{*}$ gives a functor
\[
f^{*}\colon \mathrm{Vect}^{\varphi,f^{*}\text{-adm}}_{\mathrm{kfl}}((X,\cM_{X})_{\mathrm{crys}})\to \mathrm{Vect}^{\varphi}((Y,\cM_{Y})_{\mathrm{crys}}^{\mathrm{str}}).
\]
\end{dfn}

\begin{lem}\label{crys property for kfl cry crys}
    Let $\cE$ be an $f^{*}$-admissible crystalline crystal on $(X,\cM_{X})$. Consider $(T,\cM_{T})\in (X,\cM_{X})_{\mathrm{crys}}^{\mathrm{str}}$, $(T',\cM_{T'})\in (Y,\cM_{Y})_{\mathrm{crys}}^{\mathrm{str}}$, and a map of log PD-thickenings $p\colon (T',\cM_{T'})\to (T,\cM_{T})$ in $(X,\cM_{X})_{\mathrm{crys}}$. Then $\cE_{(T,\cM_{T})}$ is $p^{*}$-admissible, and we have a natural isomorphism
    \[
    p^{*}\cE_{(T,\cM_{T})}\isom (f^{*}\cE)_{(T',\cM_{T'})}.
    \]
\end{lem}

\begin{proof}
    This follows from the same argument as the proof of Lemma \ref{crys property for kfl pris crys}.
\end{proof}

\begin{lem}\label{functoriality of pro-kfl pullback of kfl crys crys}
The following statements are true.
    \begin{enumerate}
        \item Let $(Z,\cM_{Z})$ be a saturated log scheme and $g\colon (Z,\cM_{Z})\to (Y,\cM_{Y})$ be a morphism. Let $\cE\in \mathrm{Crys}_{\mathrm{kfl}}((X,\cM_{X})_{\mathrm{crys}})$. If $\cE$ is $f^{*}$-admissible, then $\cE$ is also $(fg)^{*}$-admissible and $(fg)^{*}\cE\cong g^{*}f^{*}\cE$.
        \item Let $(W,\cM_{W})$ be an fs log scheme over $\bF_{p}$ and $h\colon (X,\cM_{X})\to (W,\cM_{W})$ be a morphism. Let $\cE\in \mathrm{Crys}_{\mathrm{kfl}}((W,\cM_{W})_{\mathrm{crys}})$. Then $\cE$ is $(hf)^{*}$-admissible if and only if $h^{*}\cE$ is $f^{*}$-admissible. Moreover, in this case, we have $(hf)^{*}\cE\cong f^{*}(h^{*}\cE)$.
    \end{enumerate}
\end{lem}

\begin{proof}
    This follows from Lemma \ref{functoriality of pro-kfl pullback of kfl vect bdles}. 
\end{proof}

\begin{lem}\label{pro-kfl pullback of kfl crys crys is well-def}
    Suppose that there exist a torsion-free fs monoid $P$, a chart $\alpha\colon P\to \cM_{X}$, and a chart $P_{\bQ_{\geq 0}}\to \cM_{Y}$ such that the following diagram is commutative:
    \[
    \begin{tikzcd}
        P \ar[r] \ar[d] & f^{-1}\cM_{X} \ar[d] \\
        P_{\bQ_{\geq 0}} \ar[r] & \cM_{Y}.
    \end{tikzcd}
    \]
    Then every kfl crystalline crystal $\cE$ on $(X,\cM_{X})$ is $f^{*}$-admissible.
\end{lem}

\begin{proof}
Let $(U,T,\cM_{T})\in (Y,\cM_{Y})^{\mathrm{str}}_{\mathrm{crys}}$ be an affine log PD-thickening and $g\colon (U,\cM_{U})\to (Y,\cM_{Y})$ denote the structure morphism. By \cite[Lemma A.1]{ino25}, the natural map $\Gamma(T,\cM_{T})\to \Gamma(U,\cM_{U})$ is surjective. By \cite[Lemma A.3]{ino25}, the monoid map
\[
P_{\bQ_{\geq 0}}\to \Gamma(Y,\cM_{Y})\to \Gamma(U,\cM_{U})
\]
lifts uniquely to a monoid map $P_{\bQ_{\geq 0}}\to \Gamma(T,\cM_{T})$. By the definition of $\cN_{T}$, we have the induced map $P\to \cN_{T}$ as in the following diagram:
\[
\begin{tikzcd}
    P_{\bQ_{\geq 0}} \ar[r] & \cM_{T} \ar[r] & \cM_{U} \\
    P \ar[r,dashed] \ar[u] \ar[rr, bend right, "\alpha"'] & \cN_{T} \ar[u] \ar[r] & (fg)^{*}\cM_{X}. \ar[u] 
\end{tikzcd}
\]
Therefore, Lemma \ref{pro-kfl pullback of kfl vect bdles is well-def} implies that every kfl coherent sheaves on $(T,\cN_{T})$ is $\pi^{*}$-admissible, and so every kfl crystalline crystal on $(X,\cM_{X})$ is $f^{*}$-admissible.
\end{proof}

\begin{prop}[``Kummer root'' descent for kfl crystalline crystals]\label{kpr descent for kfl crystalline crystals}\noindent

Let $f\colon (Y,\cM_{Y})\to (X,\cM_{X})$ be a surjection of saturated log schemes over $\bF_{p}$. Suppose that we are given a chart $\alpha\colon \bN^{r}\to \cM_{X}$ such that $f$ admits a factorization
\[
(Y,\cM_{Y})\to (X_{\infty,\alpha},\cM_{X_{\infty,\alpha}})\to (X,\cM_{X}),
\]
where $(X_{\infty,\alpha},\cM_{X_{\infty,\alpha}})$ is the one defined as in Setting \ref{setting for pro-kfl descent} and $(Y,\cM_{Y})\to (X_{\infty,\alpha},\cM_{X_{\infty,\alpha}})$ is a strict root morphism. We let $(Y^{(\bullet)},\cM_{Y^{(\bullet)}})$ be the \v{C}ech nerve of $f$ in the category of saturated log schemes. Then, for $\star\in \{\emptyset,\varphi\}$, there exist natural bi-exact equivalences
\begin{align*}
\mathrm{Vect}^{\star}_{\mathrm{kfl}}((X,\cM_{X})_{\mathrm{crys}})&\isom \varprojlim_{\bullet\in \Delta} \mathrm{Vect}^{\star}(Y^{(\bullet)}_{\mathrm{crys}}) \\
\mathrm{Coh}_{\mathrm{kfl}}((X,\cM_{X})_{\mathrm{crys}})&\isom \varprojlim_{\bullet\in \Delta} \mathrm{Coh}(Y^{(\bullet)}_{\mathrm{crys}}).
\end{align*}
\end{prop}

\begin{proof}
    It suffices to prove the assertion for $\mathrm{Coh}((X,\cM_{X})_{\mathrm{crys}})$ because the remaining assertions can be proved in the same way. By Lemma \ref{pro-kfl pullback of kfl crys crys is well-def}, we have the pullback functor
    \[
    f^{*}\colon \mathrm{Coh}_{\mathrm{kfl}}((X,\cM_{X})_{\mathrm{crys}})\to \mathrm{Coh}((Y,\cM_{Y})_{\mathrm{crys}}^{\mathrm{str}}),
    \]
    which induces an exact functor
    \[
    \mathrm{Coh}_{\mathrm{kfl}}((X,\cM_{X})_{\mathrm{crys}})\to \varprojlim_{\bullet\in \Delta}\mathrm{Coh}(Y^{(\bullet)}_{\mathrm{crys}})
    \]
    by Lemma \ref{log crys site p-div case} and Lemma \ref{functoriality of pro-kfl pullback of kfl crys crys}. We shall prove that this functor is a bi-exact equivalence.
    
    We shall construct an exact quasi-inverse functor. Let $p_{i}\colon (Y^{(1)},\cM_{Y^{(1)}})\to (Y,\cM_{Y})$ be the $i$-th projection for $i=1,2$. For an object $(\cE_{\infty},p_{1}^{*}\cE_{\infty}\cong p_{2}^{*}\cE_{\infty})\in \displaystyle \varprojlim_{\bullet\in \Delta}\mathrm{Coh}(Y^{(\bullet)}_{\mathrm{crys}})$, we define $\cE\in \mathrm{Coh}_{\mathrm{kfl}}((X,\cM_{X})_{\mathrm{crys}})$ associated with this object as follows. To do this, it suffices to define the evaluation of $\cE$ at an affine PD-thickening $(U,T,\cM_{T})\in (X,\cM_{X})^{\mathrm{str}}_{\mathrm{crys}}$. Since $\Gamma(T,\cM_{T})\to \Gamma(U,\cM_{U})$ is surjective by \cite[Lemma A.1]{ino25}, we can choose a lift $\widetilde{\alpha}:\bN^{r}\to \cM_{T}$ of $\alpha$. We set $(T_{\infty,\widetilde{\alpha}},\cM_{T_{\infty,\widetilde{\alpha}}})$ as in Setting \ref{setting for pro-kfl descent}. Since $T_{\infty,\widetilde{\alpha}}\to T$ is flat, $(U_{\infty,\alpha},T_{\infty,\widetilde{\alpha}},\cM_{T_{\infty,\widetilde{\alpha}}})$ is a log PD-thickening in $(X_{\infty,\alpha},\cM_{X_{\infty,\alpha}})_{\mathrm{crys}}^{\mathrm{str}}$. By Lemma \ref{root descent for crys crys}, there exists an affine log PD-thickening $(U',T',\cM_{T'})\in (Y,\cM_{Y})_{\mathrm{crys}}^{\mathrm{str}}$ and a Cartesian strict flat map of PD-thickenings
    \[
    (U',T',\cM_{T'})\to (U_{\infty,\alpha},T_{\infty,\widetilde{\alpha}},\cM_{T_{\infty,\widetilde{\alpha}}})
    \]
    such that the natural map $U'\to U_{\infty,\alpha}\times_{X_{\infty,\alpha}} Y=U\times_{X} Y$ is an \'{e}tale covering. In particular, $U'\to U$ is surjective. We let $(T'^{(\bullet)},\cM_{T'^{(\bullet)}})$ denote the \v{C}ech nerve of $(T',\cM_{T'})\to (T,\cM_{T})$ in the category of saturated log schemes. For each $n\geq 1$, every projection map $T'^{(n)}\to T$ is flat by Lemma \ref{sat prod of two pro-kfl cover}, and so $(T'^{(n)},\cM_{T'^{(n)}})$ is a log PD-thickening in $(Y^{(n)},\cM_{Y^{(n)}})_{\mathrm{crys}}^{\mathrm{str}}$. Let $q_{i}\colon (T'^{(1)},\cM_{T'^{(1)}})\to (T',\cM_{T'})$ be the $i$-th projection for $i=1,2$. Then the coherent sheaf $(\cE_{\infty})_{(T',\cM_{T'})}$ equipped with the isomorphism
    \[
    q_{1}^{*}((\cE_{\infty})_{(T',\cM_{T'})})\cong (p_{1}^{*}\cE_{\infty})_{(T'^{(1)},\cM_{T'^{(1)}})}\cong (p_{2}^{*}\cE_{\infty})_{(T'^{(1)},\cM_{T'^{(1)}})} \cong q_{2}^{*}((\cE_{\infty})_{(T',\cM_{T'})})
    \]
    which satisfies the cocycle condition in $\mathrm{Coh}(T'^{(2)})$
    descents to an object of $\mathrm{Coh}_{\mathrm{kfl}}(T,\cM_{T})$ by Proposition \ref{pro-kfl descent for kfl vect bdle}. The resulting object is denoted by $\cE_{(T,\cM_{T})}$.  It follows from the same argument as the proof of Proposition \ref{kqsyn descent for kfl log pris crys} that the formation of $\cE_{(T,\cM_{T})}$ is independent of the choice of $\widetilde{\alpha}$ and $T'$. Therefore, the family of $\cE_{(T,\cM_{T})}$ forms $\cE\in \mathrm{Coh}_{\mathrm{kfl}}((X,\cM_{X})_{\mathrm{crys}})$, and the functor sending $(\cE_{\infty},p_{1}^{*}\cE_{\infty}\cong p_{2}^{*}\cE_{\infty})$ to $\cE$ is indeed an exact quasi-inverse by Proposition \ref{pro-kfl descent for kfl vect bdle}. 
\end{proof}

\begin{cor}\label{kfl crys crys and kqsyn site}
    Let $(X,\cM_{X})$ be a fs log scheme over $\bF_{p}$. Suppose that $(X,\cM_{X})$ admits a free chart \'{e}tale locally and $X$ is quasi-syntomic. Then there exists a natural equivalence
    \begin{align*}
    \mathrm{Coh}_{\mathrm{kfl}}((X,\cM_{X})_{\mathrm{crys}})&\simeq \varprojlim_{(Y,\cM_{Y})\in (X,\cM_{X})'_{\mathrm{kqsyn}}} \mathrm{Coh}(Y_{\mathrm{crys}}) \\
    &\simeq \varprojlim_{(S,\cM_{S})\in (X,\cM_{X})'_{\mathrm{qrsp}}} \mathrm{Coh}(S_{\mathrm{crys}}).
    \end{align*}
    The analogous assertion also holds for categories $\mathrm{Vect}^{\star}_{\mathrm{kfl}}((X,\cM_{X})_{\mathrm{crys}})$ for $\star\in \{\emptyset,\varphi\}$. 
\end{cor}

\begin{proof}
For any $(f\colon (Y,\cM_{Y})\to (X,\cM_{X}))\in (X,\cM_{X})'_{\mathrm{kqsyn}}$, every crystalline crystal on $(X,\cM_{X})$ is $f^{*}$-admissible, and the functors $f^{*}$ induce the functor in the statement by Lemma \ref{log crys site p-div case}, Lemma \ref{functoriality of pro-kfl pullback of kfl crys crys}, and Lemma \ref{pro-kfl pullback of kfl crys crys is well-def}. 

We shall prove that this functor gives an equivalence. By working \'{e}tale locally on $X$, we may assume that there exists a free chart $\alpha\colon \bN^{r}\to \cM_{X}$. By ``Kummer root descent for log crystalline crystals'' (Proposition \ref{kpr descent for kfl crystalline crystals}), it is enough to show that the natural functor
    \[
    \varprojlim_{(Y,\cM_{Y})\in (X,\cM_{X})'_{\mathrm{kqsyn}}} \mathrm{Vect}(Y_{\mathrm{crys}})\to \varprojlim_{\bullet\in \Delta} \mathrm{Vect}((X^{(\bullet)}_{\infty,\alpha})_{\mathrm{crys}})
    \]
    is an equivalence. For any $(f\colon (Y,\cM_{Y})\to (X,\cM_{X}))\in (X,\cM_{X})'_{\mathrm{kqsyn}}$, the projection morphism from the saturated fiber product
    \[ (X_{\infty,\alpha},\cM_{X_{\infty,\alpha}})\times_{(X,\cM_{X})}^{\mathrm{sat}} (Y,\cM_{Y})\to (Y,\cM_{Y})
    \]
    is a strict root covering by Lemma \ref{four point lemma in nonfs case} and Lemma \ref{sat prod of two pro-kfl cover} (3). Therefore, the claim follows from root descent for (non-log) crystalline crystals (Lemma \ref{root descent for crys crys}).
\end{proof}

\begin{prop}\label{pris crys is equal to crys crys}
    Let $(X,\cM_{X})$ be an fs log scheme over $\bF_{p}$. Suppose that the underlying scheme $X$ is quasi-syntomic and that $(X,\cM_{X})$ admits a prismatically liftable chart \'{e}tale locally. Then there exist canonical bi-exact equivalences
    \[
    \mathrm{Vect}^{\star}_{\mathrm{kfl}}((X,\cM_{X})_{\Prism})\simeq \mathrm{Vect}^{\star}_{\mathrm{kfl}}((X,\cM_{X})_{\mathrm{crys}})
    \]
    for $\star\in \{\emptyset,\varphi\}$.
\end{prop}

\begin{proof}
    The assertion follows from Corollary \ref{kfl pris crys and kqsyn site}, Corollary \ref{kfl crys crys and kqsyn site}, and the equivalence in \cite[\S 2.1.1.1]{iky23}.
\end{proof}

\begin{dfn}[Crystalline realization]

Let $(\fX,\cM_{\fX})$ be a bounded $p$-adic fs log formal scheme and $(X_{0},\cM_{X_{0}})$ be the mod-$p$ fiber of $(\fX,\cM_{\fX})$. Suppose that both of $\fX$ and $X_{0}$ are quasi-syntomic and $(\fX,\cM_{\fX})$ admits a prismatically liftable chart \'{e}tale locally. Note that $(X,\cM_{X})$ also admits a prismatically liftable chart \'{e}tale locally by Lemma \ref{ex of pris liftable chart}. Then, by composing the equivalences in Proposition \ref{pris crys is equal to crys crys} with the restriction functors
\[
\mathrm{Vect}^{\star}_{\mathrm{kfl}}((\fX,\cM_{\fX})_{\Prism})\to \mathrm{Vect}^{\star}_{\mathrm{kfl}}((X,\cM_{X})_{\Prism}),
\]
we obtain exact functors
\[
\mathrm{Vect}^{\star}_{\mathrm{kfl}}((\fX,\cM_{\fX})_{\Prism})\to \mathrm{Vect}^{\star}_{\mathrm{kfl}}((X,\cM_{X})_{\mathrm{crys}})
\]
for $\star\in \{\emptyset,\varphi\}$. These functors are called \emph{crystalline realization functors} and denoted by $T_{\mathrm{crys}}$.

\end{dfn}

\begin{construction}[Construction of crystalline realization functors using Breuil log prisms]\label{crys real by using bk prisms}\noindent

Let $(\fX,\cM_{\fX})=(\mathrm{Spf}(R),\cM_{R})$ be a small affine log formal scheme over $\cO_{K}$, and fix a framing. Let $\star\in \{\emptyset,\varphi\}$. Let $\mathrm{DD}^{\mathrm{vect}}_{S_{R}}$ be the category of pairs $(M,\epsilon)$ consisting of a finite projective $S_{R}$-module $M$ equipped with an isomorphism
\[
\epsilon\colon M\otimes_{S_{R},p_{1}} S_{R}^{(1)}\cong M\otimes_{S_{R},p_{2}} S_{R}^{(1)}
\]
satisfying the cocycle condition over $S_{R}^{(2)}$. Let $\mathrm{DD}^{\varphi,\mathrm{vect}}_{S_{R}}$ be the category of triples $(M,\epsilon,\varphi_{M})$ consisting of a pair $(M,\epsilon)\in \mathrm{DD}^{\mathrm{vect}}_{S_{R}}$ and an isomorphism $\varphi_{M}\colon (\phi_{S_{R}}^{*}M)[1/p]\isom M[1/p]$ such that $f$ is compatible with $\varphi_{M}$.

For $\cE\in \mathrm{Vect}^{\star}((\fX,\cM_{\fX})_{\Prism})$, the evaluation at Breuil log prisms $(S_{R}^{(i)},(p),\phi^{*}_{S_{R}^{(i)}}\cM_{S_{R}^{(i)}})$ gives a functor
\[
\mathrm{Vect}^{\star}((\fX,\cM_{\fX})_{\Prism})\to \mathrm{DD}_{S_{R}}^{\star,\mathrm{vect}}.
\]
Since $(\mathrm{Spf}(S_{R}),\cM_{S_{R}})$ covers the final object of the topos associated with the absolute log crystalline site $(X_{p=0},\cM_{X_{p=0}})_{\mathrm{crys}}$, there is a natural equivalence
\[
\mathrm{DD}_{S_{R}}^{\star,\mathrm{vect}}\simeq \mathrm{Vect}^{\star}((X_{p=0},\cM_{X_{p=0}})_{\mathrm{crys}}).
\]
Taking the composition of them gives a functor
\[
\mathrm{Vect}^{\star}((\fX,\cM_{\fX})_{\Prism})\to \mathrm{Vect}^{\star}((X_{p=0},\cM_{X_{p=0}})_{\mathrm{crys}}).
\]
\end{construction}

\begin{prop}\label{two crys real}
    Let $(\fX,\cM_{\fX})$ be a small affine log formal scheme over $\cO_{K}$ with a fixed framing. Let $\star\in \{\emptyset,\varphi \}$ and $\cE\in \mathrm{Vect}^{\star}((\fX,\cM_{\fX})_{\Prism})$. Then the image of $\cE$ in $\mathrm{Vect}^{\star}((X_{p=0},\cM_{X_{p=0}})_{\mathrm{crys}})$ via the functor constructed in Construction \ref{crys real by using bk prisms} is naturally isomorphic to $T_{\mathrm{crys}}(\cE)$. 
\end{prop}

\begin{proof}
    We just treat the case that $\star=\emptyset$. Let 
    \[
    T_{\mathrm{crys}}'\colon \mathrm{Vect}((\fX,\cM_{\fX})_{\Prism})\to \mathrm{Vect}((X,\cM_{X})_{\mathrm{crys}})
    \]
    denote the functor defined in Construction \ref{crys real by using bk prisms}. By Corollary \ref{kfl pris crys and kqsyn site} and Corollary \ref{kfl crys crys and kqsyn site}, it suffices to prove that there exists a functorial isomorphism
    \[
    T_{\mathrm{crys}}(\cE)_{(A_{\mathrm{crys}}(S)\twoheadrightarrow S,\cM_{A_{\mathrm{crys}}(S)})}\cong T_{\mathrm{crys}}'(\cE)_{(A_{\mathrm{crys}}(S)\twoheadrightarrow S,\cM_{A_{\mathrm{crys}}(S)})}
    \]
    for $(S,\cM_{S})\in (X,\cM_{X})'_{\mathrm{qrsp}}$. We may assume that the structure morphism $(\mathrm{Spec}(S),\cM_{S})\to (X,\cM_{X})$ admits a factorization
    \[
    (\mathrm{Spec}(S),\cM_{S})\to (X_{\infty,\alpha},\cM_{X_{\infty,\alpha}})\to  (X,\cM_{X}).
    \]
    
    Choose a map of $p$-adic PD-thickenings in $(X,\cM_{X})_{\mathrm{crys}}$
    \[
    s\colon (S_{R}\twoheadrightarrow R/p,\cM_{S_{R}})\to (A_{\mathrm{crys}}(S)\twoheadrightarrow S,\cM_{A_{\mathrm{crys}}(S)})
    \]
    that is compatible with $\delta_{\mathrm{log}}$-structures. For example, we can construct such a map as the following. By \cite[Lemma A.3]{ino25}, the monoid map $\bQ^{r}_{\geq 0}\to \cM_{X_{\infty,\alpha}}\to \cM_{S}$ lifts uniquely to a monoid map $\bQ^{r}_{\geq 0}\to \cM_{A_{\mathrm{crys}}(S)}$, which gives a chart of $(\mathrm{Spf}(A_{\mathrm{crys}}(S)),\cM_{A_{\mathrm{crys}}(S)})$. Consider a prelog ring map $(A,\bN^{r})\to (A_{\mathrm{crys}}(S),\bQ_{\geq 0}^{r})$ such that the ring map $A\to A_{\mathrm{crys}}(S)$ sends $x_{i}$ to $[x_{i}^{\flat}]$ and the monoid map $\bN^{r}\to \bQ_{\geq 0}^{r}$ is a natural inclusion. This prelog ring map induces a map of $p$-adic PD-thickenings in $(\mathrm{Spec}(R^{0}/p),\cM_{R^{0}/p})_{\mathrm{crys}}$
    \[
    (S_{R^{0}}\twoheadrightarrow R^{0}/p,\cM_{S_{R^{0}}})\to (A_{\mathrm{crys}}(S)\twoheadrightarrow S/p,\cM_{A_{\mathrm{crys}}(S)})
    \]
    which is compatible with $\delta_{\mathrm{log}}$-structures. Since $S_{R^{0}}\to S_{R}$ is \'{e}tale, there exists a unique lift
    \[
    s\colon (S_{R}\twoheadrightarrow R/p,\cM_{S_{R}})\to (A_{\mathrm{crys}}(S)\twoheadrightarrow S/p,\cM_{A_{\mathrm{crys}}(S)})
    \]
    that is compatible with $\delta_{\mathrm{log}}$-structures. 
    
    When we fix such a choice $t$, we have a log prism map
    \[
    t\colon (S_{R},(p),\phi^{*}\cM_{S_{R}})\to (A_{\mathrm{crys}}(S),(p),\phi^{*}\cM_{A_{\mathrm{crys}}(S)}),
    \]
    and we have isomorphisms
    \begin{align*}
        T_{\mathrm{crys}}(\cE)_{(A_{\mathrm{crys}}(S)\twoheadrightarrow S,\cM_{A_{\mathrm{crys}}(S)})}
        &\cong \cE_{(A_{\mathrm{crys}}(S),(p),\phi^{*}\cM_{A_{\mathrm{crys}}(S)})} \\
        &\cong \cE_{(S_{R},(p),\phi^{*}\cM_{S_{R}})}\otimes_{S_{R},s} A_{\mathrm{crys}}(S) \\ 
        &\cong T_{\mathrm{crys}}'(\cE)_{(S_{R}\twoheadrightarrow R/p,\cM_{S_{R}})}\otimes_{S_{R},s} A_{\mathrm{crys}}(S) \\
        &\cong T_{\mathrm{crys}}'(\cE)_{(A_{\mathrm{crys}}(S)\twoheadrightarrow S,\cM_{A_{\mathrm{crys}}(S)})}.
    \end{align*}
    For another choice $t'$, the maps $t$ and $t'$ induce a map of $p$-adic log PD-thickenings
    \[
    (S_{R}^{(1)}\twoheadrightarrow R/p,\cM_{S_{R}^{(1)}})\to (A_{\mathrm{crys}}(S)\twoheadrightarrow S,\cM_{A_{\mathrm{crys}}(S)})
    \]
    and a log prism map
    \[
    (S_{R}^{(1)},(p),\phi^{*}\cM_{S_{R}^{(1)}})\to (A_{\mathrm{crys}}(S),(p),\phi^{*}\cM_{A_{\mathrm{crys}}(S)}),
    \]
    and it follows from the crystal property of $\cE$ and $T'_{\mathrm{crys}}(\cE)$ that the isomorphism defined above is independent of the choice of $t$. This completes the proof.
\end{proof}

\begin{cor}
    Let $(\fX,\cM_{\fX})$ be a semi-stable log formal scheme over $\cO_{K}$. Then the crystalline realization functor restricts to functors
    \[
    T_{\mathrm{crys}}\colon \mathrm{Vect}^{\star}((\fX,\cM_{\fX})_{\Prism})\to \mathrm{Vect}^{\star}((X_{p=0},\cM_{X_{p=0}})_{\mathrm{crys}}).
    \]
\end{cor}

\begin{proof}
    We may work \'{e}tale locally on $\fX$, and so the assertion follows from Proposition \ref{two crys real}.
\end{proof}

\section{Log prismatic Dieudonn\'{e} theory}

\subsection{Review on prismatic Dieudonn\'{e} theory}

we recall some results on prismatic Dieudonn\'{e} theory for classical $p$-divisible groups. Let $\fX$ be a bounded $p$-adic formal scheme.

\begin{construction}
    Let $G$ be a finite and locally free group scheme or a $p$-divisible group over $\fX$. We define a sheaf $G_{\Prism}$ on $\fX_{\Prism}$ by
    \[
    G_{\Prism}(A,I)\coloneqq G(A/I).
    \]
    Here, we consider an $\cO_{\Prism}$-module 
    \[
    \cM_{\Prism}(G)\coloneqq \mathcal{E}xt^{1}_{(\fX_{\Prism})}(G_{\Prism},\cO_{\Prism}).
    \]
    The morphism $\varphi\colon \cO_{\Prism}\to \cO_{\Prism}$ induces a $\cO_{\Prism}$-linear morphism $\varphi_{\cM_{\Prism}(G)}\colon \varphi^{\ast}\cM_{\Prism}(G)\to \cM_{\Prism}(G)$
\end{construction}

\begin{thm}[Ansch\"{u}tz--Le Bras]\label{prismatic Dieudonne}

Let $\fX$ be a quasi-syntomic $p$-adic formal scheme. Then the following statements hold.
\begin{enumerate}
    \item Let $G$ be a truncated Barsotti-Tate group over $\fX$ of level $n$ and height $h$. Then $\cM_{\Prism}(G)$ is a vector bundle of rank $h$ on the ringed site $(\fX_{\Prism},\cO_{\Prism}/p^{n})$.
    \item Let $G$ be a $p$-divisible group over $\fX$. Then $(\cM_{\Prism}(G),\varphi_{\cM_{\Prism}(G)})$ belongs to $\mathrm{DM}(\fX_{\Prism})$, and the contravariant functor
\[
\cM_{\Prism}\colon \mathrm{BT}(\fX)\to \mathrm{DM}(\fX_{\Prism}) 
\]
is fully faithful. Moreover, if $\fX$ has an \'{e}tale cover consisting of an affine formal scheme $\Spf R$ which admits a quasi-syntomic cover $R\to S$ with $S$ being an integral perfectoid ring, this functor gives an anti-equivalence. 
\end{enumerate}
\end{thm}

\begin{proof}

\begin{enumerate}
    \item See \cite[Proposition 4.68]{alb23}.
    \item These are direct consequences of \cite[Theorem 4.74]{alb23} and the proof of \cite[Proposition 5.10]{alb23}. 
\end{enumerate}
    
\end{proof}

\begin{thm}[Ansch\"{u}tz-Le Bras]\label{prismatic Dieudonne and crystalline realization}
Let $X$ be a quasi-syntomic scheme over $\bF_{p}$ and $G$ be a $p$-divisible group on $X$. Then there exists a canonical isomorphism 
\[
T_{\mathrm{crys}}(\cM_{\Prism}(G))\cong \bD(G),
\]
where $\bD$ is a contravariant crystalline Dieudonn\'{e} functor of Berthelot-Breen-Messing (\cite[D\'{e}finition 3.3.6]{bbm82}).
    
\end{thm}

\begin{proof}
    See \cite[Theorem 4.44]{alb23}.
\end{proof}

\begin{cor}\label{hodge filt on pris Dieudonne}
    Let $\fX$ be a quasi-syntomic $p$-adic formal scheme and $G$ be a $p$-divisible group over $\fX$. Suppose that $X\coloneqq\fX\times_{\bZ_{p}} \bF_{p}$ is also quasi-syntomic. Then we have the canonical exact sequence
    \[
    0\to \mathrm{Lie}(G)^{*}\to T_{\mathrm{crys}}(\cM_{\Prism}(G))_{(X,\fX)}\to \mathrm{Lie}(G^{*})\to 0.
    \]
\end{cor}

\begin{proof}
    This follows Theorem \ref{prismatic Dieudonne and crystalline realization} and \cite[Corollaire 3.3.5]{bbm82}
\end{proof}

\begin{thm}[Ansch\"{u}tz-Le Bras]\label{prismatic Dieudonne and etale realization}
Let $\fX$ be a quasi-syntomic $p$-adic formal scheme and $G$ be a $p$-divisible group on $\fX$. Then there exists a canonical isomorphism 
\[
T_{\et}(\cM_{\Prism}(G))^{\vee}\cong G_{\eta}.
\]
    
\end{thm}

\begin{proof}
    See \cite[Proposition 5.25]{alb23}. 
\end{proof}

\subsection{Definition of the log prismatic Dieudonn\'{e} functor}

In order to construct the log prismatic Dieudonn\'{e} functor, we work in the larger categories and construct a functor
\[
\cM_{\Prism}: \mathrm{wBT}(\fX,\cM_{\fX})\to \mathrm{DM}_{\mathrm{kfl}}((\fX,\cM_{\fX})_{\Prism})
\]
under some assumptions. For later use, we also construct a functor
\[
\widetilde{\cM}_{\Prism}: \mathrm{wBT}_{1}(\fX,\cM_{\fX})\to \mathrm{Vect}_{\mathrm{kfl}}((\fX,\cM_{\fX})_{\Prism},\widetilde{\cO}_{\Prism})
\]
as an auxiliary object. In order to achieve these, we reduce the problems to the non-log cases by using some results on the kfl descent.

\begin{construction}\label{construction of log prismatic Dieudonne}

Let $(\fX,\cM_{\fX})$ be a $p$-adic fs log formal scheme. Suppose that $\fX$ is quasi-syntomic and that $(\fX,\cM_{\fX})$ admits a prismatically liftable chart \'{e}tale locally.

Let $(\fY,\cM_{\fY})\in (\fX,\cM_{\fX})'_{\mathrm{kqsyn}}$ and $f\colon (\fY,\cM_{\fY})\to (\fX,\cM_{\fX})$ be the structure morphism. Since every weak log $p$-divisible group on $(\fX,\cM_{\fX})$ is $f^{*}$-admissible by Lemma \ref{functoriality of pro-kfl pullback of log fin grp} and Lemma \ref{pro-kfl pullback of log fin grp is well-def}, we have functors
\[
\mathrm{wBT}(\fX,\cM_{\fX})\to \mathrm{BT}(\fY)\stackrel{\cM_{\Prism}}{\to} \mathrm{DM}(\fY_{\Prism}),
\]
where $\cM_{\Prism}$ is the prismatic Dieudonn\'{e} functor in Theorem \ref{prismatic Dieudonne}. Taking the limit with respect to $(\fY,\cM_{\fY})\in (\fX,\cM_{\fX})'_{\mathrm{kqsyn}}$, we obtain a functor
\[
\mathrm{wBT}(\fX,\cM_{\fX})\to \mathrm{DM}_{\mathrm{kfl}}((\fX,\cM_{\fX})_{\Prism})
\]
by Corollary \ref{kfl pris crys and kqsyn site}. The resulting functor is also denoted by $\cM_{\Prism}$. 

In the same way, we define a functor
\[
\widetilde{\cM}_{\Prism}: \mathrm{wBT}_{\mathrm{1}}(\fX,\cM_{\fX})\to \mathrm{Vect}_{\mathrm{kfl}}((\fX,\cM_{\fX})_{\Prism},\widetilde{\cO}_{\Prism}).
\]
\end{construction}

\begin{prop}\label{log pris dieudonne weak ver} 
Let $(\fX,\cM_{\fX})$ be a $p$-adic fs log formal scheme. Suppose that $\fX$ is quasi-syntomic and $(\fX,\cM_{\fX})$ admits a prismatically liftable chart \'{e}tale locally.

\begin{enumerate}
\item The following diagram is commutative:
\[
\begin{tikzcd}
        \mathrm{wBT}(\fX,\cM_{\fX}) \ar[r,"\cM_{\Prism}"] \ar[d,"G\mapsto {G[p]}"'] & \mathrm{DM}_{\mathrm{kfl}}((\fX,\cM_{\fX})_{\Prism}) \ar[d,"(\cE{,}\varphi_{\cE})\mapsto \cE\otimes_{\cO_{\Prism}} \widetilde{\cO}_{\Prism}"] \\
        \mathrm{wBT}_{\mathrm{1}}(\fX,\cM_{\fX}) \ar[r," \widetilde{\cM}_{\Prism}"] & \mathrm{Vect}_{\mathrm{kfl}}((\fX,\cM_{\fX})_{\Prism},\widetilde{\cO}_{\Prism}).
    \end{tikzcd}
    \]
    
\item The log prismatic Dieudonn\'{e} functors are compatible with the (non-log) prismatic Dieudonn\'{e} functors. To be precise, we have the following commutative diagrams.
    \[
    \begin{tikzcd}
    \mathrm{BT}(\fX) \ar[r,"\cM_{\Prism}"] \ar[d,hook] & \mathrm{DM}(\fX_{\Prism}) \ar[d,hook] \\
    \mathrm{wBT}(\fX,\cM_{\fX}) \ar[r,"\cM_{\Prism}"] &
    \mathrm{DM}_{\mathrm{kfl}}((\fX,\cM_{\fX})_{\Prism})
    \end{tikzcd}
    \]
    \[
    \begin{tikzcd}
    \mathrm{BT}_{1}(\fX) \ar[r,"\cM_{\Prism}"] \ar[d,hook] & \mathrm{Vect}(\fX_{\Prism},\cO_{\Prism}/p) \ar[r] & \mathrm{Vect}(\fX_{\Prism},\widetilde{\cO}_{\Prism}) \ar[d,hook] \\
    \mathrm{wBT}_{1}(\fX,\cM_{\fX}) \ar[rr,"\widetilde{\cM}_{\Prism}"] & &
    \mathrm{Vect}_{\mathrm{kfl}}((\fX,\cM_{\fX})_{\Prism},\widetilde{\cO}_{\Prism})
    \end{tikzcd}
    \]
    
\item Suppose that, \'{e}tale locally, $(\fX,\cM_{\fX})$ admits a prismatically liftable chart $\alpha$ such that $\fX_{\infty,\alpha}$ has a perfectoid quasi-syntomic cover. Then the functor $\cM_{\Prism}$ gives an equivalence.
    \end{enumerate}
    
\end{prop}

\begin{proof}
    The assertions $(1)$ and $(2)$ follow from the construction.
    
    $(3)$: By working \'{e}tale locally on $\fX$, we may assume that $(\fX,\cM_{\fX})$ admits a prismatically liftable chart $\alpha:\bN^{r}\to \cM_{\fX}$. Take a strict quasi-syntomic cover $(\fY,\cM_{\fY})=(\mathrm{Spf}(R),\cM_{R})\to (\fX,\cM_{\fX})$ where $R$ is perfectoid. Then the functor
    \[
    \cM_{\Prism}\colon \mathrm{BT}(\fY)\to \mathrm{DM}(\fY_{\Prism})
    \]
    gives an equivalence, and the functor
    \[
    \cM_{\Prism}\colon \mathrm{BT}(\fY^{(n)})\to \mathrm{DM}((\fY^{(n)})_{\Prism})
    \]
    is fully faithful for $n\geq 1$. Therefore, the assertion follows from Proposition \ref{pro-kfl descent for log BT} and Proposition \ref{kqsyn descent for kfl log pris crys}.
\end{proof} 

\begin{prop}[Pullback compatibility of log prismatic Dieudonn\'{e} functors]
    Let $f\colon (\fX,\cM_{\fX})\to (\fY,\cM_{\fY})$ be a morphism of $p$-adic fs log formal schemes. Suppose that $\fX$ and $\fY$ are quasi-syntomic and that $(\fX,\cM_{\fX})$ and $(\fY,\cM_{\fY})$ admit prismatically liftable charts \'{e}tale locally. Let $G$ be a weak log $p$-divisible group on $(\fY,\cM_{\fY})$. Then there exists a natural isomorphism
    \[
    \cM_{\Prism}(G_{(\fX,\cM_{\fX})})\cong f^{*}\cM_{\Prism}(G).
    \]
    The analogous assertion holds for $\widetilde{\cM}_{\Prism}$ as well.
\end{prop}

\begin{proof}
    By working \'{e}tale locally on $\fX$ and $\fY$, we may assume that $\fX$ and $\fY$ have prismatically liftable charts $\alpha\colon \bN^{r}\to \cM_{\fX}$ and $\beta\colon \bN^{s}\to \cM_{\fY}$, a chart $\alpha'\colon \bN^{r}\oplus \bZ^{s}\to \cM_{\fX}$, and a chart $\gamma\colon \bN^{s}\to \bN^{r}\oplus \bZ^{s}$ of $f$ such that $\alpha'$ restricts to $\alpha$ by Lemma \ref{modify chart of sch to chart of mor}. Consider the full subcategory $\sC$ of $(\fX,\cM_{\fX})'_{\mathrm{kqsyn}}$ consisting of objects $(\fX',\cM_{\fX'})$ such that there exist $(\fY',\cM_{\fY'})\in (\fY,\cM_{\fY})'_{\mathrm{kqsyn}}$ and the following commutative diagram:
    \[
    \begin{tikzcd}
        (\fX',\cM_{\fX'}) \ar[r,"f'"] \ar[d,"p"] & (\fY',\cM_{\fY'}) \ar[d,"q"] \\
        (\fX,\cM_{\fX}) \ar[r,"f"] & (\fY,\cM_{\fY}).
    \end{tikzcd}
    \]
    The full subcategory $\sC$ is closed under finite products.

    We shall check that $\sC$ is cofinal in $(\fX,\cM_{\fX})'_{\mathrm{kqsyn}}$. We use the notation in Setting \ref{setting for pro-kfl descent}. We have the following commutative diagram:
    \[
    \begin{tikzcd}
    (\fX_{\infty,\alpha'},\cM_{\fX_{\infty,\alpha'}}) \ar[r] \ar[d] & (\fX_{\infty,\alpha},\cM_{\fX_{\infty,\alpha}}) \ar[r] & (\fX,\cM_{\fX}) \ar[d] \\
    (\fY_{\infty,\beta},\cM_{\fY_{\infty,\beta}}) \ar[rr] & & (\fY,\cM_{\fY}).
    \end{tikzcd}
    \]
    In this diagram, the upper left horizontal morphism is a strict quasi-syntomic cover, and every underlying formal scheme is quasi-syntomic. Then $(\fX_{\infty,\alpha'},\cM_{\fX_{\infty,\alpha'}})$ belongs to $\sC$, and we see that $\sC$ is cofinal in $(\fX,\cM_{\fX})'_{\mathrm{kqsyn}}$ by considering products with $(\fX_{\infty,\alpha'},\cM_{\fX_{\infty,\alpha'}})$. Therefore, by Corollary \ref{kfl pris crys and kqsyn site}, there exists a natural equivalence
    \[
    \mathrm{DM}((\fX,\cM_{\fX})_{\Prism})\isom \varprojlim_{(\fX',\cM_{\fX'})\in \sC} \mathrm{DM}(\fX'_{\Prism}).
    \]
    
    Let $(\fX',\cM_{\fX'})\in \sC$ and choose a commutative diagram as in the first paragraph. We have isomorphisms
    \begin{align*}
        p^{*}\cM_{\Prism}(G_{(\fX,\cM_{\fX})})&\cong \cM_{\Prism}(G_{(\fX',\cM_{\fX'})}) \\
        &\cong f'^{*}\cM_{\Prism}(G_{(\fY',\cM_{\fY'})}) \\
        &\cong f'^{*}(q^{*}\cM_{\Prism}(G)) \\
        &\cong p^{*}(f^{*}\cM_{\Prism}(G)),
    \end{align*}
    where the second isomorphism is deduced from the pullback compatibility for (non-log) prismatic Dieudonn\'{e} functor and the forth isomorphism comes from Lemma \ref{functoriality of pro-kfl pullback of kfl pris crys}. We see that this isomorphism is independent of the choice of $(\fY',\cM_{\fY'})$ by considering products in $(\fY,\cM_{\fY})'_{\mathrm{kqsyn}}$. By construction, this isomorphism is natural in $(\fX',\cM_{\fX'})\in \sC$. Therefore, we obtain an isomorphism 
    \[
    \cM_{\Prism}(G_{(\fX,\cM_{\fX})})\cong f^{*}\cM_{\Prism}(G)
    \]
    via the equivalence we checked in the previous paragraph.
\end{proof}

\begin{cor}\label{Hodge filt on log pris Dieudonne}
    Let $(\fX,\cM_{\fX})$ be a $p$-adic fs log formal scheme and $G$ be a weak log $p$-divisible group over $(\fX,\cM_{\fX})$. Suppose that $(\fX,\cM_{\fX})$ and the mod-$p$ fiber $(X,\cM_{X})$ of $(\fX,\cM_{\fX})$ are quasi-syntomic and that $(\fX,\cM_{\fX})$ admits a prismatically liftable chart \'{e}tale locally.
    Then there exists the following natural exact sequence of kfl vector bundles on $\fX$.
    \[
    0\to \mathrm{Lie}(G)^{\vee}\to T_{\mathrm{crys}}(\cM_{\Prism}(G))_{(\fX,\cM_{\fX})}\to \mathrm{Lie}(G^{*})\to 0.
    \]
\end{cor}

\begin{proof}
    Let $(\fY,\cM_{\fY})\in (\fX,\cM_{\fX})'_{\mathrm{kqsyn}}$ and $f\colon (\fY,\cM_{\fY})\to (\fX,\cM_{\fX})$ be the structure morphism. Let $f_{0}\colon (Y,\cM_{Y})\to (X,\cM_{X})$ denote the mod-$p$ fiber of $f$. By Lemma \ref{pullback compatibility for lie alg of log bt} and Corollary \ref{hodge filt on pris Dieudonne}, we have an exact sequence of vector bundles on $\fY_{\eta}$
    \[
    0\to f^{*}\mathrm{Lie}(G)^{\vee}\to T_{\mathrm{crys}}(\cM_{\Prism}(G_{(\fY,\cM_{\fY})}))_{(\fY,\cM_{\fY})}\to f^{*}\mathrm{Lie}(G^{*})\to 0.
    \]
    We have
    \begin{align*}
    T_{\mathrm{crys}}(\cM_{\Prism}(G_{(\fY,\cM_{\fY})}))_{(\fY,\cM_{\fY})} &\cong T_{\mathrm{crys}}(f^{*}\cM_{\Prism}(G))_{(\fY,\cM_{\fY})} \\
    &\cong (f_{0}^{*}T_{\mathrm{crys}}(\cM_{\Prism}(G)))_{(\fY,\cM_{\fY})} \\
    &\cong f_{\eta}^{*}(T_{\mathrm{crys}}(\cM_{\Prism}(G))_{(\fX,\cM_{\fX})}),
    \end{align*}
    where the first (resp. second) isomorphism comes from the functoriality of log prismatic Dieudonn\'{e} functors (resp. the crystalline realization functors), and the third isomorphism is defined by Lemma \ref{crys property for kfl cry crys}. Therefore, the assertion follows from Proposition \ref{pro-kfl descent for kfl vect bdle}.
\end{proof}

\begin{cor}\label{log pris dieudonne special case equivalence weak ver}
    Suppose that $(\fX,\cM_{\fX})$ is a semi-stable log formal scheme over $\cO_{K}$. Then the functor
    \[
    \mathrm{wBT}(\fX,\cM_{\fX})\to \mathrm{DM}_{\mathrm{kfl}}((\fX,\cM_{\fX})_{\Prism})
    \]
    gives an equivalence.
\end{cor}

\begin{proof}
    The assertion follows from \cite[Lemma 2.3]{ino25} and Proposition \ref{log pris dieudonne weak ver}(2). 
\end{proof}

Our goal is to prove that the functor $\cM_{\Prism}$ restricts to the equivalence
\[
\mathrm{BT}(\fX,\cM_{\fX})\isom \mathrm{DM}((\fX,\cM_{\fX})_{\Prism})
\]
when $(\fX,\cM_{\fX})$ is semi-stable over $\cO_{K}$.

\begin{prop}\label{log pris dieudonne classicality}
    Let $(\fX,\cM_{\fX})$ be a $p$-adic fs log formal scheme in Construction \ref{construction of log prismatic Dieudonne}.
    \begin{enumerate}
        \item For $G\in \mathrm{BT}_{1}(\fX,\cM_{\fX})$, the object $\widetilde{\cM}_{\Prism}(G)\in \mathrm{Vect}_{\mathrm{kfl}}((\fX,\cM_{\fX})_{\Prism},\widetilde{\cO}_{\Prism})$ belongs to $\mathrm{Vect}((\fX,\cM_{\fX})_{\Prism},\widetilde{\cO}_{\Prism})$.
        \item For $G\in \mathrm{BT}(\fX,\cM_{\fX})$, the kfl prismatic Dieudonn\'{e} crystal $\cM_{\Prism}(G)$ belongs to $\mathrm{DM}((\fX,\cM_{\fX})_{\Prism})$.
    \end{enumerate}
\end{prop}

\begin{proof}
    Consider the following commutative diagram in Proposition \ref{log pris dieudonne weak ver}.
    \[
    \begin{tikzcd}
        \mathrm{wBT}(\fX,\cM_{\fX}) \ar[r,"\cM_{\Prism}"] \ar[d] & \mathrm{DM}_{\mathrm{kfl}}((\fX,\cM_{\fX})_{\Prism}) \ar[d] \\
        \mathrm{wBT}_{\mathrm{1}}(\fX,\cM_{\fX}) \ar[r," \widetilde{\cM}_{\Prism}"] & \mathrm{Vect}_{\mathrm{kfl}}((\fX,\cM_{\fX})_{\Prism},\widetilde{\cO}_{\Prism}).
    \end{tikzcd}
    \]
    Hence, (2) follows from (1) and Corollary \ref{classicality of vect bdle can be checked after taking reduction}.

    We shall prove (1). By Lemma \ref{conn et seq}, working \'{e}tale locally allows us to  assume that there exists an exact sequence
    \[
    0\to G^{0}\to G\to G^{\et}\to 0,
    \]
    where $G^{0}$ is a (classical) finite locally free group scheme over $\fX$ and $H^{\et}$ is a (classical) finite \'{e}tale group scheme over $\fX$. The exactness of $\widetilde{\cM}_{\Prism}$ implies that we have an exact sequence
    \[
    0\to \widetilde{\cM}_{\Prism}(G^{0})\to \widetilde{\cM}_{\Prism}(G)\to \widetilde{\cM}_{\Prism}(G^{\et})\to 0.
    \]
    By applying Proposition \ref{ext of classical vector bundle} to the evaluation of this exact sequence at each log prism in $(\fX,\cM_{\fX})^{\mathrm{str}}_{\Prism}$, it follows that $\widetilde{\cM}_{\Prism}(G)$ belongs to $\mathrm{Vect}^{\varphi}((\fX,\cM_{\fX})_{\Prism},\widetilde{\cO}_{\Prism})$.
\end{proof}

\begin{dfn}
    The fully faithful functor $\cM_{\Prism}$ induces a fully faithful functor
    \[
    \mathrm{BT}(\fX,\cM_{\fX})\to \mathrm{DM}((\fX,\cM_{\fX})_{\Prism})
    \]
    by Proposition \ref{log pris dieudonne classicality}. This functor is also denoted by $\cM_{\Prism}$.
\end{dfn}

The following proposition is a log version of Theorem \ref{prismatic Dieudonne and etale realization}. 

\begin{prop}[\'{E}tale comparison isomorphisms for log prismatic Dieudonn\'{e} functors]\label{et comp for log pris dieudonne}
    Let $(\fX,\cM_{\fX})$ be a $p$-adic log formal scheme in Construction \ref{construction of log prismatic Dieudonne} and $G$ be a log $p$-divisible group on $(\fX,\cM_{\fX})$. Then there exists a natural isomorphism
    \[
    (T_{\et}\cM_{\Prism}(G))^{\vee}\cong G_{\eta}.
    \]
\end{prop}

\begin{proof}
    By working \'{e}tale locally on $\fX$, we may assume that $\fX$ admits a free chart $\alpha\colon \bN^{r}\to \cM_{\fX}$. Since the natural morphism
    \[ (\fX_{\infty,\alpha},\cM_{\fX_{\infty,\alpha}})_{\eta}^{\Diamond}\to (\fX,\cM_{\fX})_{\eta}^{\Diamond}
    \]
    is a quasi-pro-Kummer \'{e}tale cover by \cite[Proposition 7.20]{ky23}, the natural functor
    \[
    \mathrm{Loc}_{\bZ_{p}}((\fX,\cM_{\fX})_{\eta,\mathrm{qprok}\et}^{\Diamond})\to \varprojlim_{(\fY,\cM_{\fY})\in (\fX,\cM_{\fX})'_{\mathrm{kqsyn}}} \mathrm{Loc}_{\bZ_{p}}((\fY,\cM_{\fY})_{\eta,\mathrm{qprok\et}}^{\Diamond})\]
    is an equivalence.

    Take an object $(\fY,\cM_{\fY})\in (\fX,\cM_{\fX})'_{\mathrm{kqsyn}}$ and let $f\colon (\fY,\cM_{\fY})\to (\fX,\cM_{\fX})$ denote the structure morphism. Under the identification $(\fY,\cM_{\fY})_{\eta,\mathrm{qprok}\et}^{\Diamond}\simeq \fY_{\eta,\mathrm{qpro}\et}^{\Diamond}$ (\cite[Lemma B.1]{ino25}), we have isomorphisms of $\bZ_{p}$-local systems on $(\fY,\cM_{\fY})_{\eta,\mathrm{qprok}\et}^{\Diamond}$
    \begin{align*}
        f_{\eta}^{*}T_{\et}\cM_{\Prism}(G)^{\vee}
        &\cong T_{\et}f^{*}\cM_{\Prism}(G)^{\vee} \\
        &\cong T_{\et}\cM_{\Prism}(G_{(\fY,\cM_{\fY})})^{\vee} \\
        &\cong (G_{(\fY,\cM_{\fY})})_{\eta} \\
        &\cong f_{\eta}^{*}(G_{\eta}).
    \end{align*}
    Here, the third isomorphism comes from \'{e}tale comparison isomorphisms for (non-log) prismatic Dieudonn\'{e} functors (Theorem \ref{prismatic Dieudonne and etale realization}). By the functoriality of the isomorphism in Theorem \ref{prismatic Dieudonne and etale realization}, the isomorphism $f_{\eta}^{*}T_{\et}\cM_{\Prism}(G)^{\vee}\cong f_{\eta}^{*}(G_{\eta})$ constructed above is natural in $(\fY,\cM_{\fY})\in (\fX,\cM_{\fX})'_{\mathrm{kqsyn}}$. Therefore, we obtain an isomorphism $T_{\et}\cM_{\Prism}(G)^{\vee}\cong G_{\eta}$ via the equivalence in the previous paragraph.
\end{proof}

\begin{thm}\label{log pris dieudonne equiv}
    Let $(\fX,\cM_{\fX})$ be a semi-stable log formal scheme over $\cO_{K}$. Then the log prismatic Dieudonn\'{e} functor 
    \[
    \cM_{\Prism}\colon \mathrm{BT}(\fX,\cM_{\fX})\to \mathrm{DM}((\fX,\cM_{\fX})_{\Prism})
    \]
    gives an exact anti-equivalence.
\end{thm}

\begin{proof}
The fully faithfulness is already proved in Proposition \ref{log pris dieudonne weak ver}. We shall the essential surjectivity. Let $\cE$ be an object of $\mathrm{DM}((\fX,\cM_{\fX})_{\Prism})$. 
By Corollary \ref{log pris dieudonne special case equivalence weak ver}, there exists $G\in \mathrm{wBT}(\fX,\cM_{\fX})$ with $\cM_{\Prism}(G)\cong \cE$. It suffices to prove that $G$ is a log $p$-divisible group. Let $(X,\cM_{X})\coloneqq(\fX,\cM_{\fX})\times_{\bZ_{p}} \bF_{p}$. By Proposition \ref{two crys real}, the kfl crystalline crystal $T_{\mathrm{crys}}(\cM_{\Prism}(G))$ belongs to $\mathrm{Vect}^{\varphi}((X,\cM_{X})_{\mathrm{crys}})$. Proposition \ref{ext of classical vector bundle} and Corollary \ref{Hodge filt on log pris Dieudonne} imply that both of $\mathrm{Lie}(G)$ and $\mathrm{Lie}(G^{*})$ are classical. Therefore, $G$ belongs to $\mathrm{BT}(\fX,\cM_{\fX})$ by Proposition \ref{lie is classical iff weak log bt is log bt}. This completes the proof.
\end{proof}

\section{A Tannakian framework for log prismatic \texorpdfstring{$F$}--crystals}

\subsection{Preliminaries for Tannakian formalism}

Throughout this section, we consider a right action when we refer to a pseudo-torsor.

Let $\cG$ be a smooth affine group scheme over $\bZ_{p}$ and $\mathrm{Rep}_{\bZ_{p}}(\cG)$ denote the category of finite free representations of $\cG$ over $\bZ_{p}$. Fix a finite free $\bZ_{p}$-module $\Lambda_{0}$, a faithful representation $\cG\hookrightarrow \mathrm{GL}(\Lambda_{0})$, and tensors $\bT_{0}\subset \Lambda_{0}^{\otimes}$ with $\cG=\mathrm{Fix}(\bT_{0})$. (This is possible by \cite[Theorem 1.1]{bro13}.)

\begin{dfn}
    For an exact $\bZ_{p}$-category $\sA$, the groupoid of $\bZ_{p}$-linear exact tensor functors $\mathrm{Rep}_{\bZ_{p}}(\cG)\to \sA$ is denoted by $\cG\text{-}\sA$.
\end{dfn}

Let $\sC$ be a site and $\cO$ be a sheaf of $\bZ_{p}$-algebras on $\sC$. Let $\mathrm{Vect}(\sC,\cO)$ denote the category of vector bundles on the ringed site $(\sC,\cO)$, which is an exact $\bZ_{p}$-category. For objects $\omega\in \cG\text{-}\mathrm{Vect}(\sC,\cO)$ and $T\in \sC$, the object of $\cG\text{-}\mathrm{Vect}(\sC_{/T},\cO)$ obtained by restricting $\omega$ to the site $\sC_{/T}$ is denoted by $\omega_{T}$. We define a group sheaf $\cG_{\cO}$ on $\sC$ by $\cG_{\cO}(T)\coloneqq\cG(\cO(T))$ for $T\in \sC$.

\begin{dfn}(\cite[Appendix]{iky24})\noindent
\begin{enumerate}
    \item  We define an object $\omega_{\mathrm{triv}}$ of $\cG\text{-}\mathrm{Vect}(\sC,\cO)$ by $\omega_{\mathrm{triv}}(\Lambda)\coloneqq\Lambda\otimes_{\bZ_{p}} \cO$ for $\Lambda\in \mathrm{Rep}_{\bZ_{p}}(\cG)$. We say that $\cG$ is \emph{reconstructible} in $(\sC,\cO)$ if the natural map $\cG_{\cO}\to \underline{\mathrm{Aut}}(\omega_{\mathrm{triv}})$ is an isomorphism, where $\underline{\mathrm{Aut}}(\omega_{\mathrm{triv}})$ is a group sheaf on $\sC$ given by 
    \[
    \underline{\mathrm{Aut}}(\omega_{\mathrm{triv}})(T)\coloneqq \mathrm{Aut}(\omega_{\mathrm{triv},T})
    \]
    for $T\in \sC$.
    \item For an object $\omega$ of $\cG\text{-}\mathrm{Vect}(\sC,\cO)$, we define an $\underline{\mathrm{Aut}}(\omega_{\mathrm{triv}})$-pseudo-torsor $\underline{\mathrm{Isom}}(\omega_{\mathrm{triv}},\omega)$ by
    \[
    \underline{\mathrm{Isom}}(\omega_{\mathrm{triv}},\omega)(T)\coloneqq\mathrm{\mathrm{Isom}}_{\cG\text{-}\mathrm{Vect}(\sC_{/T},\cO)}(\omega_{\mathrm{triv},T},\omega_{T})
    \]
    for $T\in \sC$. The object $\omega$ is called \emph{locally trivial} if the pseudo-torsor $\underline{\mathrm{Isom}}(\omega_{\mathrm{triv}},\omega)$ is a torsor. Let $\cG\text{-}\mathrm{Vect}^{\mathrm{lt}}(\sC,\cO)$ denote the full subcategory of $\cG\text{-}\mathrm{Vect}(\sC,\cO)$ consisting of locally trivial objects.
\end{enumerate}
\end{dfn}

We let $\mathrm{Twist}_{\cO}(\Lambda_{0},\bT_{0})$ denote the groupoid of a pair $(\cE,\bT)$ consisting of $\cE\in \mathrm{Vect}(\sC,\cO)$ and $\bT\in \cE^{\otimes}$ such that $\cG_{\cO}$-pseudo-torsor $\underline{\mathrm{Isom}}((\Lambda_{0}\otimes_{\bZ_{p}} \cO,\bT_{0}\otimes 1),(\cE,\bT))$ given by 
\[
T\mapsto \{f\colon \Lambda_{0}\otimes_{\bZ_{p}} \cO_{T}\isom \cE_{T}:f(\bT_{0}\otimes 1)=\bT\}
\]
is a $\cG_{\cO}$-torsor. We have a natural functor $\mathrm{Twist}_{\cO}(\Lambda_{0},\bT_{0})\to \mathrm{Tors}_{\cG_{\cO}}(\sC)$ sending $(\cE,\bT)$ to $\underline{\mathrm{Isom}}((\Lambda_{0}\otimes_{\bZ_{p}} \cO,\bT_{0}\otimes 1),(\cE,\bT))$.

 For an object $\cP\in \mathrm{Tors}_{\cG_{\cO}}(\sC)$, we define $\omega_{\cP}\in \cG\text{-}\mathrm{Vect}^{\mathrm{lt}}(\sC,\cO)$ by $\omega_{\cP}(\Lambda)\coloneqq\cP\wedge^{\cG_{\cO}} (\Lambda\otimes_{\bZ_{p}} \cO)$ for $\Lambda\in \mathrm{Rep}_{\bZ_{p}}(\cG)$.

\begin{prop}(\cite[Proposition A.17]{iky24})\label{fund thm of tann framework}
    Assume that $\cG$ is reconstructible in $(\sC,\cO)$. Then we have the following triangle of equivalences.
    \[
    \begin{tikzcd}
        \mathrm{Twist}_{\cO}(\Lambda_{0},\bT_{0}) \ar[rr] & & \mathrm{Tors}_{\cG_{\cO}}(\sC) \ar[d,"\cP\mapsto \omega_{\cP}"] \\
        & & \cG\text{-}\mathrm{Vect}^{\mathrm{lt}}(\sC,\cO) \ar[llu,"\omega\mapsto (\omega(\Lambda_{0})\text{,}\omega(\bT_{0}))"]
    \end{tikzcd}
    \]
\end{prop}

\subsection{The Tannakian theory of Kummer \'{e}tale local systems}

For convenience, we write $\bZ_{p}/p^{\infty}$ for $\bZ_{p}$.

\begin{dfn}
   Let $n\in \bN_{\geq 1}\cup \{\infty \}$. Let $(X,\cM_{X})$ be a locally noetherian fs log adic space. We let $\mathrm{Loc}_{\bZ_{p}/p^{n}}(X,\cM_{X})$ (resp. $\mathrm{Loc}_{\bZ_{p}/p^{n}}(X)$) denote the category of $\bZ_{p}/p^{n}$-local system on $(X,\cM_{X})_{\mathrm{prok\et}}$ (resp. $X_{\mathrm{pro\et}}$). When $n\in \bN_{\geq 1}$, the category $\mathrm{Loc}_{\bZ/p^{n}}(X,\cM_{X})$ (resp. $\mathrm{Loc}_{\bZ/p^{n}}(X)$) is equivalent to the category of $\bZ/p^{n}$-local systems on $(X,\cM_{X})_{\mathrm{k\et}}$ (resp. $X_{\et}$). We regard $\mathrm{Loc}_{\bZ_{p}/p^{n}}(X)$ as the full subcategory of $\mathrm{Loc}_{\bZ_{p}/p^{n}}(X,\cM_{X})$.

   An object of $\cG\text{-}\mathrm{Loc}_{\bZ_{p}/p^{n}}(X,\cM_{X})$ (resp. $\cG\text{-}\mathrm{Loc}_{\bZ_{p}/p^{n}}(X)$) is called a \emph{k\'{e}t $\bZ_{p}/p^{n}$-local system with $G$-structure} on $(X,\cM_{X})$ (resp. \emph{\'{e}tale $\bZ_{p}/p^{n}$-local system with $G$-structure} on $X$).
\end{dfn}

\begin{prop}[\cite{dllz23b}] \label{ket loc sys forms gal cat}
    Let $(X,\cM_{X})$ be a connected and locally noetherian fs log adic space and $\overline{x}$ be a log geometric point on $(X,\cM_{X})$ (for the definition, see \cite[Definition 4.4.2]{dllz23b}). The geometric point on $X$ defined by $\overline{x}$ is also denoted by $\overline{x}$. 
    \begin{enumerate}
        \item The category of finite Kummer \'{e}tale fs log adic spaces over $(X,\cM_{X})$ (resp. finite \'{e}tale log adic spaces over $X$) together with the functor taking a fiber at $\overline{x}$ forms a Galois category. The corresponding fundamental group is denoted by $\pi^{\mathrm{k\et}}_{1}((X,\cM_{X}),\overline{x})$ (resp. $\pi^{\et}_{1}(X,\overline{x})$). When there exists no possibility of confusion, $\overline{x}$ is omitted.
        \item For an integer $n\geq 1$, the functor taking a stalk at $\overline{x}$ gives bi-exact equivalences
        \[
        \mathrm{Loc}_{\bZ/p^{n}}(X,\cM_{X})\isom \pi^{\mathrm{k\et}}_{1}(X,\cM_{X})\text{-}\mathrm{FMod}_{\bZ/p^{n}}
        \]
        \[
        \mathrm{Loc}_{\bZ/p^{n}}(X)\isom \pi^{\mathrm{\et}}_{1}(X)\text{-}\mathrm{FMod}_{\bZ/p^{n}}.
        \]
        Here, $G\text{-}\mathrm{FMod}_{\bZ/p^{n}}$ denote the category of finite and free $\bZ/p^{n}$-modules (equipped with discrete topology) with continuous $G$-action.
        \item there exists a natural surjection
        \[
        \pi^{\mathrm{k\et}}_{1}(X,\cM_{X})\to \pi^{\mathrm{\et}}_{1}(X).
        \]
    \end{enumerate}
\end{prop}

\begin{proof}
    \begin{enumerate}
        \item See \cite[Lemma 4.4.16]{dllz23b}.
        \item See \cite[Theorem 4.4.15]{dllz23b}.
        \item Since the natural functor from the category of finite \'{e}tale adic spaces over $X$ to the category of finite Kummer \'{e}tale fs log adic spaces over $(X,\cM_{X})$ is fully faithful, the assertion follows from the well-knows property of Galois categories.
    \end{enumerate}
\end{proof}

\begin{lem}\label{subquot of et loc sys is et loc sys}
    Let $(X,\cM_{X})$ be a locally noetherian fs log adic space and $n\in \bN_{\geq 1}$. Let 
    \[
    0\to \bL_{1}\to \bL_{2}\to \bL_{3}\to 0
    \]
    be an exact sequence in $\mathrm{Loc}_{\bZ/p^{n}}(X,\cM_{X})$. Suppose that $\bL_{2}$ belongs to $\mathrm{Loc}_{\bZ/p^{n}}(X)$. Then $\bL_{1}$ and $\bL_{3}$ also belong to $\mathrm{Loc}_{\bZ/p^{n}}(X)$.
\end{lem}

\begin{proof}
    We may assume that $X$ is connected. The exact sequence in the statement corresponds to an exact sequence
    \[
    0\to M_{1}\to M_{2}\to M_{3}\to 0
    \]
    of finite free $\bZ/p^{n}$-modules $M_{i}$ with continuous $\pi^\mathrm{k\et}_{1}(X,\cM_{X})$-actions for $i=1,2,3$. By the assumption, $\mathrm{Ker}(\pi^{\mathrm{k\et}}_{1}(X,\cM_{X})\twoheadrightarrow \pi^{\et}_{1}(X))$ acts trivially on $M_{2}$. Hence $\pi^{\mathrm{k\et}}_{1}(X,\cM_{X})$-action on $M_{1}$ and $M_{3}$ factors $\pi^{\et}_{1}(X)$-action, which implies that $\bL_{1}$ and $\bL_{3}$ belong to $\mathrm{Loc}(X)$.
\end{proof}

\begin{lem}\label{g-ket loc sys is et after fket base change}
    Let $(X,\cM_{X})$ be a locally noetherian fs log adic space and $n\in \bN_{\geq 1}$. Then, for $\omega\in \cG\text{-}\mathrm{Loc}_{\bZ/p^{n}}(X,\cM_{X})$, there exists a finite Kummer \'{e}tale cover $(Y,\cM_{Y})\to (X,\cM_{X})$ such that $\omega|_{(Y,\cM_{Y})}$ belongs to $\cG\text{-}\mathrm{Loc}_{\bZ/p^{n}}(Y)$.
\end{lem}

\begin{proof}
    Take a finite Kummer \'{e}tale cover $(Y,\cM_{Y})\to (X,\cM_{X})$ such that $\omega(\Lambda_{0})|_{(Y,\cM_{Y})}$ belongs to $\mathrm{Loc}_{\bZ/p^{n}}(Y)$. Then it follows from Lemma \ref{subquot of et loc sys is et loc sys} and \cite[Proposition 12]{ds09} that $\omega(\Lambda)|_{(Y,\cM_{Y})}$ belongs to $\mathrm{Loc}_{\bZ/p^{n}}(Y)$ for every $\Lambda\in \mathrm{Rep}_{\bZ_{p}}(\cG)$.
\end{proof}

\begin{prop}\label{fund thm of tann framework for ket loc sys}
    Let $(X,\cM_{X})$ be a locally noetherian fs log adic space. Then $\cG$ is reconstructible in $((X,\cM_{X})_{\mathrm{prok\et}},\widehat{\bZ}_{p})$ and every object of $\cG\text{-}\mathrm{Loc}_{\bZ_{p}}(X,\cM_{X})$ is locally trivial on $(X,\cM_{X})_{\mathrm{prok\et}}$.
\end{prop}

\begin{proof}
    Since $\cG\text{-}\mathrm{Loc}_{\bZ_{p}}(X)\to \cG\text{-}\mathrm{Loc}_{\bZ_{p}}(X,\cM_{X})$ is fully faithful, it follows from \cite[Proposition 2.8]{iky24} that $\cG$ is reconstructible in $((S,\cM_{X})_{\mathrm{prok\et}},\widehat{\bZ}_{p})$. 
    
    Next, let $\omega\in \cG\text{-}\mathrm{Loc}_{\bZ_{p}}(X,\cM_{X})$. To prove $\omega$ is locally trivial, it suffices to show that $\omega_{n}$ is trivial after finite Kummer \'{e}tale base change for every $n\geq 1$. Here, $\omega_{n}\in \cG\text{-}\mathrm{Loc}_{\bZ/p^{n}}(X,\cM_{X})$ is the mod-$p^{n}$ reduction of $\omega$. Fix $n\geq 1$. By Lemma \ref{g-ket loc sys is et after fket base change}, we may assume that $\omega_{n}$ belongs to $\cG\text{-}\mathrm{Loc}_{\bZ/p^{n}}(X)$. Then the claim follows from the argument in \cite[Proposition 2.3]{iky24}. 
\end{proof}

\begin{prop}\label{comm triangle for g-loc sys}
We have the following commutative triangle of equivalences.
    \[
    \begin{tikzcd}
        \mathrm{Twist}_{\widehat{\bZ}_{p}}(\Lambda_{0},\bT_{0}) \ar[r] & \mathrm{Tors}_{\cG_{\widehat{\bZ}_{p}}}((X,\cM_{X})_{\mathrm{prok\et}}) \ar[d] \\
         & \cG\text{-}\mathrm{Loc}_{\bZ_{p}}(X,\cM_{X}) \ar[lu]
    \end{tikzcd}
    \]
\end{prop}

\begin{proof}
    The assertion follows from Proposition \ref{fund thm of tann framework} and Proposition \ref{fund thm of tann framework for ket loc sys}.
\end{proof}

\subsection{The Tannakian theory of prismatic \texorpdfstring{$F$}--crystals}

Let $(\fX,\cM_{\fX})$ be a bounded $p$-adic log formal scheme.

\begin{dfn}
    An object of $\cG\text{-}\mathrm{Vect}^{\varphi}((\fX,\cM_{\fX})_{\Prism})$ is called a \emph{log prismatic $F$-crystal with $G$-structure} on $(\fX,\cM_{\fX})$.
\end{dfn}

We define a group sheaf $\cG_{\Prism}$ on $(\fX,\cM_{\fX})_{\Prism}$ by
\[
\cG_{\Prism}(A,I,\cM_{A})\coloneqq\cG_{\cO_{\Prism}}(A,I,\cM_{A})=\cG(A).
\]
The ring map $\phi_{A}\colon A\to A$ for each $(A,I,\cM_{A})\in (\fX,\cM_{\fX})_{\Prism}$ induces a group map $\phi\colon \cG_{\Prism}\to \cG_{\Prism}$. Let $\cG_{\Prism}[1/\cI_{\Prism}]$ denote the group sheaf on $(\fX,\cM_{\fX})_{\Prism}$ sending $(A,I,\cM_{A})$ to $\cG(A[1/I])$. We have a natural injection $\iota\colon \cG_{\Prism}\to \cG_{\Prism}[1/\cI_{\Prism}]$. For a $\cG_{\Prism}$-torsor $\cP$, we set $\cP[1/\cI_{\Prism}]$ to be the $\cG_{\Prism}[1/\cI_{\Prism}]$-torsor $\iota_{*}\cP$.

\begin{dfn}
    \begin{enumerate}
        \item For a bounded prism $(A,I)$, we let $\mathrm{Twist}^{\varphi}_{(A,I)}(\Lambda_{0},\bT_{0})$ denote the groupoid of a pair $((\cE,\varphi_{\cE}),\bT_{\Prism})$ consisting of an object $(\cE,\varphi_{\cE})\in \mathrm{Vect}^{\varphi}(A,I)$ and tensors $\bT_{\Prism}\subset \cE^{\otimes}$ fixed by $\varphi_{\cE}$ such that $\underline{\mathrm{Isom}}((\Lambda_{0}\otimes_{\bZ_{p}} A,\bT_{0}\otimes 1),(\cE,\bT_{\Prism}))$ is a $\cG$-torsor on $\mathrm{Spf}(A)_{\mathrm{fpqc}}$.
        \item We let $\mathrm{Twist}^{\varphi}_{\cO_{\Prism}}(\Lambda_{0},\bT_{0})$ denote the groupoid of a pair $((\cE,\varphi_{\cE}),\bT_{\Prism})$ consisting of an object $(\cE,\varphi_{\cE})\in \mathrm{Vect}^{\varphi}((\fX,\cM_{\fX})_{\Prism})$ and tensors $\bT_{\Prism}\subset \cE^{\otimes}$ fixed by $\varphi_{\cE}$ such that $\underline{\mathrm{Isom}}((\Lambda_{0}\otimes_{\bZ_{p}} \cO_{\Prism},\bT_{0}\otimes 1),(\cE,\bT_{\Prism}))$ is a $\cG_{\Prism}$-torsor on $(\fX,\cM_{\fX})_{\Prism}$.
    \end{enumerate}
\end{dfn}

\begin{lem}
Let $(\fX,\cM_{\fX})$ be a bounded $p$-adic log formal scheme. Then there exists a natural equivalence
\[
\mathrm{Twist}^{\varphi}_{\cO_{\Prism}}(\Lambda_{0},\bT_{0})\simeq \varprojlim_{(A,I,\cM_{A})\in (\fX,\cM_{\fX})_{\Prism}} \mathrm{Twist}^{\varphi}_{(A,I)}(\Lambda_{0},\bT_{0}).
\]
\end{lem}

\begin{proof}
The equivalence in Remark \ref{pris crys as lim} gives a fully faithful functor
\[
\mathrm{Twist}^{\varphi}_{\cO_{\Prism}}(\Lambda_{0},\bT_{0})\to \varprojlim_{(A,I,\cM_{A})\in (\fX,\cM_{\fX})_{\Prism}} \mathrm{Twist}^{\varphi}_{(A,I)}(\Lambda_{0},\bT_{0}).
\]
To check essential surjectivity, it is enough to show that, for a pair $((\cE,\varphi_{\cE}),\bT_{\Prism})$ consisting of $(\cE,\varphi_{\cE})\in \mathrm{Vect}^{\varphi}((\fX,\cM_{\fX})_{\Prism})$ and tensors $\bT_{\Prism}\subset \cE^{\otimes}$ fixed by $\varphi_{\cE}$ such that 
\[
\cP_{(A,I,\cM_{A})}\coloneqq \underline{\mathrm{Isom}}((\Lambda_{0}\otimes_{\bZ_{p}} A,\bT_{0}\otimes 1),(\cE_{(A,I,\cM_{A})},\bT_{\Prism,(A,I,\cM_{A})}))
\]
is a $\cG$-torsor on $\mathrm{Spf}(A)_{\mathrm{fpqc}}$ for each log prism $(A,I,\cM_{A})\in (\fX,\cM_{\fX})_{\Prism}$, the object
\[
\cP\coloneqq \underline{\mathrm{Isom}}((\Lambda_{0}\otimes_{\bZ_{p}} \cO_{\Prism},\bT_{0}\otimes 1),(\cE,\bT_{\Prism}))
\]
is a $\cG_{\Prism}$-torsor. The smoothness of $\cG$ implies that $\cP_{(A,I,\cM_{A})}$ is a $\cG$-torsor on $\mathrm{Spf}(A)_{\et}$. Since every \'{e}tale covering $\mathrm{Spf}(B)\to \mathrm{Spf}(A)$ extends to an \'{e}tale covering $(A,I,\cM_{A})\to (B,IB,\cM_{B})$ in $(\fX,\cM_{\fX})_{\Prism}$ by \cite[Lemma 2.13]{kos22}, $\cP$ is a $\cG_{\Prism}$-torsor.
\end{proof}

\begin{construction}\label{construction pris twist to g-pris crys}
For a bounded $p$-adic log formal scheme $(\fX,\cM_{\fX})$, we define a functor 
\[
\mathrm{Twist}^{\varphi}_{\cO_{\Prism}}(\Lambda_{0},\bT_{0})\to \cG\text{-}\mathrm{Vect}^{\varphi}((\fX,\cM_{\fX})_{\Prism})
\]
as follows. Let $((\cE,\varphi_{\cE}),\bT_{\Prism})\in \mathrm{Twist}^{\varphi}_{\cO_{\Prism}}(\Lambda_{0},\bT_{0})$. We set a $\cG_{\Prism}$-torsor 
\[
\cP\coloneqq \underline{\mathrm{Isom}}((\Lambda_{0}\otimes_{\bZ_{p}} \cO_{\Prism},\bT_{0}\otimes 1),(\cE,\bT_{\Prism})).
\]
Let $\omega_{\cP}$ denote the object of $\cG\text{-}\mathrm{Vect}((\fX,\cM_{\fX})_{\Prism})$ defined by $\omega_{\cP}(\Lambda)\coloneqq \cP\wedge^{\cG_{\Prism}} (\Lambda\otimes \cO_{\Prism})$.
The isomorphism $\varphi_{\cE}\colon (\phi^{*}\cE)[1/\cI_{\Prism}]\isom \cE[1/\cI_{\Prism}]$ induces isomorphisms of $\cG_{\Prism}[1/\cI_{\Prism}]$-torsors
\begin{align*}
    (\phi_{*}\cP)[1/\cI_{\Prism}]&\cong \underline{\mathrm{Isom}}((\Lambda_{0}\otimes \cO_{\Prism},\bT_{0}\otimes 1),((\phi^{*}\cE)[1/I],\phi^{*}\bT_{\Prism})) \\
    &\isom \underline{\mathrm{Isom}}((\Lambda_{0}\otimes \cO_{\Prism},\bT_{0}\otimes 1),(\cE[1/I],\bT_{\Prism})) \\
    &\cong \cP[1/\cI_{\Prism}],
\end{align*}
which gives isomorphisms of $\cO_{\Prism}[1/\cI_{\Prism}]$-modules
\begin{align*}
(\phi^{*}\omega_{\cP}(\Lambda))[1/\cI_{\Prism}]&\cong (\phi_{*}\cP)[1/\cI_{\Prism}]\wedge^{\cG_{\Prism}[1/\cI_{\Prism}]} (\Lambda\otimes \cO_{\Prism}[1/\cI_{\Prism}]) \\
&\isom \cP[1/\cI_{\Prism}]\wedge^{\cG_{\Prism}[1/\cI_{\Prism}]} (\Lambda\otimes \cO_{\Prism}[1/\cI_{\Prism}]) \\
&\cong \omega_{\cP}(\Lambda)[1/\cI_{\Prism}].
\end{align*}     
\end{construction}

\begin{dfn}
    Let $\mathrm{Loc}^{\Prism\text{-gr}}_{\bZ_{p}}((\fX,\cM_{\fX})_{\eta})$ be the essential image of the \'{e}tale realization functor
    $T_{\et}:\mathrm{Vect}^{\varphi}((\fX,\cM_{\fX})_{\Prism})\to \mathrm{Loc}_{\bZ_{p}}((\fX,\cM_{\fX})_{\eta})$ (which is fully faithful by \cite[Corollary 6.28]{ino25}). Let $\mathrm{Loc}^{\Prism^{\mathrm{an}}\text{-gr}}_{\bZ_{p}}((\fX,\cM_{\fX})_{\eta})$ be the essential image of the \'{e}tale realization functor
    $T_{\et}\colon \mathrm{Vect}^{\mathrm{an},\varphi}((\fX,\cM_{\fX})_{\Prism})\to \mathrm{Loc}_{\bZ_{p}}((\fX,\cM_{\fX})_{\eta})$.
\end{dfn}

\begin{prop}\label{g-an pris crys to g-loc sys}
    Let $(\fX,\cM_{\fX})$ be a horizontally semi-stable log formal scheme over $\cO_{K}$. Then the functor induced the \'{e}tale realization 
    \[
    \cG\text{-}\mathrm{Vect}^{\mathrm{an},\varphi}((\fX,\cM_{\fX})_{\Prism})\to \cG\text{-}\mathrm{Loc}^{\Prism^{\mathrm{an}}\text{-gr}}_{\bZ_{p}}((\fX,\cM_{\fX})_{\eta})
    \]
    gives an equivalence.
\end{prop}

\begin{proof}
    The assertion follows from the fact that the \'{e}tale realization functor
    \[
    T_{\et}\colon \mathrm{Vect}^{\varphi}((\fX,\cM_{\fX})_{\Prism})\to \mathrm{Loc}_{\bZ_{p}}((\fX,\cM_{\fX})_{\eta})
    \]
    gives a bi-exact equivalence to the essential image \cite[Theorem 6.27 and Theorem 6.32]{ino25}.
\end{proof}

\begin{thm}\label{G-pris crys eq to G-gr loc sys}
    Let $(\fX,\cM_{\fX})$ be a horizontally semi-stable log formal scheme over $\cO_{K}$. Suppose that $\cG$ is reductive. Then the functor induced the \'{e}tale realization 
    \[
    \cG\text{-}\mathrm{Vect}^{\varphi}((\fX,\cM_{\fX})_{\Prism})\to \cG\text{-}\mathrm{Loc}^{\Prism\text{-gr}}_{\bZ_{p}}((\fX,\cM_{\fX})_{\eta})
    \]
    gives an equivalence.
\end{thm}

\begin{proof}[proof of Theorem \ref{G-pris crys eq to G-gr loc sys}]

We follow the argument of the proof of \cite[Theorem 2.28]{iky24}. By working \'{e}tale locally on $\fX$, we may assume that $(\fX,\cM_{\fX})$ is small affine log formal scheme with a fixed framing. Since $T_{\et}\colon \mathrm{Vect}^{\varphi}((\fX,\cM_{\fX})_{\Prism})\to \mathrm{Loc}^{\Prism\text{-gr}}_{\bZ_{p}}((\fX,\cM_{\fX})_{\eta})$ is fully faithful by \cite[Corollary 6.28]{ino25}, the functor in the statement is also fully faithful. We shall prove the essential surjectivity. Let $\omega_{\et}\in \cG\text{-}\mathrm{Loc}^{\Prism\text{-gr}}_{\bZ_{p}}((\fX,\cM_{\fX})_{\eta})$. By Proposition \ref{fund thm of tann framework for ket loc sys}, we have the corresponding object $(\omega_{\et}(\Lambda_{0}),\bT_{\et})\in \mathrm{Twist}_{\widehat{\bZ}_{p}}(\Lambda_{0},\bT_{0})$ where $\omega_{\et}(\Lambda_{0})$ is a $\bZ_{p}$-local system on $(\fX,\cM_{\fX})_{\eta}$ and tensors $\bT_{\et}\subset \omega_{\et}(\Lambda_{0})^{\otimes}$ is induced from $\bT_{0}\subset \Lambda_{0}^{\otimes}$. Let $(\cE,\varphi_{\cE})$ be a prismatic $F$-crystal on $(\fX,\cM_{\fX})$ with $T_{\et}(\cE,\varphi_{\cE})\cong \omega_{\et}(\Lambda_{0})$. Since $T_{\et}$ is fully faithful, the tensors $\bT_{\et}\subset \omega_{\et}(\Lambda_{0})^{\otimes}$ come from unique tensors $\bT_{\Prism}\subset \cE^{\otimes}$ which is fixed by $\varphi_{\cE}$. 

We shall show that $((\cE,\varphi_{\cE}),\bT_{\Prism})$ belongs to $\mathrm{Twist}^{\varphi}_{\cO_{\Prism}}(\Lambda_{0},\bT_{0})$. To achieve this, it suffices to prove that 
\[
\cP\coloneqq\underline{\mathrm{Isom}}((\Lambda_{0}\otimes_{\bZ_{p}} \cO_{\Prism},\bT_{0}\otimes 1),(\cE,\bT_{\Prism}))
\]
is a $\cG_{\Prism}$-torsor. Proposition \ref{g-an pris crys to g-loc sys} implies that 
\[
\cP_{(A,I,\cM_{A})}\coloneqq\underline{\mathrm{Isom}}((\Lambda_{0}\otimes_{\bZ_{p}} A,\bT_{0}\otimes 1),(\cE_{(A,I,\cM_{A})},\bT_{\Prism}))
\]
is a $\cG$-torsor on $U(A,I)_{\et}$ for every $(A,I,\cM_{A})\in (\fX,\cM_{\fX})_{\Prism}$. Since $\fS_{R}$ is normal and $U(\fS_{R},(E))$ contains all points of codimension $1$ in $\mathrm{Spec}(\fS_{R})_{\et}$, we see that $\cP_{(\fS_{R},(E),\cM_{\fS_{R}})}$ is a $\cG$-torsor on $\mathrm{Spec}(\fS_{R})$ by \cite[Proposition A.26 or Remark A.27]{iky24}. Hence, $\cP_{(A,I,\cM_{A})}$ is a $\cG$-torsor on $\mathrm{Spf}(A)_{\mathrm{fpqc}}$ for every log prism $(A,I,\cM_{A})\in (\fX,\cM_{\fX})_{\Prism}$ by \cite[Lemma 5.12]{ino25}. Therefore, $\cP$ is a $\cG_{\Prism}$-torsor. 

Since the \'{e}tale realization functor sends $((\cE,\varphi_{\cE}),\bT_{\Prism})$ to $(\omega_{\et}(\Lambda_{0}),\bT_{\et})$, the object $\omega_{\Prism}\in \cG\text{-}\mathrm{Vect}^{\varphi}((\fX,\cM_{\fX})_{\Prism})$ which comes from $((\cE,\varphi_{\cE}),\bT_{\Prism})\in \mathrm{Twist}^{\varphi}_{\cO_{\Prism}}(\Lambda_{0},\bT_{0})$ via the functor defined in Construction \ref{construction pris twist to g-pris crys} is sent to $\omega_{\et}$ by Proposition \ref{comm triangle for g-loc sys}. This completes the proof.
\end{proof}

\section{Application to Shimura varieties}

\subsection{Log prismatic realization on the toroidal compactifications of Shimura varieties}

Throughout this section, let $p>2$ be a prime number.

\begin{setting}
    Define the following notation:
    \begin{itemize}
        \item $(G,X)$ is a Shimura datum of Hodge type with reflex field $E$,
        \item $v$ is a place of $E$ above $p$, and $E_{v}$ denotes the completion of $E$ at $v$,
        \item $\cO_{E_{v}}$ is the integer ring of $E_{v}$ with residue field $k_{v}$,
        \item $\cG$ is a reductive model of $G_{\bQ_{p}}$ over $\bZ_{p}$,
        \item $K_{p}\coloneqq\cG(\bZ_{p})\subset G(\bQ_{p})$.
    \end{itemize}
\end{setting}

Let $\mathrm{Sh}_{K}(G,X)$ denote the associated Shimura variety over $E$. By \cite[3.3.1]{kim18}, we can choose an embedding of Shimura datum $(G,X)\hookrightarrow (\mathrm{GSp}(V),S^{\pm})$ and a self-dual $\bZ_{(p)}$-lattice $\Lambda_{\bZ_{(p)}}\subset V$ such that $\Lambda\coloneqq \Lambda_{\bZ_{(p)}}\otimes_{\bZ_{(p)}} \bZ_{p}\subset V_{\bQ_{p}}$ is $K_{p}$-stable. Let $K'_{p}\coloneqq \mathrm{GSp}(\Lambda)$. Let $K'^{p}\subset \mathrm{GSp}(V)(\bA_{f}^{p})$ be a neat open compact subgroup containing $K^{p}$ and let $K'\coloneqq K'_{p}K'^{p}$. Then the Shimura variety $\mathrm{Sh}_{K'}(\mathrm{GSp}(V),S^{\pm})_{\bQ_{p}}$ has an integral model $\sS_{K'}(\mathrm{GSp}(V),S^{\pm})$ over $\bZ_{p}$ defined as the moduli space of polarized abelian varieties with $K'$-level structures. Let $K^{p}$ be a open compact subgroup of $G(\bA_{f}^{p})$ contained in $K'^{p}$. Let $\sS_{K_{p}K^{p}}(G,X)$ denote the normalization of the scheme theoretic image of 
\[
\mathrm{Sh}_{K_{p}K^{p}}(G,X)_{E_{v}}\to \mathrm{Sh}_{K'}(\mathrm{GSp}(V),S^{\pm})_{E_{v}}\hookrightarrow \sS_{K'}(\mathrm{GSp}(V),S^{\pm})_{\cO_{E_{v}}}.
\]
Then the tower $\{\sS_{K_{p}K^{p}}(G,X)\}_{K^{p}}$ is an integral canonical model of $\mathrm{Sh}_{K}(G,X)_{E_{v}}$, which is smooth over $\cO_{E_{v}}$ (\cite[Theorem 2.3.8]{kis10}). Fix such a $K^{p}$ and set $K\coloneqq K_{p}K^{p}$.

Fix a smooth projective admissible rational cone decomposition $\Sigma'$ for $((\mathrm{GSp}(V),S^{\pm}),K')$ (\cite[2.1.23]{mad19}).
Let $\Sigma''$ denote the induced admissible rational cone decomposition for $((G,X),K)$ from $\Sigma'$, and fix a smooth projective admissible rational cone decomposition $\Sigma$ for $((G,X),K)$ refining $\Sigma''$ (\cite[2.1.28]{mad19}). Let $\sS_{K'}^{\Sigma'}(\mathrm{GSp}(V),S^{\pm})$ be the toroidal compactification of the Siegel modular variety with level $K'$ constructed in \cite{fc90}. Let $\sS_{K}^{\Sigma''}(G,X)$ denote the normalization of $\sS_{K'}^{\Sigma'}(\mathrm{GSp}(V),S^{\pm})$ in $\mathrm{Sh}_{K}(G,X)$, which is the toroidal compactification of $\sS_{K}(G,X)$ constructed in \cite{mad19}. By the technique of toroidal embeddings (\cite[Chapter II]{kkmsd}), we have an integral model $\sS_{K}^{\Sigma}(G,X)$ of $\mathrm{Sh}_{K}^{\Sigma}(G,X)$ over $\cO_{E_{v}}$ and a birational morphism $\sS_{K}^{\Sigma}(G,X)\to \sS_{K}^{\Sigma''}(G,X)$
(see also \cite[4.1.4 and Remark 4.1.6]{mad19}). As a result, we obtain the following diagram:
\[
\begin{tikzcd}
    \mathrm{Sh}_{K}(G,X)_{E_{v}} \ar[r,hook] \ar[d,hook] & \mathrm{Sh}_{K'}(\mathrm{GSp}(V),S^{\pm})_{E_{v}} \ar[d,hook] \\
    \mathrm{\sS}_{K}(G,X) \ar[r] \ar[d,hook] & \mathrm{\sS}_{K'}(\mathrm{GSp}(V),S^{\pm})_{\cO_{E_{v}}} \ar[d,hook] \\
    \mathrm{\sS}_{K}^{\Sigma}(G,X) \ar[r] & \mathrm{\sS}_{K'}^{\Sigma'}(\mathrm{GSp}(V),S^{\pm})_{\cO_{E_{v}}}.
\end{tikzcd}
\]
Here, the vertical morphisms are open immersions.

We write $\sS_{K}^{\Sigma}$, $\sS_{K'}^{\Sigma'}$ for $\sS_{K}^{\Sigma}(G,X)$, $\sS_{K'}^{\Sigma'}(\mathrm{GSp}(V),S^{\pm})$ respectively. We equip $\sS_{K}^{\Sigma}$ (resp. $\sS_{K'}^{\Sigma'}$) with log structures $\cM_{\sS_{K}^{\Sigma}}$ (resp. $\cM_{\sS_{K'}^{\Sigma'}}$) defined from relatively normal crossings boundary divisors. Note that the $p$-completions of them are horizontally semi-stable log formal schemes by the smoothness and projectivity of $\Sigma$ and $\Sigma'$.

In \cite{kkn21} and \cite[Theorem 2.2.2, Proposition 4.3.4]{kkn22}, the toroidal compactification $\sS_{K'}^{\Sigma'}$ is realized as the moduli space of principally polarized log abelian varieties with $K'$-level structure. In particular, the universal abelian scheme over the interior extends to the universal log abelian varieties $\cA_{\mathrm{univ}}^{\mathrm{log}}$ over $(\sS_{K'}^{\Sigma'},\cM_{\sS_{K'}^{\Sigma'}})$. Taking $p$-power torsion points, we obtain a log $p$-divisible group $\cA_{\mathrm{univ}}^{\mathrm{log}}[p^{\infty}]$ by \cite[Proposition 18.1]{kkn15} and \cite[Proposition 4.5]{kat23}. The log $p$-divisible group $\cA_{\mathrm{univ}}^{\mathrm{log}}[p^{\infty}]$ is pulled back to $\sS_{K}^{\Sigma}$ along
\[
(\sS_{K}^{\Sigma},\cM_{\sS_{K}^{\Sigma}})\to (\sS_{K'}^{\Sigma'},\cM_{\sS_{K'}^{\Sigma'}}).
\]

\subsection{Log prismatic realization}

First, we recall the \'{e}tale realization on Shimura varieties (cf. \cite[\S 3.3]{iky23}). By varying levels at $p$, we obtain a tower of Shimura varieties
\[
\varprojlim_{K'_{p}\subset K_{p}} \mathrm{Sh}_{K'_{p}K^{p}}(G,X)\to \mathrm{Sh}_{K}(G,X),
\]
which is a pro-\'{e}tale $K_{p}$-torsor after taking the analytification. Taking the pushforward along $K_{p}\to \cG_{\widehat{\bZ}_{p}}$, we obtain a $\bZ_{p}$-local system on $\mathrm{Sh}_{K}(G,X)^{\mathrm{an}}$ with $\cG$-structure $\omega_{\et}\in \cG\text{-}\mathrm{Loc}_{\bZ_{p}}(\mathrm{Sh}_{K}(G,X)^{\mathrm{an}})$ via Proposition \ref{comm triangle for g-loc sys}. 

By purity for Kummer \'{e}tale local systems (\cite[Theorem 4.6.1]{dllz23b}), there exists a bi-exact equivalence
\[
\mathrm{Loc}_{\bZ_{p}}(\mathrm{Sh}_{K}(G,X)^{\mathrm{an}})\simeq \mathrm{Loc}_{\bZ_{p}}(\mathrm{Sh}_{K}^{\Sigma}(G,X)^{\mathrm{an}},\cM_{\mathrm{Sh}_{K}^{\Sigma}(G,X)^{\mathrm{an}}}).
\]
Hence, $\omega_{\et}$ gives a unique Kummer \'{e}tale $\bZ_{p}$-local system with $\cG$-structure
\[
\omega_{\mathrm{k\et}}\colon \mathrm{Rep}_{\bZ_{p}}(\cG)\to \mathrm{Loc}_{\bZ_{p}}(\mathrm{Sh}_{K}^{\Sigma}(G,X)^{\mathrm{an}},\cM_{\mathrm{Sh}_{K}^{\Sigma}(G,X)^{\mathrm{an}}}).
\] 

\begin{thm}\label{log pris realization}
    There exists a unique log prismatic $F$-crystal with $G$-structure 
    \[
    \omega_{\Prism,\mathrm{log}}\colon \mathrm{Rep}_{\bZ_{p}}(\cG)\to \mathrm{Vect}^{\varphi}((\widehat{\sS}_{K}^{\Sigma}(G,X),\cM_{\widehat{\sS}_{K}^{\Sigma}(G,X)})_{\Prism})
    \]
    with an isomorphism $T_{\et}\circ \omega_{\Prism,\mathrm{log}}\cong \omega_{\mathrm{k\et}}$.
\end{thm}

\begin{proof}
    By Theorem \ref{G-pris crys eq to G-gr loc sys} and \cite[Proposition 12]{ds09}, it suffices to prove that there exists a log prismatic $F$-crystal $\cE$ on $(\widehat{\sS}_{K}^{\Sigma}(G,X),\cM_{\widehat{\sS}_{K}^{\Sigma}(G,X)})$ with  $T_{\et}(\cE)\cong \omega_{\mathrm{k\et}}(\Lambda)$. Since $\omega_{\mathrm{k\et}}(\Lambda)$ is characterized as a unique extension of $\omega_{\et}(\Lambda)\cong \sA_{\mathrm{univ}}[p^{\infty}]$, we have an isomorphism $\omega_{\mathrm{k\et}}(\Lambda)\cong \sA_{\mathrm{univ}}^{\mathrm{log}}[p^{\infty}]$. Therefore, Proposition \ref{et comp for log pris dieudonne} implies that the log prismatic Dieudonn\'{e} crystal $\cM_{\Prism}(\sA_{\mathrm{univ}}^{\mathrm{log}}[p^{\infty}])$ is the desired one.
\end{proof}

\begin{rem}
    In \cite[Theorem 3.17]{iky23}, a prismatic $F$-crystal with $\cG$-structure $\omega_{\Prism}\in \cG\text{-}\mathrm{Vect}^{\varphi}(\widehat{\sS}_{K}(G,X)_{\Prism})$ with $T_{\et}\circ \omega_{\Prism}\cong \omega_{\et}|_{\widehat{\sS}_{K}(G,X)_{\eta}}$ is constructed. We see that $\omega_{\Prism,\mathrm{log}}$ constructed in Theorem \ref{log pris realization} restricts to $\omega_{\Prism}$ by fully faithfulness of \'{e}tale realization functor.
\end{rem}

\begin{rem}
    In \cite{lov17}, a strongly divisible filtered $F$-crystal with $\cG$-structure 
    \[
    \omega_{\mathrm{crys}}\colon \mathrm{Rep}_{\bZ_{p}}(\cG)\to \mathrm{FilCrys}^{\varphi}(\widehat{\sS}_{K}(G,X))
    \]
    is constructed, and it is proved that $\omega_{\et}$ and $\omega_{\mathrm{crys}}$ are associated in the sense of Fontaine-Laffaille. As a relationship with the prismatic theory, a full subcategory $\mathrm{Vect}^{\varphi,\mathrm{lff}}(\widehat{\sS}_{K}(G,X)_{\Prism})$ of $\mathrm{Vect}^{\varphi}(\widehat{\sS}_{K}(G,X)_{\Prism})$ is defined in \cite[Definition 1.24]{iky23}, an integral crystalline realization functor
    \[
    \bD_{\mathrm{crys}}\colon \mathrm{Vect}^{\varphi}(\widehat{\sS}_{K}(G,X)_{\Prism})\to  \mathrm{FilCrys}^{\varphi}(\widehat{\sS}_{K}(G,X))
    \]
    is constructed in \cite[\S 2.2]{iky23}, and it is proved that the functor $\omega_{\Prism}$ takes values in $\mathrm{Vect}^{\varphi,\mathrm{lff}}(\widehat{\sS}_{K}(G,X)_{\Prism})$ and that there exists an isomorphism $\bD_{\mathrm{crys}}\circ \omega_{\Prism}\cong \omega_{\mathrm{crys}}$ in \cite[Theorem 3.20]{iky23}.
    
    The generalization of this picture to our setting remains as a future work. It seems that a key step is to show the boundedness of $\omega_{\Prism,\mathrm{log}}$. 
\end{rem}

\subsection{Lovering's Conjecture}

Our work on toroidal compactifications enables us to generalize the comparison isomorphism for $\bQ_{p}$-local systems coming from $\omega_{\et}$, which is proved in the case that the Shimura variety is proper in \cite[Theorem 3.6.1]{lov17} or \cite[Proposition 3.23]{iky23}.

\begin{thm}\label{comparison isom on shimura var}
    Let $(\overline{\sS}_{K}^{\Sigma}(G,X),\cM_{\overline{\sS}_{K}^{\Sigma}(G,X)})\coloneqq (\sS_{K}^{\Sigma}(G,X),\cM_{\sS_{K}^{\Sigma}(G,X)})\otimes_{\cO_{E_{v}}} k$, and consider a filtered $F$-isocrystal $\omega_{\mathrm{fisoc}}\coloneqq T_{\mathrm{fisoc}}\circ \omega_{\Prism,\mathrm{log}}$ on $(\widehat{\sS}_{K}^{\Sigma}(G,X),\cM_{\widehat{\sS}_{K}^{\Sigma}(G,X)})$. Then the $\mathrm{Gal}(\overline{E_{v}}/E_{v})$-representation $H^{i}_{\et}(\mathrm{Sh}_{K}(G,X)_{\overline{E}_{v}},\omega_{\et}(\xi)[1/p])$ is crystalline, and we have an isomorphism
    \begin{align*}
    &H^{i}_{\et}(\mathrm{Sh}_{K}(G,X)_{\overline{E}_{v}},\omega_{\et}(\xi)[1/p])\otimes_{\bQ_{p}} B_{\mathrm{crys}} \\
    &\cong H^{i}_{\mathrm{logcrys}}((\overline{\sS}_{K}^{\Sigma}(G,X),\cM_{\overline{\sS}_{K}^{\Sigma}(G,X)})_{\mathrm{crys}},\omega_{\mathrm{fisoc}}(\xi))\otimes_{W(k_{v})} B_{\mathrm{crys}}  
    \end{align*}
    that is compatible with Galois actions and Frobenius isomorphisms and a filtered isomorphism after taking the base change along $B_{\mathrm{crys}}\to B_{\mathrm{dR}}$,
    for any $\xi\in \mathrm{Rep}_{\bZ_{p}}(\cG)$.
\end{thm}

\begin{proof}
    By \cite[Theorem 6.18]{ino25}, the $\bQ_{p}$-local system $\omega_{\mathrm{k\et}}(\xi)[1/p]$ is associated with $\omega_{\mathrm{fisoc}}(\xi)$ in the sense of \cite[Definition 4.18(2)]{ino25}. Then the assertion follows from Faltings' $p$-adic Hodge theory \cite[Theorem 6.3]{fal89}.
\end{proof}

\begin{rem}
    This comparison isomorphism is conjectured in \cite[\S 3.6.2]{lov17}. Furthermore, in \emph{loc. cit.}, it is also conjectured that there exists an integral version of $\omega_{\mathrm{fisoc}}$ on the toroidal compactification, which is still open.
\end{rem}

\begin{rem}
    The fact that the $\mathrm{Gal}(\overline{E}_{v}/E_{v})$-representation $H^{i}_{\et}(\mathrm{Sh}_{K}(G,X)_{\overline{E}_{v}},\omega_{\et}(\xi)[1/p])$ is crystalline can be deduced more easily as follows. We may assume that $\xi=\Lambda^{\otimes n}$, and so it is enough to show that the $\mathrm{Gal}(\overline{E}_{v}/E_{v})$-representation 
    \[
    H^{i}_{\et}(\mathrm{Sh}_{K}(G,X)_{\overline{E}_{v}},R^{n}f_{\et,*}\bQ_{p})
    \]
    is crystalline, where $f$ is the structure morphism of the $n$-th fold self product of the universal abelian scheme. We have a spectral sequence
    \[
    E_{2}^{p,q}=H^{p}_{\et}(\mathrm{Sh}_{K}(G,X)_{\overline{E}_{v}},R^{q}f_{\et,*}\bQ_{p})\Rightarrow H^{p+q}_{\et}(\sA^{n}_{\mathrm{univ},\overline{E}_{v}},\bQ_{p}).
    \]
    Since the multiplication by $m$ on $\sA_{\mathrm{univ},\overline{E}_{v}}$ induces the multiplication by $m^{nq}$ on $E_{2}^{p,q}$, this spectral sequence degenerates at the $E_{2}$-page. Hence, it suffices to prove that the $\mathrm{Gal}(\overline{E}_{v}/E_{v})$-representation $H^{i}_{\et}(\sA^{n}_{\mathrm{univ},\overline{E}_{v}},\bQ_{p})$ is crystalline, which follows from Faltings' $p$-adic Hodge theory and the existence of a smooth compactification of $\sA^{n}_{\mathrm{univ},\cO_{E_{v}}}$ with a relatively normal crossings boundary divisor over $\cO_{E_{v}}$ (\cite[Theorem 2.15]{lan12} ). We learned this argument from Kai-Wen Lan.
    
    Clearly, this method gives no explicit description of the filtered $F$-isocrystal
    \[
    D_{\mathrm{crys}}(H^{i}_{\et}(\mathrm{Sh}_{K}(G,X)_{\overline{E}_{v}},\omega_{\et}(\xi)[1/p])).
    \]
\end{rem}

\appendix

\section{Some lemmas on log schemes}

\begin{lem}\label{free monoid to free monoid}
    Let $P$ and $Q$ be a finitely generated free monoids (i.e. $P\simeq \bN^{r}$ and $Q\simeq \bN^{s}$ for some $r,s\geq 0$). Let $\alpha: P\to Q$ be a homomorphism inducing $P/\alpha^{-1}(0)\isom Q$. Then, there exist a basis $p_{1},\dots,p_{r}$ (resp. $q_{1},\dots,q_{s}$) of $P$ (resp. $Q$) such that 
    \[
    \alpha(p_{i})=
    \begin{cases}
        q_{i} \ (1\leq i\leq s) \\
        0 \ (s+1\leq i\leq r).
    \end{cases}
    \]
\end{lem}

\begin{proof}
    Choose a basis $p_{1},\dots,p_{r}$ (resp. $q_{1},\dots,q_{s}$) of $P$ (resp. $Q$). Since $\alpha^{-1}(0)$ is a face of $P$, the submonoid $\alpha^{-1}(0)$ is generated by $\{p_{i}| i\in I\}$ for some $I\in \{1,\dots, r\}$. For each $1\leq j\leq s$, there exist an integer $1\leq i_{j}\leq r$ such that $\alpha(p_{i_{j}})=q_{j}$ by the surjectivity of $\alpha$ and the irreducibility of $q_{j}$.

    By renumbering, we may assume that there exists an integer $t$ with $1\leq s\leq t\leq r$ such that $\alpha(p_{i})=q_{i}$ for $1\leq i\leq s$ and $\alpha^{-1}(0)$ is generated by $p_{s+1},\cdots, p_{t}$. By $P^{gp}/\alpha^{-1}(0)^{gp}\isom Q^{gp}$ and the computation of ranks, we conclude that $s=t$.
\end{proof}

\begin{lem}\label{sat prod of two pro-kfl cover}
    We use the notation in Setting \ref{setting for pro-kfl descent}. Let $(\fX,\cM_{\fX})$ be an fs log formal scheme. Let $n\geq 2$. Suppose that we are given fs charts $\alpha_{i}\colon P_{i}\to \cM_{\fX}$ for each $1\leq i\leq n$. For a non-empty subset $I\subset \{1,\dots,n\}$, we let $(\fX_{\infty,I},\cM_{\fX_{\infty,I}})$ denote the fiber product of $(\fX_{\infty,\alpha_{i}},\cM_{\fX_{\infty,\alpha_{i}}})$ for $i\in I$ over $(\fX,\cM_{\fX})$ in the category of saturated log formal schemes. Let $\emptyset \neq J\subset I\subset \{1,\dots,n\}$.
    \begin{enumerate}
        \item The projection map $(\fX_{\infty,I},\cM_{\fX_{\infty,I}})\to (\fX_{\infty,J},\cM_{\fX_{\infty,J}})$ is a strict affine flat cover.
        \item Suppose that $\fX$ is a bounded $p$-adic formal scheme and $P_{i}=\bN^{r_{i}}$ for each $1\leq i\leq n$. Then the projection map $(\fX_{\infty,I},\cM_{\fX_{\infty,I}})\to (\fX_{\infty,J},\cM_{\fX_{\infty,J}})$ is a strict affine quasi-syntomic cover.
        \item Suppose that $\fX$ is a scheme and $P_{i}=\bN^{r_{i}}$ for each $1\leq i\leq n$. Then the projection map $(\fX_{\infty,I},\cM_{\fX_{\infty,I}})\to (\fX_{\infty,J},\cM_{\fX_{\infty,J}})$ is a strict affine root covering.
    \end{enumerate}
\end{lem}

\begin{proof}
    We may assume $n=2$ because every projection map is written as the composition of maps which are the base change of projection maps in the $n=2$ case. (For example, in the $n=3$ case, a projection map
    $(\fX_{\infty,\{1,2,3\}},\cM_{\fX_{\infty,\{1,2,3\}}})\to (\fX_{\infty,\{3\}},\cM_{\fX_{\infty,\{3\}}})$ factors as
    \[
    (\fX_{\infty,\{1,2,3\}},\cM_{\fX_{\infty,\{1,2,3\}}})\to (\fX_{\infty,\{2,3\}},\cM_{\fX_{\infty,\{2,3\}}}) \to  (\fX_{\infty,\{3\}},\cM_{\fX_{\infty,\{3\}}}),
    \]
    where the projection map $(\fX_{\infty,\{1,2,3\}},\cM_{\fX_{\infty,\{1,2,3\}}})\to (\fX_{\infty,\{2,3\}},\cM_{\fX_{\infty,\{2,3\}}})$ is written as the base change of a projection map $(\fX_{\infty,\{1,2\}},\cM_{\fX_{\infty,\{1,2\}}}) \to (\fX_{\infty,\{2\}},\cM_{\fX_{\infty,\{2\}}})$.) By the symmetry, it suffice to prove the assertion for the projection map
    \[
    (\fX_{\infty,\{1,2\}},\cM_{\fX_{\infty,\{1,2\}}})\to (\fX_{\infty,\{2\}},\cM_{\fX_{\infty,\{2\}}}).
    \]

    First, we shall prove $(1)$. For any $n\geq 1$, there exists an integer $m\geq 1$ such that the projection map from the saturated fiber product
    \[ (\fX_{n,\alpha_{1}},\cM_{\fX_{n,\alpha_{1}}})\times_{(\fX,\cM_{\fX})} (\fX_{m,\alpha_{2}},\cM_{\fX_{m,\alpha_{2}}})\to (\fX_{m,\alpha_{2}},\cM_{\fX_{m,\alpha_{2}}})
    \]
    is a strict affine flat cover by Lemma \ref{kummer map is strict after n-power ext} and the fact that a saturation map for a quasi-coherent log formal scheme is affine. Therefore, the projection from the saturated fiber product
    \[ (\fX_{n,\alpha_{1}},\cM_{\fX_{n,\alpha_{1}}})\times_{(\fX,\cM_{\fX})} (\fX_{\infty,\alpha_{2}},\cM_{\fX_{\infty,\alpha_{2}}})\to (\fX_{\infty,\alpha_{2}},\cM_{\fX_{\infty,\alpha_{2}}})
    \]
    is a strict affine flat cover for any $n\geq 1$. Taking a limit with respect to $n\geq 1$, we obtain the claim.

    Next, we shall prove $(2)$ and $(3)$ simultaneously. Take $x\in \fX$. Applying Lemma \ref{free monoid to free monoid} to maps $\alpha_{1,x}:\bN^{r_{1}}\to \overline{\cM_{\fX,\overline{x}}}$ and $\alpha_{2,x}:\bN^{r_{2}}\to \overline{\cM_{\fX,\overline{x}}}$ induced by $\alpha_{1}$ and $\alpha_{2}$, we see that, after changing orders of bases of $\bN^{r_{1}}$ and $\bN^{r_{2}}$ appropriately, there exists an integer $s$ with $1\leq s\leq \mathrm{min}\{r_{1},r_{2}\}$ such that $\alpha_{1,x}(e_{i})=\alpha_{2,x}(e_{i})$ for every $1\leq i\leq s$ and $\alpha_{1,x}(e_{j})=\alpha_{2,x}(e_{k})=1$ for every $s< j\leq r_{1}$ and $s<k\leq r_{2}$. By shrinking $\fX$ to an \'{e}tale neighborhood of $x$, we may assume that we can write
    \begin{align*}
    \alpha_{1}(e_{i})&=u_{i}\alpha_{2}(e_{i})\ \ (1\leq i\leq s) \\
    \alpha_{1}(e_{j})&=v_{j}\ \ (s<j\leq r_{1}) \\
    \alpha_{2}(e_{k})&=w_{k}\ \ (s<k\leq r_{2})
    \end{align*}
    for $u_{i},v_{j},w_{k}\in \cO_{\fX}(\fX)^{\times}$. Then $\fX_{\infty,\{1,2\}}$ is obtained by adding $n$-th power roots of $u_{i}$ and $v_{j}$ to $\fX_{\infty,\alpha_{2}}$. We shall give a more precise explanation. Consider a map $\fX_{\infty,\alpha_{2}}\to \mathrm{Spf}(\bZ_{p}\langle x_{1},\dots,x_{r_{1}}\rangle)$ corresponding to a ring map $\bZ_{p}\langle x_{1},\dots,x_{r_{1}}\rangle\to \cO_{\fX}(\fX)$ mapping $x_{i}$ to $u_{i}$ or $v_{i}$ depending on whether $i\leq s$ or not. Let
    \[ \fX_{\infty,\alpha_{1},\alpha_{2}}\coloneqq\fX_{\alpha_{2}}\times_{\mathrm{Spf}(\bZ_{p}\langle x_{1},\dots,x_{r_{1}}\rangle)} \mathrm{Spf}(\bZ_{p}\langle x_{1}^{\bQ_{\geq 0}},\dots,x_{r_{1}}^{\bQ_{\geq 0}}\rangle)
    \]
    and the log structure $\cM_{\fX_{\infty,\alpha_{1},\alpha_{2}}}$ on $\fX_{\infty,\alpha_{1},\alpha_{2}}$ be the inverse image log structure of $\cM_{\fX_{\infty,\alpha_{2}}}$. We have a strict morphism $(\fX_{\infty,\alpha_{1},\alpha_{2}},\cM_{\fX_{\infty,\alpha_{1},\alpha_{2}}})\to (\fX_{\infty,\alpha_{1}},\cM_{\fX_{\infty,\alpha_{1}}})$ whose composition with $(\fX_{\infty,\alpha_{1}},\cM_{\fX_{\infty,\alpha_{1}}})\to (\mathrm{Spf}(\bZ_{p}\langle \bQ^{r_{1}}_{\geq 0}\rangle),\bQ^{r_{1}}_{\geq 0})$ corresponds to a monoid map $\bQ^{r_{1}}_{\geq 0}\to \Gamma(\fX_{\infty,\alpha_{1},\alpha_{2}},\cM_{\fX_{\infty,\alpha_{1},\alpha_{2}}})$ mapping $(1/n)e_{i}$ to $x_{i}^{1/n}\alpha_{2}((1/n) e_{i})$ or $x_{i}^{1/n}$ for every $n\geq 1$ and $i$ depending on whether $i\leq s$ or not. Here, $\alpha_{2}((1/n) e_{i})$ is the image of $(1/n)e_{i}$ by the map $\bQ_{\geq 0}^{r_{2}}\to \Gamma(\fX_{\infty,\alpha_{1},\alpha_{2}},\cM_{\fX_{\infty,\alpha_{1},\alpha_{2}}})$ corresponding to the morphism
    \[ (\fX_{\infty,\alpha_{1},\alpha_{2}},\cM_{\fX_{\infty,\alpha_{1},\alpha_{2}}})\to (\fX_{\infty,\alpha_{2}},\cM_{\fX_{\infty,\alpha_{2}}})\to (\mathrm{Spf}(\bZ_{p}\langle \bQ^{r_{2}}_{\geq 0}\rangle),\bQ^{r_{2}}_{\geq 0})^{a}.
    \]
    Then we can check directly that $(\fX_{\infty,\alpha_{1},\alpha_{2}},\cM_{\fX_{\infty,\alpha_{1},\alpha_{2}}})$ is the fiber product of $(\fX_{\infty,\alpha_{1}},\cM_{\fX_{\infty,\alpha_{1}}})$ and $(\fX_{\infty,\alpha_{2}},\cM_{\fX_{\infty,\alpha_{2}}})$ over $(\fX,\cM_{\fX})$, and this description of $(\fX_{\infty,\{1,2\}},\cM_{\fX_{\infty,\{1,2\}}})$ implies the assertion $(2)$ and $(3)$ for the projection map
    \[ 
    (\fX_{\infty,\{1,2\}},\cM_{\fX_{\infty,\{1,2\}}})\to (\fX_{\infty,\alpha_{2}},\cM_{\fX_{\infty,\alpha_{2}}}).
    \]
\end{proof}

\begin{lem}\label{modify chart of sch to chart of mor}
    Let $f\colon (\fX,\cM_{\fX})\to (\fY,\cM_{\fY})$ be a morphism of fs log formal schemes. Suppose that we are given a chart $\alpha\colon \bN^{r}\to \cM_{\fX}$ and a monoid map $\beta\colon \bN^{s}\to \cM_{\fY}$. Let $x\in \fX$. Then, after shrinking $\fX$ to an \'{e}tale neighborhood of $x$, there exist a chart $\alpha'\colon \bN^{r}\oplus \bZ^{s}\to \cM_{\fX}$ which restricts to $\alpha$ and a monoid map $\gamma\colon \bN^{s}\to \bN^{r}\oplus \bZ^{s}$ such that $g$ gives a chart of $f$.
\end{lem}

\begin{proof}
    For each $1\leq i\leq s$, we can write the image of $\beta(e_{i})$ in $\cM_{\fX,\overline{x}}$ as $f(m_{i})u_{i}$ for some $m_{i}\in \bN^{r}$ and $u_{i}\in \cO_{\fX,\overline{x}}^{\times}$ because $\alpha$ is a chart. By shrinking $\fX$ to an \'{e}tale neighborhood of $x$, we may assume that $u_{i}\in \Gamma(\fX,\cO_{\fX}^{\times})$ for every $1\leq i\leq s$. Let $\{e'_{i}\}_{1\leq i\leq s}$ denote the standard basis of $\bZ^{s}$. We set $\alpha'\colon \bN^{r}\oplus \bZ^{s}\to \cM_{\fX}$ to be the monoid map induced by $\alpha\colon \bN^{r}\to \cM_{\fX}$ and $\bZ^{s}\to \cM_{\fX}$ mapping $e'_{i}$ to $u_{i}$, and we define a monoid map $\gamma\colon \bN^{s}\to \bN^{r}\oplus \bZ^{s}$ by $\gamma(e_{i})\coloneqq m_{i}+e'_{i}$. Then the desired condition is satisfied.
\end{proof}


\begin{thebibliography}{99}
\bibitem[ALB23]{alb23} J.\ Ansch\"{u}tz, A.-C.\ Le Bras: \textit{Prismatic Dieudonné theory}, Forum Math.\ Pi 11 (2023), Paper No.\ e2, 92 pp.
\bibitem[BBM82]{bbm82} P.\ Berthelot, L.\ Breen, W.\ Messing: \textit{Th\'{e}orie de Dieudonn\'{e} cristalline. II} Lecture Notes in Math, 930
Springer-Verlag, Berlin, 1982. 
\bibitem[Bro13]{bro13} M.\ Broshi: \textit{G-torsors over a Dedekind scheme}, J.\ Pure Appl.\ Algebra 217 (2013), no.\ 1, 11-19.
\bibitem[BS22]{bs22} B.\ Bhatt, P.\ Scholze: \textit{Prisms and prismatic cohomology}, Ann.\ of Math.\ (2) 196 (2022), no.\ 3, 1135-1275.
\bibitem[DLLZ23]{dllz23b} H.\ Diao, K.-W.\ Lan, R.\ Liu, X.\ Zhu: \textit{Logarithmic adic spaces: some foundational results}, Simons Symp.\ p-adic Hodge theory, singular varieties, and non-abelian aspects (2023), 65-182.
\bibitem[DLMS24]{dlms24} H.\ Du, T.\ Liu, Y.\ S.\ Moon, K.\ Shimizu: \textit{Log prismatic $F$-crystals and purity}, 2024, arXiv:2404.19603
\bibitem[DS09]{ds09} J.\ P.\ dos Santos: \textit{The behaviour of the differential Galois group on the generic and special fibres: a
Tannakian approach}, J.\ Reine Angew.\ Math.\ 637 (2009), 63-98.
\bibitem[Fal89]{fal89} G.\ Faltings: \textit{Crystalline cohomology and $p$-adic Galois representations}, Johns Hopkins University Press, Baltimore, MD, 1989, 25-80
\bibitem[FC90]{fc90} G.\ Faltings, C.-L.\ Chai: \textit{Degeneration of abelian varieties},
Ergeb.\ Math.\ Grenzgeb.\ (3), 22, Springer-Verlag, Berlin, 1990.
\bibitem[GK19]{gk19} W.\ Goldring, J.-S.\ Koskivirta: \textit{Strata Hasse invariants, Hecke algebras and Galois representations}, Invent.\ Math.\ 217 (2019), no.\ 3, 887-984
\bibitem[GR24]{gr24} H.\ Guo, E.\ Reinecke: \textit{A prismatic approach to crystalline local systems}, Invent.\ Math.\ 236 (2024), no.\ 1, 17-164.
\bibitem[IKY23]{iky23} N.\ Imai, H.\ Kato, A.\ Youcis: \textit{The prismatic realization functor for Shimura varieties of abelian type}, 2023, arXiv:2310.08472
\bibitem[IKY24]{iky24} N.\ Imai, H.\ Kato, A.\ Youcis: \textit{A Tannakian framework for prismatic F-crystals}, 2024, arXiv:2406.08259
\bibitem[Ill85]{ill85} L.\ Illusie: \textit{Deformations de groupes de Barsotti-Tate (d’apr\`{e}s
A.\ Grothendieck)}, Ast\'{e}risque.\ 127 (1985), 151-198
\bibitem[Ino23]{ino23} K.\ Inoue: \textit{Slope filtrations of log $p$-divisible groups}, 2023, arXiv:2307.04629
\bibitem[Ino25]{ino25} K.\ Inoue: \textit{Log prismatic $F$-crystals and realization functors}, 2025, arXiv:2505.01084
\bibitem[Kat88]{kat88} K.\ Kato: \textit{Logarithmic structures of Fontaine-Illusie}, Algebraic analysis, geometry, and number theory (Baltimore, MD, 1988), 191-224, Johns Hopkins Univ.\ Press, Baltimore, MD, 1989
\bibitem[Kat21]{kat21} K.\ Kato: \textit{Logarithmic structures of Fontaine-Illusie. II.}, Tokyo J.\ Math.\ 44 (2021), no.\ 1, 125-155.
\bibitem[Kat23]{kat23} K.\ Kato: \textit{Logarithmic Dieudonn\'{e} theory}, 2023, arXiv:2306.13943
\bibitem[Kim18]{kim18} W.\ Kim, \textit{Rapoport-Zink uniformization of Hodge-type Shimura varieties}, Forum Math.\ Sigma 6
(2018), Paper No.\ e16,
\bibitem[Kis10]{kis10} M.\ Kisin: \textit{Integral models for Shimura varieties of abelian type}, J.\ Amer.\ Math.\ Soc.\ 23 (2010), no.\ 4,
967-1012.
\bibitem[KKMSD]{kkmsd} G.\ Kempf, F.\ Knudsen, D.\ Mumford, B.\ Saint-Donat, \textit{Toroidal embeddings. I}, Lecture Notes in Mathematics, vol.\ 339, Springer-Verlag, Berlin-New York, 1973
\bibitem[KKN08]{kkn08} T.\ Kajiwara, K.\ Kato, C.\ Nakayama: \textit{Logarithmic abelian varieties}, Nagoya Math.\ J.\ 189 (2008), 63-138.
\bibitem[KKN15]{kkn15} T.\ Kajiwara, K.\ Kato, C.\ Nakayama: \textit{Logarithmic abelian varieties, Part IV: Proper models}, Nagoya Math.\ J.\ 219 (2015), 9-63.
\bibitem[KKN18]{kkn18} T.\ Kajiwara, K.\ Kato, C.\ Nakayama: \textit{Logarithmic abelian varieties, Part V: Projective models}, Yokohama Math.\ J.\ 64 (2018), 21-82.
\bibitem[KKN19]{kkn19} T.\ Kajiwara, K.\ Kato, C.\ Nakayama: \textit{Logarithmic abelian varieties, Part VI: Local moduli and GAGF}, Yokohama Math.\ J.\ 65 (2019), 53-75.
\bibitem[KKN21]{kkn21} T.\ Kajiwara, K.\ Kato, C.\ Nakayama: \textit{Logarithmic abelian varieties, part VII: moduli}, Yokohama Math.\ J.\ 67 (2021), 9-48.
\bibitem[KKN22]{kkn22} T.\ Kajiwara, K.\ Kato, C.\ Nakayama: \textit{Moduli of logarithmic abelian varieties with PEL structure}, 2022, arXiv:2205.10985
\bibitem[Kos22]{kos22} T.\ Koshikawa: \textit{Logarithmic prismatic cohomology I}, 2022, arXiv:2007.14037
\bibitem[KU09]{ku09} K.\ Kato, S.\ Usui: \textit{Classifying spaces of degenerating polarized Hodge structures}, Ann.\ of Math.\ Stud.\ 169 Princeton University Press, Princeton, NJ, 2009
\bibitem[KY23]{ky23} T.\ Koshikawa, Z.\ Yao: \textit{Logarithmic prismatic cohomology II}, 2023, arXiv:2006.00364
\bibitem[Lan12]{lan12} K.-W.\ Lan: \textit{Toroidal compactifications of PEL-type Kuga families}, Algebra Number Theory 6 (2012), no.\ 5, 885-966.
\bibitem[Lan13]{lan13} K.-W.\ Lan: \textit{Arithmetic compactifications of PEL-type Shimura varieties}, London Mathematical Society monographs 36, Princeton Univ.\ Press, 2013.
\bibitem[Lau18]{lau18} E.\ Lau: \textit{Divided Dieudonn\'{e} crystals}, 2018, arXiv:1811.09439
\bibitem[Lov17]{lov17} T.\ Lovering: \textit{Filtered F-crystals on Shimura varieties of abelian type}, 2017, arXiv:1702.06611
\bibitem[Mad19]{mad19} K.\ Madapusi: \textit{Toroidal compactifications of integral models of Shimura varieties of Hodge type}, Ann.\ Sci.\ \'{E}c.\ Norm.\ Sup\'{e}r.\ 52 (2019), no.\ 2, 393-514.
\bibitem[Mes72]{mes72} W.\ Messing: \textit{The crystals associated to Barsotti-Tate groups: with applications to abelian schemes}, 
Lecture Notes in Math, Vol.\ 264
Springer-Verlag, Berlin-New York, 1972.
\bibitem[Niz08]{niz08} W.\ Nizio{\l}: \textit{K-theory of log-schemes I}, Doc.\ Math.\ 13 (2008), 505-551.
\bibitem[Ogu18]{ogu18} A.\ Ogus: \textit{Lectures on logarithmic algebraic geometry}, Cambridge Studies in Advanced Mathematics, 178.\ Cambridge University Press, Cambridge, 2018
\bibitem[Pin90]{pin90} R.\ Pink: \textit{Arithmetic compactification of mixed Shimura varieties}, Bonner Mathematische Schriften, Universit\"{a}t Bonn Mathematisches Institut, 1990.\ Dissertion, Rheinische Friedrich-Wilhelms-Universit\"{a}t Bonn, 1989.
\bibitem[SP24]{sp24} The Stacks Project Authors, \textit{Stacks Project}, 2024
\bibitem[WZ23]{wz23} M.\ W$\mathrm{\ddot{u}}$rthen, H.\ Zhao: \textit{Log prismatic Dieudonné theory for log p-divisible groups over $\cO_{K}$}, 2023, arXiv:2310.15732

\end{thebibliography}
\end{document}